\documentclass[12pt,table]{article}
\usepackage{color}
\usepackage[table,xcdraw,dvipsnames]{xcolor}
\usepackage{soul}
\usepackage[hidelinks,colorlinks=true,bookmarks=true,citecolor=Green ,linkcolor=red,urlcolor=blue]{hyperref}
\usepackage{amsmath,scalerel,amssymb}
\usepackage{amsthm}
\usepackage{mathrsfs}
\usepackage{cite}
\usepackage{MnSymbol}
\usepackage[mathletters]{ucs}
\usepackage[utf8x]{inputenc}
\usepackage[]{graphicx} 
\usepackage[percent]{overpic}
\usepackage{wrapfig}
\usepackage{float} 
\usepackage{rotating}
\usepackage{tikz}

\usepackage{accents}

\usepackage{array}
\newcolumntype{M}[1]{>{\centering\arraybackslash}m{#1}}
\newcolumntype{N}{@{}m{0pt}@{}}
\usepackage{multicol}
\usepackage{multirow}
\usepackage{caption}
\usepackage{subcaption}
\usepackage[left=2cm,right=2cm,top=2cm,bottom=2cm]{geometry}
\usepackage{setspace}
\usepackage[most]{tcolorbox}
\usepackage{pifont}
\usepackage{wasysym}

\numberwithin{lemma}{section}
\newtheorem{proposition}{Proposition}
\numberwithin{proposition}{section}

\numberwithin{definition}{section}

\numberwithin{example}{section}

\numberwithin{theorem}{section}

\numberwithin{corollary}{section}

\newtheorem{remark}{Remark}
\numberwithin{remark}{section}

\newcommand{\beqn}{\begin{equation} \begin{aligned}}
\newcommand{\eeqn}{\end{aligned}\end{equation}}
\newcommand{\beqnn}{\begin{equation*} \begin{aligned}}
\newcommand{\eeqnn}{\end{aligned}\end{equation*}}
\soulregister\cite7
\soulregister\ref7
\soulregister\pageref7
\colorlet{Mycolor1}{green!10!orange!90!}
\colorlet{Mycolor2}{Yellow!80!}
\colorlet{Mycolor3}{green!50!}
\colorlet{Mycolor4}{Red!90!}

\def\sgn{{\rm sgn}}
\newcommand{\review}[1]{\textcolor{black}{#1}}
\makeatletter
\makeatletter
\renewcommand*{\@fnsymbol}[1]{\ifcase#1\or *\or $\star$\or$\dagger$\or$\ddagger$\or \else\@arabic{#1}\fi}
\makeatother

\date{} 
\title{A Phase Model with Large Time Delayed Coupling\footnote{This research is supported in part by the Natural Sciences and Engineering Research Council of Canada.}}

\makeatletter
\newcommand{\specificthanks}[1]{\@fnsymbol{#1}}
\makeatother

\author{Isam Al-Darabsah\footnote{Department of Applied Mathematics, University of Waterloo, Waterloo, ON,  N2L 3G1, Canada.} \textsuperscript{,}\footnote{Email: ialdarabsah@uwaterloo.ca} \and Sue Ann Campbell\textsuperscript{\specificthanks{2}}\textsuperscript{,}\footnote{Email: sacampbell@uwaterloo.ca} }

\begin{document}

\maketitle
	
\begin{abstract}
We consider two identical oscillators with weak, time delayed coupling. \review{We start with a general system of delay differential equations then  reduce it to a phase model.} 
\review{With the assumption of large time delay, the resulting phase model has an explicit delay and phase shift in the argument of the phases and connection function, respectively.} Using the phase model, we prove that for any type of oscillators and 
any coupling, the in-phase and anti-phase phase-locked solutions always 
exist and give conditions for their stability. We show that for small delay 
these solutions are unique, but with large enough delay multiple solutions of
each type with different frequencies may occur. We give conditions for
the existence and stability of other types of phase-locked solutions. 
We discuss the various bifurcations that can occur in the phase model as the
time delay is varied.  The results of the phase
model analysis are applied to Morris-Lecar oscillators with diffusive coupling and compared with numerical studies of the full system of delay differential equations. We also consider the case of small time delay and compare the results with the existing  ones in the literature.

\end{abstract}
	

\section{Introduction}

Coupled oscillator models have been used to study different aspects of biology, chemistry  and engineering, for example chemical waves \cite{kuramoto2003chemical}, flashing of fireflies \cite{mirollo1990synchronization}, laser arrays  \cite{winful1988stability, wang1988dynamics},  power system networks \cite{dorfler2013synchronization}, neural networks \cite{kopell1988coupled,hansel1993phase,crook1997role,park2016weakly}, movement of a slime mold \cite{TFE}, and coupled
predator-prey systems \cite{wall2013synchronization,zhang2015robust}. 
Time delays in the connections between the oscillators are inescapable due to the time  
for a signal to propagate from one element to the other.
Many of these systems exhibit phase-locking behaviour,  i.e., all the oscillators have 
similar waveforms and frequencies, but with some fixed phase difference between different
oscillators.  To study the existence and stability of such phase-locked solutions and how 
they are related to the time delay and other parameters, one must formulate a model
for the system. We discuss two approaches below. 

One approach to study connected networks of oscillators is through phase models
\cite{DorflerB14}.  In these models, each oscillator is represented only by 
its phase along its limit cycle, 
with amplitude variation neglected \cite{hoppensteadt2012weakly,Porter2014Dynamical,schwemmer2012theory}.  
Phase models take the general form \cite{campbell2018phase,hoppensteadt2012weakly}:
 \begin{equation}
 \label{General_Phase_Model}
\frac{d \theta_{i}}{d \xi}=\Omega_{i}+ H_{i}\left(\theta_{1}(\xi), \ldots, \theta_{n}(\xi)\right),  \quad i=1, \ldots, n,
\end{equation}
where $\theta_i\in[0,2\pi)$ is the phase of the $i^{\rm th}$ oscillator, $\Omega_i>0$
the natural frequency and $H_i$ are the connection functions. 
 Motivated by the famous \textit{Kuramoto model} \cite{kuramoto2003chemical}, 
in the literature the functions $H_i$ often take the form: 
\begin{equation}
 \label{fun_H_KM}
 	H_{i}\left(\theta_{1}(\xi), \ldots, \theta_{n}(\xi)\right)= \sum_{j=1}^{n} K_{ij}H \left(\theta_{j}(\xi)-\theta_{i}(\xi)\right), \quad i=1, \ldots, n,
 \end{equation}
where  $K_{ij}$ is the adjacency matrix of
an unweighted network \cite{Porter2014Dynamical,earl2003synchronization}.  
In the original Kuramoto model \cite{kuramoto2003chemical}   the function  $H$ in (\ref{fun_H_KM}) is the sine function.  Usually, transmission time delay is introduced as an explicit delay in the argument of the phases
\cite{earl2003synchronization,kim1997multistability,Luz,NSK,yeung1999time,Schuster1989Mutual}:
 \begin{equation} 
 \label{Kuramoto_model_Exp_Delay}
\frac{d \theta_{i}}{d \xi}=\Omega_{i}+ \sum_{j=1}^{n} K_{ij}H\left(\theta_{j}(\xi-\tau)-\theta_{i}(\xi)\right), \quad i=1, \ldots, n.
\end{equation}
Most studies of this model focus only on synchronization 
\cite{earl2003synchronization} or use simplifications such as 
$H(\cdot)=\sin(\cdot)$ \cite{ermentrout2009delays,kim1997multistability,Luz,NSK,Schuster1989Mutual,yeung1999time} or  $n=2$  
\cite{kim1997multistability,Schuster1989Mutual,yeung1999time}.

Other authors introduce additional processes into  system (\ref{Kuramoto_model_Exp_Delay}). For instance, in \cite{kim1997multistability}, 
the  dynamic  behavior  of coupled oscillators with time delayed interaction under a pinning  force is studied. In \cite{NSK,yeung1999time}, the authors study time delayed phase models with  $H=\sin(\cdot)$ and random noise forcing. 
Finally, a \textit{phase shift} is sometimes included in the model of a network of 
connected oscillators to represent the temporal distance between the oscillators.
In general, the phase shift between two oscillators  $\alpha_{ij}$ is  incorporated 
in the phase model as, see e.g., \cite{brede2016frustration,ermentrout2009delays,sakaguchi1986soluble},  
 \begin{equation} 
 \label{Kuramoto_model_phasShift}
\frac{d \theta_{i}}{d \xi}=\Omega_{i}+ \sum_{j=1}^{n} K_{ij}H\left(\theta_{j}(\xi)-\theta_{i}(\xi)-\alpha_{ij}\right), \quad i=1, \ldots, n. 
\end{equation}
In the case where $H(\cdot)=\sin(\cdot)$ this model is called the Kuramoto-Sakaguchi 
model\cite{sakaguchi1986soluble}. 
In fact, there is a relation between such phase shifts and the transmission time delay. 
In \cite{izhikevich1998phase,ermentrout1994introduction}, the authors  have shown how the model with delay and the model with the phase shift are linked.  We will review the 
details of this link later in this section.

Models of coupled oscillators are also formulated as physically or biological derived 
differential equations \cite{wall2013synchronization,zhang2015robust,campbell2018phase}.  
These models are of the form 
\begin{equation}
\label{phase_eq1}
\review{\frac{d \mathbf{X}_{i}}{d \rho}={\mathbf{F}}_{i}\left(\mathbf{X}_{i}(\rho)\right)+\epsilon {\mathbf{G}}_{i}\big(\mathbf{X}_{1}(\rho), \ldots,\mathbf{X}_{i}(\rho),\ldots, \mathbf{X}_{n}(\rho)\big), \quad i=1, \ldots, n,\quad \mathbf{X}_i\in\mathbb{R}^m,}
\end{equation}
and are such that when $\epsilon=0$ the dynamical system of each uncoupled oscillator 
has an exponentially asymptotically stable $T_i-$periodic limit cycle with corresponding 
(natural) frequency $\Omega_i$.  In these models, \review{$\mathbf{X}_{i}$ represents the state of the 
$i^{th}$ oscillator of the system}, ${\mathbf{G}}_i$ are the coupling functions
and $\epsilon>0$ is the coupling strength  \cite{ET10,hoppensteadt2012weakly,izhikevich1998phase,KE02}.  Note that $\mathbf{X}_{i}$ is a vector of dimension at least $2$, but 
can be high dimensional. 
For example, in a pendulum model $\mathbf{X}_{i}$ represents the position and velocity of the 
$i^{th}$ pendulum, while in a neural model $\mathbf{X}_{i}$ represents the voltage and gating 
variables of the  $i^{th}$ neuron. 

If the coupling is weak, $0<\epsilon\ll 1$, then the theory of weakly coupled 
oscillators can be used to connect the physical model \eqref{phase_eq1} to a 
phase model \cite{crook1997role,ET10,KE02,galan2009phase,ZS09}
.
More precisely,  
the dynamics of each oscillator in the network can be rigorously reduced to a single 
equation that indicates how the phase of the oscillator changes in time 
\cite{hoppensteadt2012weakly,izhikevich1998phase,schwemmer2012theory}. 
One form of weakly coupled oscillator theory is \textit{Malkin's Theorem} where the 
connection functions in the phase model are determined explicitly in terms of ${\mathbf{G}}_i$ and 
the limit cycles of the uncoupled system, (\ref{phase_eq1}) with $\epsilon=0$.
Let $\varphi_i(t)\in\mathbb{S}^1$ be the {\em phase deviation} of the $i^{\rm th}$ oscillator
of (\ref{phase_eq1}), i.e., the change in the phase due to the coupling.
It then follows from Malkin's Theorem (see  e.g., \cite[Theorem 9.2]{hoppensteadt2012weakly}) that 
the dynamics of  (\ref{phase_eq1}) can be described by the phase deviation model:
\begin{equation}
\begin{aligned}
\label{intor_21}
\frac{d\varphi_i}{d{t}}&=H_i\big(\varphi_{1}\left(t\right)-\varphi_{i}\left(t\right), \ldots, \varphi_{n}\left(t\right)-\varphi_{i}\left(t\right)\big)+\mathcal{O}(\epsilon), \quad i=1, \ldots, n,
\end{aligned}
\end{equation}
where $H_i$ are the phase interaction functions and the variable $t:=\epsilon \rho$ 
represents slow time because the phase deviations $\varphi_i$ are slow 
variables.  The references 
\cite{KE02,hoppensteadt2012weakly,schwemmer2012theory} provide other forms  
of the theory and give further references.
\review{We also refer the reader to the recent articles  \cite{R1pietras2019network,R2nakao2016phase,R3ashwin2016mathematical} for  an overview of various numerical and analytical techniques for phase reduction.}
In \cite{izhikevich1998phase}, Izhikevich  generalizes Malkin's theorem to weakly connected oscillators with fixed delay, $\tau$, in their interaction:
\beqn
\label{phase_eq2}
\review{\frac{d \mathbf{X}_{i}}{d \rho}={\mathbf{F}}_{i}\left(\mathbf{X}_{i}(\rho)\right)+\epsilon {\mathbf{G}}_{i}\big(\mathbf{X}_{1}(\rho-\tau), \ldots, \mathbf{X}_{i}(\rho-\tau), \ldots \mathbf{X}_{n}(\rho-\tau)\big), \quad i=1, \ldots, n,\quad \mathbf{X}_{i}\in\mathbb{R}^m}
\eeqn
where all uncoupled oscillators have nearly identical natural frequencies.  
Assuming the natural frequency is $1$,  Izhikevich  shows that the phase deviation model corresponding to (\ref{phase_eq2}) is 
\begin{equation}
\begin{aligned}
\label{intor_2}
\frac{d\varphi_i}{d{t}}&=H_i\big(\varphi_{1}\left(t-\eta\right)-\varphi_{i}\left(t\right)-\zeta, \ldots, \varphi_{n}\left(t-\eta\right)-\varphi_{i}\left(t\right)-\zeta\big)+\mathcal{O}(\epsilon), \quad i=1, \ldots, n,
\end{aligned}
\end{equation}
where $\eta:=\epsilon\tau$ and $\zeta:=\tau\mod2\pi$.
The functions $H_i$ are still defined explicitly in terms of ${\mathbf{G}}_i$ and the uncoupled limit cycle in (\ref{phase_eq2}).  
 It is clear that the time delay $\tau$ enters the phase model (\ref{intor_2}) as both 
an explicit delay, $\eta$, and a phase shift, $\zeta$.
The major result that Izhikevich proved in \cite{izhikevich1998phase} is that if the delay $\tau$ in (\ref{phase_eq2}) satisfies $\epsilon\tau=\mathcal{O}(1)$  (large delay), then 
the explicit delay occurs in the phase model (\ref{intor_2}). 
However, when the delay satisfies $\tau=\mathcal{O}(1)$ with respect to $\epsilon$  (small delay), no delay appears in the argument of the phases. Hence, (\ref{intor_2}) 
becomes: 
\begin{equation}
\begin{aligned}
\label{intor_7}
\frac{d\varphi_i}{d{t}}&=H_i\big(\varphi_{1}\left(t\right)-\varphi_{i}\left(t\right)-\zeta, \ldots, \varphi_{n}\left(t\right)-\varphi_{i}\left(t\right)-\zeta\big)+\mathcal{O}(\epsilon), \quad i=1, \ldots, n.\end{aligned}
\end{equation}
 We refer the reader to the review article 
\cite{ermentrout2009delays} and the references therein for different scenarios where large or small delay appears in-phase models.

\review{In this article we focus on physical models with the following particular form
\begin{equation}
\frac{{d{{\mathbf{X}}_i}}}{{d\rho}} = {\mathbf{F}}({{\mathbf{X}}_i}(\rho)) + \epsilon\sum\limits_{j = 1}^n {{K_{ij}}} {\mathbf{G}}({{\mathbf{X}}_i}(\rho),{{\mathbf{X}}_j}(\rho - \tau )), \quad i = 1, \ldots n,\ {{\mathbf{X}}_i} \in {\mathbb{R}^m} 
\label{newmodel}
\end{equation}
where $K_{ii}=0$. 
This represent the following modelling assumptions. The oscillators are
identical.  The coupling occurs pairwise between the oscillators and there
is no coupling from an oscillator to itself. The coupling to the $i^{th}$
oscillator occurs close to that oscillator, so the time delay represents the
time it takes for information to travel from the $j^{th}$  oscillator to the
$i^{th}$ oscillator. Models with such structure occur in models of
biological systems \cite{crook1997role,wall2013synchronization}.}

\review{Assuming the uncoupled
oscillators in \eqref{newmodel} have a natural frequency $\Omega$ and
the $K_{ij}=\mathcal{O}(1)$ with respect to $\epsilon$, we show in the
appendix that the approach of \cite{izhikevich1998phase} can be applied to
yield 
\begin{equation}
\frac{d\varphi_i}{dt}=\frac{1}{\Omega}\sum_{j=1}^{n} K_{ij} H(\varphi_j(t-\eta)-\varphi_i(t)-\zeta) 
+\mathcal{O}(\epsilon) 
\label{pairwise}
\end{equation}
where $\eta:=\epsilon\Omega\tau$ and $\zeta:=\Omega\tau\mod2\pi$, in the case of large delay, i.e., when $\epsilon\Omega\tau=\mathcal{O}(1)$.
In the case of small delay \eqref{pairwise}
becomes
\begin{equation}\label{smaleDealy_intro} 
    \frac{{d{\varphi _i}(t)}}{{dt}} = \frac{1}{{\Omega}}\sum\limits_{j = 1}^n {{K_{ij}}}H\left( {{\varphi _j}(t) - {\varphi _i}(t) - \Omega \tau } \right)+ \mathcal{O}(\epsilon ).
\end{equation}
}

To see how the phase deviation model relates to the standard phase model, note 
that the phase of oscillations $\theta_i$ in (\ref{newmodel}) have the form: 
\begin{equation}
\label{eq1234}
\theta_i(\xi)=\Omega\xi+\varphi_i(t), \quad i=1, \ldots, n, 
\end{equation}
where $t=\epsilon \Omega\xi$, see \cite{izhikevich1998phase,hoppensteadt2012weakly}. Notice  
that the natural frequency of each uncoupled oscillator in (\ref{eq1234}) is $\Omega$. 
Then,
\beqn
\label{New_equation_1}
\frac{d \theta_{i}}{d \xi}=\Omega+ \epsilon \Omega\frac{d \varphi_i} {d t}=\Omega +\epsilon \sum_{j=1}^{n} K_{ij}H\left(\theta_{j}(\xi-\tau)-\theta_{i}(\xi)\right)+\mathcal{O}(\epsilon^2).
\eeqn
Similarly, when the time delay is small, we have
\beqn
\label{New_equation_2}
\frac{d \theta_{i} }{d \xi}=\Omega+ \epsilon\sum_{j=1}^{n} K_{ij}H
\left(\theta_{j}(\xi)-\theta_{i}(\xi)-\zeta\right)+\mathcal{O}(\epsilon^2).
\eeqn
Thus in the phase model formulation, the coupling strength parameter $\epsilon$ 
explicitly appears in front of the connection function $h$. 
Regarding the dynamics,  
it follows from (\ref{eq1234}) that
\[
\theta_{i+1}-\theta_{i}=\varphi_{i+1}-\varphi_{i},\quad i=1,\ldots,{n-1}
\]
i.e., phase-locked solutions   are the same as phase deviation locked solutions \cite{hoppensteadt2012weakly}.  The existence and stability of phase-locked solutions of
system (\ref{New_equation_2}) has been studied in the case of two oscillators
\cite{campbell2012phase,ermentrout2009delays} and many oscillators with 
structured coupling \cite{campbell2018phase,ermentrout2009delays,ko2004wave}.  

The goals in this paper are twofold. First, the majority of studies of coupled
oscillators with large delays have been done in the context of isolated phase models, often with just sine function coupling. Thus we will revisit and extend 
this analysis in the case where the phase model is explicitly connected to a 
physical differential equation model and the function $H$ is general. In 
particular, we will show that the
multiple stable phase-locked solutions of the same type may occur even when the
coupling is weak.
Second, note that the small delay phase deviation model 
\eqref{smaleDealy_intro} is a system of ordinary differential equations, while the large 
delay model \eqref{newmodel} is a delay differential equation model. Thus the 
spectrum of Floquet multipliers of a periodic solution is finite for the former and countably
infinite for the latter.  Nevertheless, several studies have verified numerically
that the model \eqref{smaleDealy_intro} gives an accurate description of existence and stability
of phase-locked periodic solutions of \eqref{phase_eq1} in the case of weak coupling 
and small delay 
\cite{campbell2012phase,campbell2018phase}. Here we will show why this is the case.
In particular we will show how the solutions of system (\ref{newmodel}) reduce to those of system (\ref{smaleDealy_intro}) if the delay is small.
In this article, \review{we will focus on  (\ref{newmodel})}  when $n=2$ as this is enough to illustrate our main points.

The paper is organized as follows. In the next section, we reduce the model of two weakly connected oscillators with large time delay to a phase model, and study the existence of 
phase-locked solutions.  In Section \ref{sec_stability}, we give a complete
description of the stability criteria for all phase-locked solutions and 
describe the potential bifurcations 
that can occur in the system. Then we compare our results with the 
stability criteria in \cite{campbell2012phase} when the time delay is small.   
In Section \ref{Sec3}, 
we consider a particular application to Morris-Lecar oscillators with diffusive coupling. Numerically, we derive the corresponding 
phase model, calculate the phase-locked solutions, determine their stability and explore the existence of bifurcations.
We also compare prediction of the phase model and solutions of the full model. Finally, we examine the behaviour when the time delay is small. 
In Section \ref{sec_conc}, we discuss our results.

\section{Phase Model}

Consider the system of ODEs 
	\begin{equation}
	\label{ODE}
		\frac{d\mathbf{X}_i}{d\rho}={\mathbf{F}}({\mathbf{X}_i}(\rho))\quad  i=1,2,\quad \mathbf{X}_i\in\mathbb{R}^n.
	\end{equation}
Assume that the system (\ref{ODE}) admits an exponentially asymptotically stable periodic
orbit given by $\mathbf{X}=\tilde{\mathbf{X}}(\rho)$ with natural frequency $\Omega$, $0\le \rho\le T=2\pi/\Omega$. 	
	
Next, consider a weakly connected system
		of  two identical coupled oscillators of the form (\ref{ODE}) with time delayed coupling:	
\beqn
  			\label{Full_Mod}
\frac{d\mathbf{X}_1}{d{\rho}}&=\mathbf{F}\left(\mathbf{X}_1(\rho) \right)+\epsilon \mathbf{G}\left(\mathbf{X}_1(\rho),\mathbf{X}_2(\rho-\tau);\epsilon\right),\\
			\frac{d\mathbf{X}_2}{d\rho}&=\mathbf{F}\left(\mathbf{X}_2(\rho) \right)+\epsilon \mathbf{G}\left(\mathbf{X}_2(\rho),\mathbf{X}_1(\rho-\tau);\epsilon\right),
\eeqn
where $\mathbf{G}:\mathbb{R}^n\times\mathbb{R}^n\to \mathbb{R}^n $ describes the coupling between the two oscillators and $\epsilon$ is the coupling strength. 
Assume that $\epsilon$ is sufficiently small and $\eta:=\epsilon\Omega\tau=\mathcal{O}(1)$. Let
${t}=\epsilon \rho$ be slow time and $\varphi_i(t)\in \mathbb{S}^1$ be the phase deviation from the
natural oscillation $\hat{X}(\rho)$, $\rho\ge 0$. Then, by applying weakly coupled oscillator theory  with delayed interactions in \cite{izhikevich1998phase},  $(\varphi_1,\varphi_2)^T\in \mathbb{T}^2$ is a solution to 
\beqn
\label{phase_ModEpsilon}
\frac{d\varphi_1}{d{t}}&=\frac{1}{\Omega} H(\varphi_2(t-\eta)-\varphi_1(t)-\Omega\tau)+\mathcal{O}(\epsilon),\\
\frac{d\varphi_2}{d{t}}&=\frac{1}{\Omega} H(\varphi_1(t-\eta)-\varphi_2(t)-\Omega\tau)+\mathcal{O}(\epsilon),
\eeqn
where $H$ is a $2\pi-$periodic function defined by 
\begin{equation}\label{New:funH}
    H(\phi ) = \frac{1}{{2\pi }}\int\limits_0^{2\pi } {\hat{\mathbf{Z}}{{(\rho)}^T}\mathbf{G}\left( {\hat {\mathbf X}(\rho),\hat {\mathbf X}(\rho + \phi )} \right)} d\rho.
\end{equation}
Here $\hat{\mathbf{Z}}{(\rho)}$ is the unique nontrivial $2\pi-$periodic solution to the adjoint linear system
\[\frac{{d\hat{\mathbf{Z}}}}{{d\rho}} =  - {\left[ {D\mathbf{F}\left( {\hat {\mathbf{X}}(\rho)} \right)} \right]^T}\hat{\mathbf{Z}}\]
satisfying the normalization condition
\[\frac{1}{{2\pi }}\int\limits_0^{2\pi } {\hat{\mathbf{Z}}(\rho) \cdot } \mathbf{F}\left( {\hat {\mathbf{X}}(\rho)} \right)d\rho = 1.\]
\review{The derivation of system (\ref{phase_ModEpsilon}) from (\ref{Full_Mod}) follows from  the Appendix with $n=2$ and $K_{12}=K_{21}=1$}.

Dropping the terms $\mathcal{O}(\epsilon)$ in (\ref{phase_ModEpsilon}), we obtain the phase deviation model:
\beqn
\label{phase_Mod}
\frac{d\varphi_1}{d{t}}&=\frac{1}{\Omega} H(\varphi_2(t-\eta)-\varphi_1(t)-\Omega\tau),\\
\frac{d\varphi_2}{d{t}}&=\frac{1}{\Omega} H(\varphi_1(t-\eta)-\varphi_2(t)-\Omega\tau).
\eeqn
For simplicity, in the rest of the paper we will refer to (\ref{phase_Mod}) as the \textit{phase model} instead of the phase deviation model.

We study the dynamics of the model (\ref{phase_Mod}) by exploring \textit{phase
locking} in (\ref{phase_Mod}), that is, solutions of (\ref{phase_Mod}) such that $\varphi_2-\varphi_1=\text{constant}$ \cite{hoppensteadt2012weakly}. 
We suppose that 
\begin{equation}
	\varphi_1(t)=\omega t\qquad\text{and}\qquad \varphi_2(t)=\omega t+\psi
	\label{phases}
\end{equation}
where $\omega$ is the frequency deviation of the oscillator and $\psi$ is the 
natural phase difference \cite{hoppensteadt2012weakly}. Substituting (\ref{phases}) into (\ref{phase_Mod}) leads to 
\beqn
\label{sol_sys}
\omega-\frac{1}{\Omega} H(\psi-\omega\eta-\Omega\tau)&=0,\\
\omega-\frac{1}{\Omega} H(-\psi-\omega\eta-\Omega\tau)&=0.
\eeqn
We rewrite this as 
\beqn
F(\omega,\psi)=0=F(\omega,-\psi)
\label{Fsys}
\eeqn
where 
\beqn
F(\omega,\cdot):=\omega-\frac{1}{\Omega} H(\cdot-\omega\eta-\Omega\tau).
\label{Fdef}
\eeqn

In this article, we are interested in exploring how the solutions \review{($\psi$ and $\omega$) of (\ref{sol_sys}) vary with $\tau$ when the coupling strength ($\epsilon$) and  frequency ($\Omega$) are fixed. Note that, we need only to investigate $\psi$  
 in $[0,2\pi)$, due to the $2\pi$ periodicity of $H$, and $\omega\in\mathbb{R}$.}

First, by subtracting the equations of (\ref{sol_sys}), we obtain
\begin{equation}
\label{eq_H}
	H(\psi-\omega\eta-\Omega\tau)-H(-\psi-\omega\eta-\Omega\tau)=0.
\end{equation}
Since $H$ is $2\pi-$periodic function, equation (\ref{eq_H})  always has 
the solutions $\psi=0,\pi$.  
The corresponding frequency deviation is determined from the equation
\begin{equation}
\label{omega0}
	F(\omega,0)=\omega-\frac{1}{\Omega} H(-\omega\eta-\Omega\tau)=0
\end{equation}
when $\psi=0$ and 
\begin{equation}
\label{omegapi}
	F(\omega,\pi)=\omega-\frac{1}{\Omega} H(\pi-\omega\eta-\Omega\tau)=0
\end{equation}
when $\psi=\pi$.

Equations (\ref{omega0}) and (\ref{omegapi}) are guaranteed to have a least one 
solution due to the continuity and $2\pi$ periodicity of  $H$. 
In fact, if $\tau$ is sufficiently large, they may have multiple solutions. To see 
this, recall that $\eta=\epsilon\omega \tau$ and note that 
\beqn
F_\omega(\omega,0)=1+\epsilon \tau  H'(-\omega \epsilon\Omega \tau-\Omega\tau),
\label{Fp0def}
\eeqn
\review{where $F_\omega$ is the partial derivative of $F$ with respect to $\omega$.
If 
there exists $\overline{\omega}$ such that $F(\overline{\omega},0)=0$ and 
$F_\omega(\overline{\omega},0)<0$ then \eqref{omega0} has more than one solution.
Similar arguments apply to equation (\ref{omegapi}).
This may be possible if $\tau$ is sufficiently large.}
 
\begin{remark}
The solutions $\psi^*=0$ and $\psi^*=\pi$ of (\ref{eq_H}) correspond to \textbf{in-phase} and \textbf{anti-phase} periodic solutions of the original model  (\ref{Full_Mod}), respectively.  By {in-phase} solution we mean both oscillators reach their highest peak at 
the same time,  whereas an anti-phase solution means one oscillator reaches its 
highest peak one half-period after the other oscillator. 
Examples of these solutions are given in Figure \ref{Fig_case}.
\end{remark}

\begin{figure}[hbt!]
\centering
  \hspace{1cm}\includegraphics[width=0.9\textwidth]{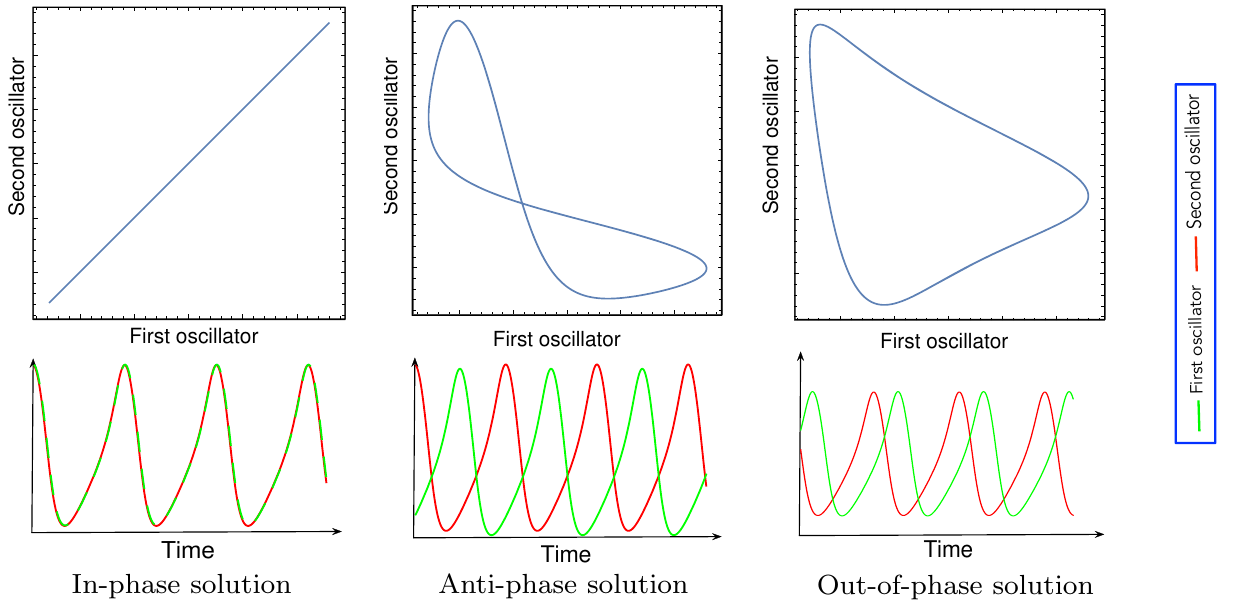}
  \caption{Illustrations of the phase-locked dynamics of model (\ref{Full_Mod}).} 
  \label{Fig_case}
\end{figure}

In fact, system (\ref{Full_Mod}) could have other phase-locked solutions (neither in-phase nor anti-phase) corresponding to the solutions $\psi$  of (\ref{sol_sys}) such that $\psi\notin \{0,\pi\}$. 
As in \cite{Scholarpedia1}, we will refer to these solutions of (\ref{Full_Mod})  as \textbf{out-of-phase} solutions. Let  
$(\omega^*,\psi^*)$  be a solution of (\ref{sol_sys}) at $\tau=\tau^*$ such that $\psi^*\notin \{0,\pi\}$.  Then $\omega^*$  and $\psi^*$ satisfy \eqref{Fsys}, that is,
$(\omega^*,\psi^*)$  is an intersection point of the contours $F(\omega,\psi)=0$ and $F(\omega,-\psi)=0$ in the $\omega\psi-$plane. 
 Suppose that $\psi^*\in (0,\pi)$ with a corresponding $\omega^*$ are solutions to (\ref{sol_sys}) at $\tau=\tau^*$, then
  \beqnn
\frac{1}{\Omega} H(\psi^*-2\pi-\omega^*\eta^*-\Omega\tau^*)&=\frac{1}{\Omega} H(\psi^*-\omega^*\eta^*-\Omega\tau^*)=\omega^*
\eeqnn
and
 \beqnn
\frac{1}{\Omega} H(2\pi-\psi^*-\omega^*\eta^*-\Omega\tau^*)&=\frac{1}{\Omega} H(-\psi^*-\omega^*\eta^*-\Omega\tau^*)=\omega^*
\eeqnn
due to the periodicity of $H$. Thus, $2\pi-\psi^*$ is also a solution in (\ref{sol_sys}) with corresponding $\omega^*$. 

This leads to the following.
\begin{proposition}[\textup{Existence of phase-locked solutions}]
For any interaction function $H$ and any values of $\Omega$, $\epsilon$ and $\tau$ 
the phase model (\ref{phase_Mod}) has the solutions $\psi^*=0$ and  $\psi^*=\pi$ with corresponding frequency deviations determined by (\ref{omega0}) and (\ref{omegapi}), respectively. 
If $\psi^*\in(0,\pi)$ with corresponding $\omega^*$ are solutions to (\ref{sol_sys}) at $\tau=\tau^*$ then so is $2\pi-\psi^*$ with $\omega^*$, i.e., solutions come in pairs.
\label{prop_1}
\end{proposition}

\section{Stability }
	\label{sec_stability}
	
In this section, we discuss the linear stability of the solutions (\ref{phases}) of  (\ref{phase_Mod}). The linearization of (\ref{phase_Mod}) about the solution  (\ref{phases}) is
\beqn
\label{LinSys}
\frac{{d{u_1}}}{{dt}} &=  - a{u_1}(t) + a{u_2}(t - \eta ),\\
\frac{{d{u_2}}}{{dt}} &=  - b{u_2}(t) + b{u_1}(t - \eta ),
\eeqn
where 
\begin{equation}
\label{ab}
a=\frac{1}{\Omega} H'(\psi-\omega\eta-\Omega\tau)\qquad\text{and}\qquad b=\frac{1}{\Omega} H'(-\psi-\omega\eta-\Omega\tau).
\end{equation}
In (\ref{ab}), $H'$ represents the derivative of $H$ with respect to its argument. It is 
useful for our analysis to scale time so the delay becomes one. Applying the scaling
\[ \eta s=t,U_1(s)=u_1(t),\ U_2(s)=u_2(t),\ \]
results in 
\beqn
	\label{LinSys_Scaled}
\frac{{d{U_1}}}{{ds}} &=  - \eta a {U_1}(s) + \eta a{U_2}(s - 1 ),\\
\frac{{d{U_2}}}{{ds}} &=  - \eta b{U_2}(s) +  \eta b{U_1}(s -1 ).
\eeqn
It follows that the corresponding characteristic equation is 
\begin{equation}\label{chactEq}
	\Delta (\lambda ;\eta ) = {\lambda ^2} + \eta(a + b)\lambda  + \eta^2 ab - \eta^2 ab{e^{ - 2\lambda  }}=0.
\end{equation}

In the following we study the distribution of roots of this equation.
\begin{proposition}\label{prop000A}
	Assume $ab=0$. Then  $\Delta(\lambda;\eta)$ has:  
	\begin{enumerate}
		\item [i.] One positive root and one zero root when $a+b<0$;
		\item [ii.] Two zero roots when $a+b=0$;
		\item [iii.] One negative root and one zero root when $a+b>0$.
		\end{enumerate}
\end{proposition}	
\begin{proof} The characteristic equation in this case reduces to
\[ {\lambda ^2} + \eta(a + b)\lambda  =0. \]
The result follows.
\end{proof}

\begin{proposition}
	$\Delta(\lambda;\eta)$ has a positive real root when one of the following holds.
	\begin{enumerate}
		\item [i.] $ab>0$ and $a+b<0$;
		\item [ii.] $ab< 0$ and $a+b\le0$;
		\item [iii.]  $ab<0$, $a+b>0$ and $a+b+2\eta a b<0$.
		\end{enumerate}
\end{proposition}	
	
\begin{proof} Define
\begin{equation}
	\label{f&g}
	f(\lambda)=(\lambda+\eta a)(\lambda+\eta b) \quad \text{and}
	\quad g(\lambda)=\eta^2 a b e^{-2\lambda}.
\end{equation} 
Then $f(0)=g(0)=\eta^2 a b$ and 
\begin{equation}
\label{eq1543}
	\Delta(\lambda;\eta)=0 \quad \iff \quad f(\lambda)= g(\lambda).
\end{equation}
	\begin{enumerate}
	\item [i.] It follows from $ab>0$ and $a+b<0$ that $a<0$ and $b<0$. Since (\ref{chactEq}) is symmetric in $a$ and $b$, without loss of generality, we may assume $b<a<0$. 
Note that $f(-\eta b)=0<g(-\eta b)$. Since $f$ is positive and increasing for $\lambda>-\eta b>0$ and $g$ is positive and decreasing for $\lambda>0$, there exists $\lambda^*>-\eta b$ such that $f(\lambda^*)=g(\lambda^*)$, see Figure \ref{Fig0_A}.

	\item[ii.] Assume $a> 0$ and $b<0$. When $a+b<0$, $f$ is decreasing for $\lambda\in \left( 0,-\frac{a+b}{2}\eta\right)$ and is increasing for $\lambda> -\frac{a+b}{2}\eta$. Further, $g$ increases for $\lambda> 0$ and $\lim_{\lambda\to \infty}g(\lambda)=0$, thus there exists $\lambda^*\in \left( -\frac{a+b}{2}\eta,-\eta b\right)$ such that $f(\lambda^*)=g(\lambda^*)$, see Figure \ref{Fig0_B}.
	When $a+b=0$, $f(0)=g(0)=\eta^2 a b$, $f'(0)=0<g'(0)$ and $f$ is increasing for $\lambda> 0$. Thus, with the same arguments, $\lambda^*$ lies in  $\left(0,-\eta b\right)$.
	
		\item[iii.] Assume $a>0$ and $b<0$. In this case $f$ and $g$ are increasing for $\lambda>0$ and $g<0$ for $\lambda\ge 0$. Since  $f'(0)=\eta(a+b)<-2\eta^2 a b=g'(0)$, then there exists $\lambda^*\in (0,-\eta b)$ such that $f(\lambda^*)=g(\lambda^*)$, see Figure \ref{Fig0_C}.
	\end{enumerate}
\end{proof}

 \begin{figure}[hbt!]
  \begin{subfigure}[t]{0.32\textwidth}
    \includegraphics[width=1\textwidth]{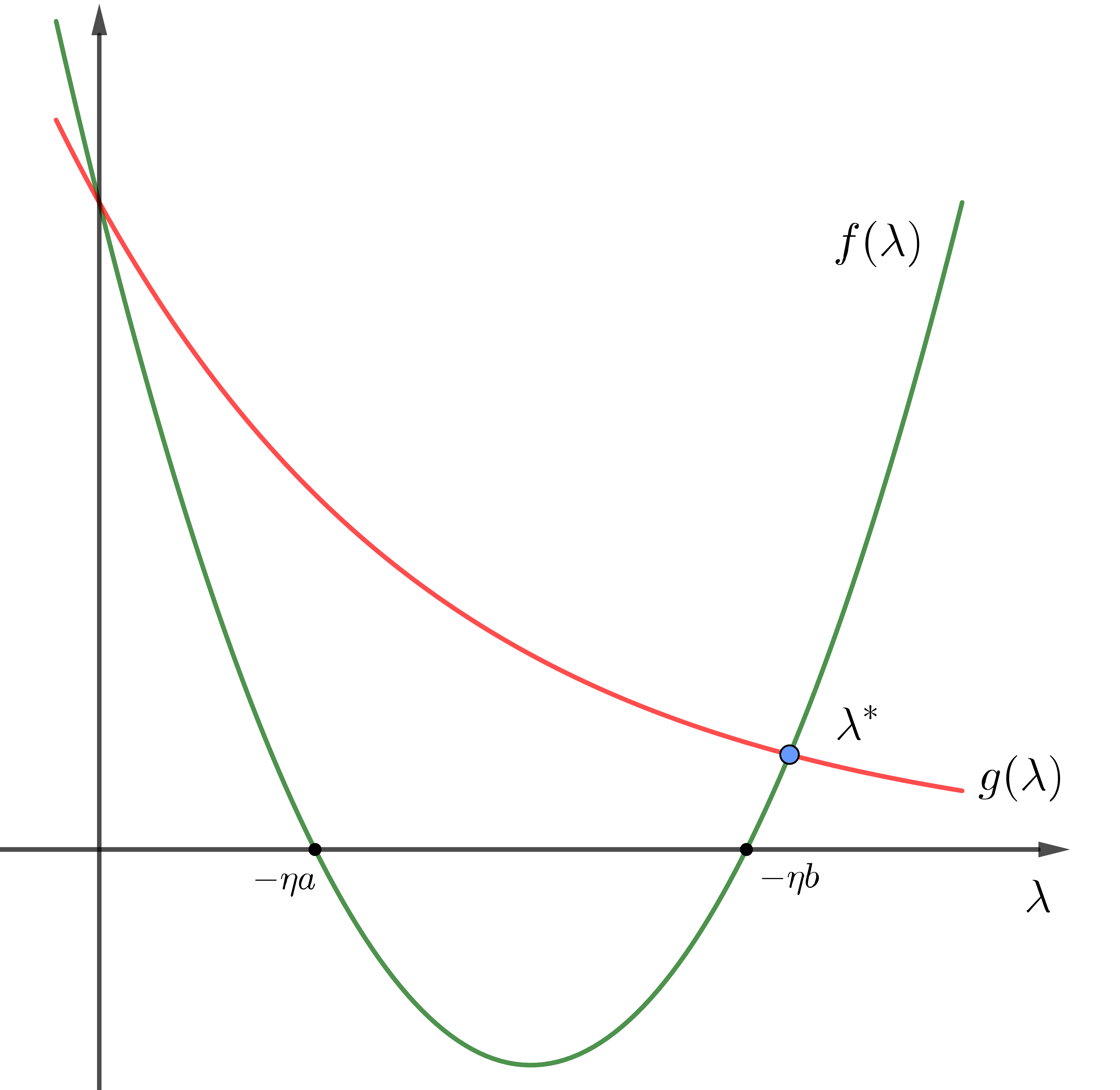}
    \caption{$ab>0$ and $a+b<0$.}
    \label{Fig0_A}
  \end{subfigure}\hfill
  \begin{subfigure}[t]{0.32\textwidth}
    \includegraphics[width=0.80\textwidth]{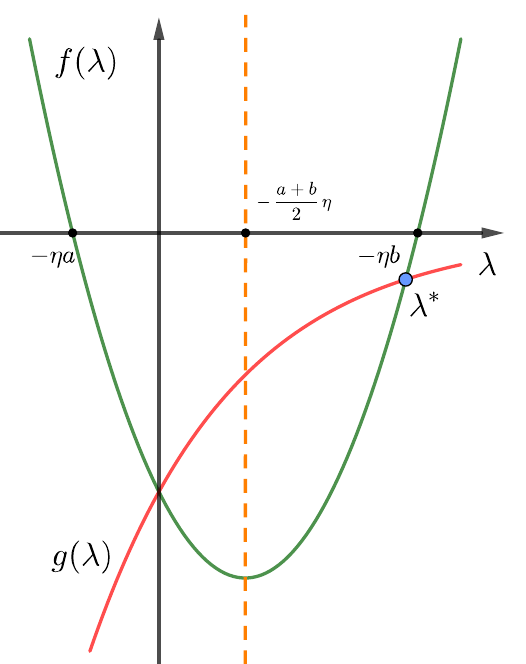}
    \caption{$ab<0$ and $a+b<0$.}
    \label{Fig0_B}
  \end{subfigure}\hfill
  \begin{subfigure}[t]{0.32\textwidth}
    \includegraphics[width=0.90\textwidth]{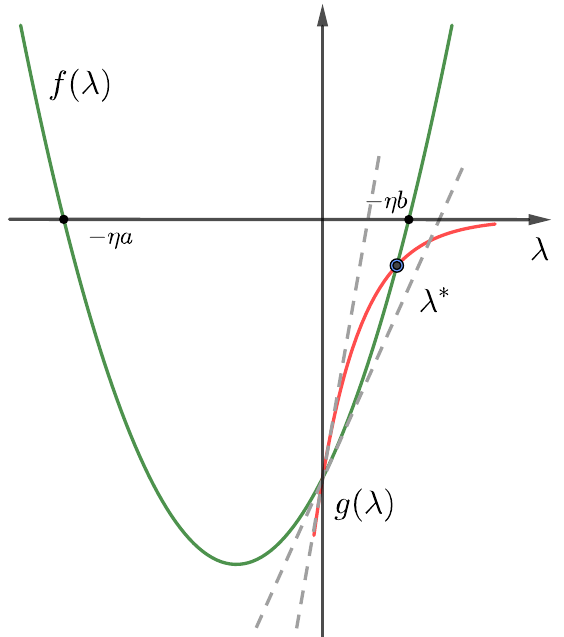}
    \caption{$ab<0$ and $a+b>0$.}
    \label{Fig0_C}
  \end{subfigure}
  \caption{Positive real roots in $\Delta(\lambda;\eta)=0$. } 
  \label{Fig0}
\end{figure}

 \begin{proposition}
	When $ab>0$ and $a+b>0$, $\Delta(\lambda;\eta)$ has no roots with positive real part.
	\end{proposition}
\begin{proof}
Since $ab>0$ and $a+b>0$, we have $a>0$ and $b>0$. Assume there is a root $\lambda^*=x+iy$ of $\Delta(\lambda;\eta)=0$ with $x>0$. Then, it follows from (\ref{f&g}) and (\ref{eq1543}) that
\begin{equation}
	|f(\lambda^*)|=|g(\lambda^*)|.
	\label{eq1201}
\end{equation}
Notice that, due \review{to} the positivity of $x$ we get
\[\left| {f\left( {{\lambda ^*}} \right)} \right| = \sqrt {{{(x + \eta a)}^2} + {y^2}} \sqrt {{{(x + \eta b)}^2} + {y^2}}  > {\eta ^2}ab\]
and
\[\left| {g\left( {{\lambda ^*}} \right)} \right| = {\eta ^2}ab{e^{ - 2x}} < {\eta ^2}ab.\]
Hence, $|f(\lambda^*)|>|g(\lambda^*)|$, which contradicts (\ref{eq1201}). Thus, all roots of $\Delta(\lambda;\eta)=0$ have nonpositive real parts when $ab>0$ and $a+b>0$.
\end{proof}
	
\begin{proposition}
\label{prop_zero_root}
$\lambda=0$ is a root of (\ref{chactEq}) for any $\eta$. If $\eta\neq \eta^*:=-\frac{a+b}{2ab}$ then  $\lambda=0$ is a simple root. Otherwise, it is a double root. The double multiplicity of $\lambda=0$ occurs only in the following cases. 
\begin{enumerate}
	\item [i.] $ab>0$ and $a+b<0$;
	\item [ii.] $ab<0$ and $a+b>0$.
\end{enumerate}

\end{proposition}

\begin{proof}
It is clear that $\Delta(0;\eta)=0$ and 
$\Delta'(0;\eta)=\eta\left(a+b+2ab\eta\right)$
where $'$ is the derivative with respect to $\lambda$. If $\eta\neq \eta^*$ then $\Delta'(0;\eta)\neq0$, and hence  $\lambda=0$ is a simple root.  When $\eta=\eta^*$, we have $\Delta'(0;\eta^*)=0$ and  
\[\Delta''(0;\eta^*)=-\frac{a^2+b^2}{ab}\neq 0.\]
Thus, $\lambda=0$ has double multiplicity. 

\noindent It is clear that $\eta^*$ exists if and only if 
	\[-\frac{a+b}{2ab}>0\iff \{ab>0\ \text{and}\ a+b<0\}\ \text{or}\ \{ab<0\ \text{and}\ a+b>0\}.\]
\end{proof}

  \begin{proposition}\label{prop_new1}
	When $ab<0$, $a+b>0$ and $a+b+2\eta ab\ge 0$,  $\Delta(\lambda;\eta)$ has no roots with positive real part.
\end{proposition}
\begin{proof}
\review{Note that the characteristic equation \eqref{chactEq} can be written as 
\[{\Delta }(\lambda ;\eta )= \lambda^2+\eta(a+b)\lambda  +\eta^2 ab\int_0^2 \lambda e^{-u\lambda} du=0.\]
Suppose that ${\Delta }(\lambda ;\eta )=0$ has root $\bar{\lambda}$ with ${\rm{Re}}(\bar{\lambda})>0$. Then
\[\left| {\bar \lambda (\bar \lambda  + \eta(a+b))} \right| = \left| {{\eta ^2}ab\bar \lambda \int\limits_0^2 {{e^{ - u\bar \lambda }}du} } \right| \le {\eta ^2}\left| {ab} \right||\bar \lambda |\left| {\int\limits_0^2 {{e^{ - u({\rm{Re}}(\bar \lambda ))}}du} } \right| \le 2{\eta ^2}\left| {ab} \right||\bar \lambda |.\]
Since $ab<0$ and $a+b+2\eta ab\ge 0$, we have
\[|\bar \lambda (\bar \lambda  + \eta(a+b))| \le  - 2{\eta ^2}ab|\bar \lambda | \le \eta (a + b)|\bar \lambda |\]
which is satisfied if $\bar \lambda=0$ (a contradiction) or 
\[|\bar{\lambda}+\eta(a+b)| \le \eta(a+b).\]
This implies that $\bar{\lambda}$  is in the disk of radius $\eta ( {a + b} )$ centred at the point $-\eta ( {a + b} )$ in the complex
plane. Thus, ${\rm{Re}}(\bar{\lambda})<0$ or $\bar{\lambda}=0$. In both cases we arrive at a
contradiction.}
\end{proof}

\review{Finally, we show that (\ref{chactEq}) does not have pure imaginary roots for any value of the parameters.}
			
\begin{proposition}
\label{prop_imag_roots}
 The characteristic equation (\ref{chactEq})
has no pure imaginary roots. 
\end{proposition}
\begin{proof}
Assume $\lambda=iy$ ($y>0$) is a root of (\ref{chactEq}). 
Separating the real and imaginary parts, we obtain 
\begin{align*}
   \eta^2ab-y^2&=\eta^2ab\cos(2y)\\
\eta(a+b)y&=-\eta^2ab\sin(2y)
\end{align*}
Squaring and adding these equations leads to
\begin{equation*}
   y^2\left(y^2+\eta^2(a^2+b^2)\right)=0.
\end{equation*}
which has no real roots. Thus, there are no roots of the form $iy$.
\end{proof}

The distribution of roots in (\ref{chactEq}) is summarized in Figure \ref{Fig2}.

 \begin{figure}[hbt!]
		\centering
			\includegraphics[width=1\textwidth]{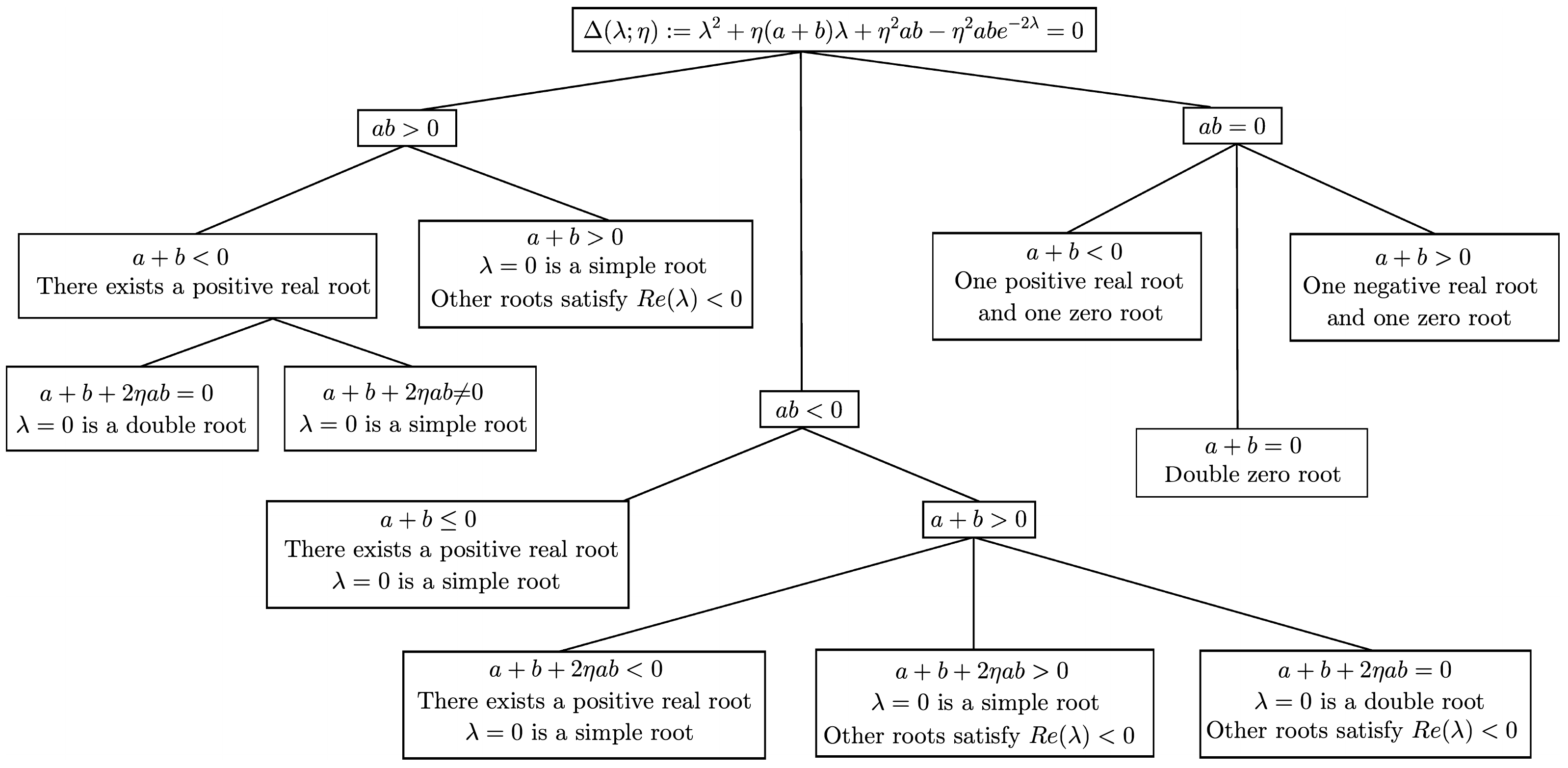}
	\caption{The distribution of roots in (\ref{chactEq}) \review{as discussed in Propositions \ref{prop000A}$-$\ref{prop_imag_roots}}.}
	 \label{Fig2}
	 \end{figure}

Recall the structure of the phase-locked solutions (\ref{phases}) of the phase model 
(\ref{phase_Mod}). From this we see that a phase-locked periodic solution of the 
original model  (\ref{Full_Mod}) corresponds to a line in the phase model 
(\ref{phase_Mod}), that is, when $\psi^*$ and $\omega^*$ are solutions 
of (\ref{sol_sys}), it follows that
\[\left\{ {\begin{array}{*{20}{c}}
{{\varphi _1} = {\omega ^*}t\qquad}&{\left( {\bmod 2\pi } \right)}\\
{{\varphi _2} = {\omega ^*}t + {\psi ^*}}&{\left( {\bmod 2\pi } \right)}
\end{array}} \right. \Rightarrow {\varphi _2} = {\varphi _1} + {\psi ^*}\left( {\bmod 2\pi } \right).\]
From Proposition \ref{prop_zero_root}, 
we know that for any $\tau>0$, $\Delta(\lambda;\eta(\tau))=0$ has a zero root. 
The simple zero root corresponds to the motion along these lines. It corresponds to 
the Floquet multiplier $1$ which is associated with the periodic solution
of the original model \eqref{Full_Mod}. Thus phase-locked solutions will be 
asymptotically stable if $\lambda=0$ is a simple root of the characteristic equation
\eqref{chactEq} and all other roots have negative real part.

\begin{remark}
\label{remark2}
The solution $\psi^*\ne 0,\pi$ is asymptotically stable for values of $a,b$ such
that $a>0$ and $b>0$ or $ab<0,\ a+b>0$ and $a+b+2\eta ab>0$.
 Since $H'$ is a $2\pi-$periodic function, the solutions 
$\psi^*$ and $2\pi-\psi^*$ have the same stability. 
\end{remark}

\begin{remark}
	\label{remark1}
Since $H$ is a $2\pi-$periodic function, $a=b=\frac{1}{\Omega}H'(\psi^*-\omega^*\eta-\Omega\tau)$ in (\ref{ab}) when $\psi^*=0,\pi$. Hence,  the stability of solutions when $\psi^*=0,\pi$ is determined by the sign of $H'(\psi^*-\omega^*\eta-\Omega\tau)$, that is, the solution is asymptotically stable when $H'(\psi^*-\omega^*\eta-\Omega\tau)>0$ and unstable when $H'(\psi^*-\omega^*\eta-\Omega\tau)<0$.
\end{remark}

\subsection{Bifurcation}\label{bif:sec}
Suppose that $\Omega$ and $\epsilon$ are fixed, but $\tau$ may be varied. From the 
discussion above, potential bifurcation points of the model \eqref{phase_Mod} 
are values $\tau=\tau^*$ where the characteristic equation for a particular phase-locked 
solution, $\psi^*,\omega^*$ has a double zero root. Let $\eta^*=\epsilon\Omega\tau^*$. 
When $\psi^*=0$ or $\pi$ there are two types of potential bifurcation points: 
\begin{itemize}
\item[(1)] $\tau^*$ where $H'(\psi^*-\omega^*\eta^*-\Omega\tau^*)=0$ (see Remark \ref{remark1});
\item[(2)] $\tau^*$ where $1+\eta^* \frac{1}{\Omega}H'(\psi^*-\omega^*\eta^*-\Omega\tau^*)=0$ (see Proposition~\ref{prop_zero_root}). 
\end{itemize}
For other values of $\psi^*$, Proposition~\ref{prop_zero_root} indicates there is a potential bifurcation point at 
\begin{itemize}
\item[(3)] $\tau^*$ where $\eta^*=-\frac{a+b}{2ab}$.  
\end{itemize}
Note that it is impossible to find an 
explicit \review{expression} for the bifurcation values because each of these conditions are 
implicit equations for $\tau^*$. 

Now we consider what type of bifurcations may occur at these points. We do not make 
a rigorous proof, which would require centre manifold and normal form theory. However,
we can make some plausible arguments based on the equations for the equilibrium
solutions.  Recall that $(\psi^*,\omega^*)$ with $\psi^*=0$ or $\pi$ defines
a phase-locked solution at $\tau$ if $F(\omega^*,\psi^*;\tau)=0$ where
\[ F(\omega,\psi^*;\tau)=\omega-\frac{1}{\Omega} H(\psi^*-\omega\eta-\Omega\tau). \]
Differentiating $F$ with respect to $\omega$ shows that the condition (2) 
corresponds to $F_\omega(\omega^*,\psi^*;\tau^*)=0$, that is, $\omega^*$ is 
a double root of $F$ when $\tau=\tau^*$. Thus as $\tau$ varies near $\tau^*$
we may expect that there should be two roots of $F$ near $\omega^*$ or none 
\footnote{More precisely, we expect this will occur if $F$ satisfies the further conditions
$F_{\tau}(\omega^*,\psi^*;\tau^*)=-\Omega(1+\epsilon\omega^*)/\eta^*\ne 0$ and 
$F_{\omega \omega}(\omega^*,\psi^*;\tau^*)=-(\eta^*)^2H''(\psi^*-\omega^*\eta^*-\Omega\tau^*)\ne 0$ 
\cite{Kuznetsov}.}.
Thus the bifurcation associated with condition (2) should be a saddle-node bifurcation 
involving two different phase-locked solutions with the same $\psi^*$. Note that this 
bifurcation is only physically relevant if $\eta^*>0$, i.e., 
$H'(\psi^*-\omega^*\eta^*-\Omega\tau^*)<0$. Thus, from Remark~\ref{remark1}, the
associated solutions will be unstable. In a similar manner one can show that
condition (3) corresponds to $(\psi^*,\omega^*)$ at $\tau=\tau^*$ being a point
 of tangency of the curves defined by equations \eqref{sol_sys}. Thus we expect 
it to correspond to a saddle-node bifurcation involving two out-of-phase solutions 
with different $\psi^*$. The stability of these
solutions will depend on which case of Proposition~\ref{prop_zero_root} applies.
Finally, we consider phase-locked solutions near $\psi=0$. Expanding 
equations \eqref{eq_H} and the first of \eqref{sol_sys} in $\psi$ and keeping the
two lowest order terms we have
\begin{eqnarray}
0&=&2H'(-\omega\eta-\Omega\tau)\psi+\frac{2}{3}H'''(-\omega\eta-\Omega\tau)\psi^3\\
\omega&=&\frac{1}{\Omega}\left(H(-\omega\eta-\Omega\tau)+H'(-\omega\eta-\Omega\tau)\psi\right).
\end{eqnarray}
Thus we see that $\psi^*=0$, \review{$\omega^*=H(-\omega^*\eta-\Omega\tau)/\Omega$,} is always a solution 
of this system and if there is $\tau^*$ such that condition (1) is satisfied
and $H'''(-\omega^*\eta^*-\Omega\tau^*)\ne 0$
then this will be a triple root of the system.
Thus we expect that condition 
(1) with $\psi^*=0$ corresponds to a pitchfork bifurcation where two out-of-phase solutions
are created near $0$. Similarly condition (1) with $\psi^*=\pi$ should correspond
to a pitchfork bifurcation where two out-of-phase solutions are created near $\pi$.

\review{Note that the phase interaction function $H$ can be represented by Fourier series expansion
\[H(\phi) = {a_0} + \sum\limits_{k = 1}^\infty  {\left[ {{a_k}\cos (k\phi) + {b_k}\sin (k\phi)} \right]}. \]
When the interaction function $H$ is represented by the first set of Fourier   modes 
\begin{equation}\label{Eq:Fourier_0}
    H(\phi)=a_0+a_1\cos(\phi)+b_1\sin(\phi),
\end{equation}
the authors in \cite{campbell2012phase} show that the out-of-phase solutions and pitchfork bifurcation  cannot occur in the phase model \eqref{phase_Mod} with small time delay.
However, it may occur when the time delay is large.  
Indeed, when $H$ has the form in \eqref{Eq:Fourier_0}, then it follows from \eqref{sol_sys} and \eqref{eq_H} that
\begin{align}
\Omega {\omega ^*} &= {a_0} + A({\omega ^*})\sin ({\psi ^*}) + B({\omega ^*})\cos ({\psi ^*}),\label{PB_FFM_1}\\
0 &= 2A({\omega ^*})\sin ({\psi ^*})\label{PB_FFM_2}
\end{align}
respectively, where 
\begin{align*}
  A({\omega ^*}) &= {b_1}\cos ({\omega ^*}\eta  + \Omega \tau ) + {a_1}\sin ({\omega ^*}\eta  + \Omega \tau ),\\
B({\omega ^*}) &= {a_1}\cos ({\omega ^*}\eta  + \Omega \tau ) - {b_1}\sin ({\omega ^*}\eta  + \Omega \tau ).
\end{align*}
Thus,  from $\sin ({\psi ^*})=0$, we have that $\psi ^*=0,\pi$ with the corresponding $\omega^*$  determined by 
    \begin{equation}\label{PB_FFM_3}
        \Omega {\omega ^*} - {a_0} =  \pm B({\omega ^*}).
    \end{equation}
 where the $+$ corresponds to $\psi^*=0$
and the $-$ to $\psi^*=\pi$. Also,  from $A({\omega ^*})=0$ we  determine $\omega ^*$ and the corresponding $\psi^*$ is obtained from
       \begin{equation}\label{PB_FFM_4}
       \cos ({\psi ^*}) = \frac{{\Omega {\omega ^*} - {a_0}}}{{B({\omega ^*})}}.
    \end{equation}
Consequently, we have the following cases
\begin{itemize}
    \item   if $\left| {\Omega {\omega ^*} - {a_0}} \right| < \left| {B({\omega ^*})} \right|$, then two out-of-phase solutions $\psi^*$ and $2\pi-\psi^*$ exist,
     \item if $\left| {\Omega {\omega ^*} - {a_0}} \right| = \left| {B({\omega ^*})} \right|$, then one  solution exists ($\psi^*=0$ or $\psi^*=\pi$),
     \item if $\left| {\Omega {\omega ^*} - {a_0}} \right| > \left| {B({\omega ^*})} \right|$, then no solution satisfying \eqref{PB_FFM_4} exists.
\end{itemize}
Note that $H'({-\omega ^*}\eta  - \Omega \tau )=A(\omega ^*)$ and $H'({\pi-\omega ^*}\eta  - \Omega \tau )=-A(\omega ^*)$.
Thus, the solutions $0$ and $\pi$ change stability when $A(\omega ^*)=0$ where $\omega^*$ satisfies \eqref{PB_FFM_3}. 
As $\tau$ varies, out-of-phase solutions will disappear if $ \frac{{\Omega {\omega ^*} - {a_0}}}{{B({\omega ^*})}}-1$ changes its sign from negative to positive. 
When  $ \frac{{\Omega {\omega ^*} - {a_0}}}{{B({\omega ^*})}}=1$, then $\psi^*=0$. Hence, a pitchfork bifurcation   occurs at $\psi^*=0$. Similarly when $ \frac{{\Omega {\omega ^*} - {a_0}}}{{B({\omega ^*})}}=-1$ a pitchfork bifurcation occurs at
$\psi^*=\pi$.}

\subsection{The full model with small delay}
\label{Sec2_small_Delay}

  When the time delay, $\tau$, in (\ref{Full_Mod}) is relatively small, in the sense that $\Omega \tau=\mathcal{O}(1)$, it follows from the theory of averaging that the time delay $\tau$
enters the interaction function $H$ in  (\ref{phase_Mod}) as a
phase shift \cite{ermentrout2009delays,hoppensteadt2012weakly,izhikevich1998phase,campbell2012phase}.  
 In \cite{campbell2012phase}, the authors considered  this case and consequently the time delay $\eta$ in the phase model (\ref{phase_Mod}) was neglected, and hence, it becomes 
\beqn
\label{phase_Mod22}
\frac{d\varphi_1}{d{t}}&=\frac{1}{\Omega} H(\varphi_2(t)-\varphi_1(t)-\Omega\tau),\\
\frac{d\varphi_2}{d{t}}&=\frac{1}{\Omega} H(\varphi_1(t)-\varphi_2(t)-\Omega\tau),
\eeqn
Therefore,  
 they were able to reduce (\ref{phase_Mod}) into a one dimensional ordinary differential equation
\begin{equation}
\label{eq_small_delay_1}
	\frac{d \phi}{d t}=-2 \epsilon[H(\phi-\Omega \tau)-H(-\phi-\Omega \tau)].
\end{equation} 
where $\phi=\varphi_2-\varphi_1$.
The existence of phase-locked solutions of (\ref{eq_small_delay_1}) was discussed in \cite{campbell2012phase} without introducing the frequency deviation $\omega$. Hence, the in-phase and anti-phase solutions were unique. Moreover, the stability of the phase-locked solution $\phi^*$ in (\ref{eq_small_delay_1}) was determined by the sign of 
\begin{equation}
\label{eq_small_delay_2}
{\widehat H}'(\phi^*):=\overline{a}+\overline{b} 
\end{equation}
where $\overline{a}=H^{\prime}(\phi^*-\Omega \tau)$ and $\overline{b}=H^{\prime}(-\phi^*-\Omega \tau)$. If ${\widehat H}'(\phi^*)>0$ then $\phi^*$  is asymptotically stable and if ${\widehat H}'(\phi^*)<0$  it is unstable. When ${\widehat H}'(\phi^*)=0$ the stability is not determined by the linearization.  

\begin{remark}
\review{In \cite{campbell2012phase}, 
due to the reduction of the two dimensional  system  
(\ref{phase_Mod22}) into a single equation (\ref{eq_small_delay_1}), the  zero root was omitted  in characteristic equation.}  Indeed, 
the characteristic equation of (\ref{eq_small_delay_1}) is $\lambda+{\widehat H}'(\phi^*)=0$  
while the characteristic equation of (\ref{phase_Mod22}) is
\begin{equation} 
\lambda({\lambda} + {\widehat H}'(\phi^*))  =0.
\end{equation}
It is clear that the latter characteristic equation always has a zero root. 
\end{remark}

\review{Now we compare these results with what happens when $\tau$ is small, i.e.,
$\Omega \tau=\mathcal{O}(1)$, in our model \eqref{phase_Mod}}.  Recall that 
$\eta=\epsilon\Omega\tau$ thus the assumption on $\tau$ 
implies that $\eta=\mathcal{O}(\epsilon)$. Also, note that the phase difference 
$\phi^*$ of the phase locked solutions for the model \eqref{eq_small_delay_2} is 
the same as the phase deviation difference $\psi^*$ for our model.  

First consider the existence of phase-locked solutions. For our model we must
solve the equations \eqref{eq_H} and one of \eqref{sol_sys} simultaneously for $\psi$ and
$\omega$. When $\eta=\mathcal{O}(\epsilon)$, however, to first order in $\epsilon$
the $H$ function no longer depends on $\omega$. Thus phase-locked solutions
are determined by $\psi^*$ satisfying $H_{\tau}(\psi^*)=0$, with
$\omega^*=\frac{1}{\Omega}H(\psi^*-\Omega\tau)$. \review{This equation for $\psi^*$ is the
same as in \cite{campbell2012phase}}. In \cite{campbell2012phase} they did not solve for $\omega^*$ as it was not needed 
to determine the phase-locked solutions or their stability. It remains to consider the
uniqueness of the in-phase and anti-phase solutions. 
From equations (\ref{omega0}) and (\ref{omegapi}), these solutions correspond to frequency 
deviations $\omega^*$ satisfying $F(\omega^*,\psi^*)=0$
with $\psi^*=0,\pi$, respectively. 
Since $H$ and $H'$ are continuous and $2\pi$ periodic they are bounded. Thus we see that
$\lim_{\omega \rightarrow \pm\infty}F(\omega)=\pm\infty$. Further, recalling
\eqref{Fp0def}, since $\eta=\mathcal{O}(\epsilon)$, $F_\omega(\omega^*,\psi^*)>0$.  
Thus for any $\tau$ sufficiently small, there 
will be a unique frequency deviation $\omega^*$ for $\psi^*=0$ and for $\psi^*=\pi$. 
This is consistent with the results in \cite{campbell2012phase} which have only
one in-phase and anti-phase solution for each value of $\tau$. 

Now consider the stability of the phase-locked solutions. Recall that the stability 
for our model is summarized in Figure~\ref{Fig2}. When $\eta=\mathcal{O}(\epsilon)$, 
$\sgn(a+b+ab\eta)\approx \sgn(a+b)$, thus the conditions for stability/instability of 
phase-locked solutions of our model reduce to the stability if $a+b>0$ and
\review{instability} if $a+b<0$. Further $a\approx \overline{a}$ and $b\approx\overline{b}$, thus
the stability results of our model reduce to those of  \cite{campbell2012phase}
when $\Omega\tau=\mathcal{O}(1)$. The key point is that, regardless of the size
of $\tau$, the countable infinity of complex roots of the characteristic equation 
\eqref{chactEq} all have negative real part. Thus the stability 
of the phase-locked solutions is determined by finitely many real roots, 
and it is possible for an ordinary differential equation to accurately reflect
this stability.

In Section \ref{Sec3_3}, we will show numerically that our model with 
$\Omega\tau=\mathcal{O}(1)$ fully recovers \cite[Figure 4b]{campbell2012phase} and 
\cite[Figure 5b]{campbell2012phase}. 

\section[Application to Morris-Lecar model]{Application to Morris-Lecar oscillators with diffusive\\ coupling} 
\label{Sec3}

In this section we  apply the results from the previous sections to a network of dimensionless Morris-Lecar oscillators with time delayed diffusive coupling, see e.g.,  \cite{prasad2008universal,buric2003dynamics}. This model is given by
\beqn
\label{MLmodel}
{v'_i} &= {I_{app}} - {g_{Ca}}{m_\infty }({v_i})({v_i} - {v_{Ca}}) - {g_K}{w_i}({v_i} - {v_K}) - {g_L}({v_i} - {v_L}) - \epsilon ({v_j}(t - \tau ) - {v_i}(t)),\\
{w'_i} &= \varphi \lambda ({v_i})({w_\infty }({v_i}) - {w_i}),
\eeqn
for $i,j=1,2$ such that $i\neq j$, where
\beqnn
 m_{\infty}(v) &=\frac{1}{2}\left(1+\tanh \left(\left(v-\nu_{1}\right) / \nu_{2}\right)\right), \\
 w_{\infty}(v) &=\frac{1}{2}\left(1+\tanh \left(\left(v-\nu_{3}\right) / \nu_{4}\right)\right), \\
  \lambda(v) &=\cosh \left(\left(v-\nu_{3}\right) /\left(2 \nu_{4}\right)\right).
\eeqnn

Using the parameter set I$\backslash$II from \cite[Table 1]{campbell2012phase}, when there is no
coupling in the network each oscillator has a unique exponentially
asymptotically stable limit cycle with period $T=23.87\backslash13.81$ corresponding to frequency $\Omega=0.2632\backslash 0.455$.
The normalized  system, such that the frequency is $1$, corresponding to (\ref{MLmodel}) is 
\beqn
\label{MLmodelNor}
{v'_i} &= \frac{1}{\Omega}({I_{app}} - {g_{Ca}}{m_\infty }({v_i})({v_i} - {v_{Ca}}) - {g_K}{w_i}({v_i} - {v_K}) - {g_L}({v_i} - {v_L})) - \frac{\epsilon}{\Omega} ({v_j}(t - \Omega \tau ) - {v_i}(t)),\\
{w'_i} &= \frac{1}{\Omega}(\varphi \lambda ({v_i})({w_\infty }({v_i}) - {w_i})),
\eeqn
$i=1,2$. 
\review{Note that this is in the form (\ref{Full_Mod}) with ${\bf{X}}_i(t)=(v_{i}(t), w_{i}(t))^T$ and the function $\mathbf{G}:\mathbb{R}^2\times\mathbb{R}^2\to \mathbb{R}^2$ is  given by ${\bf{G}}=(G_1,G_2)$ where ${{G}}_1({\bf{X}}_1(t), {\bf{X}}_2(t)) = \frac{1}{\Omega}(v_2(t−\Omega \tau)−v_1(t))$ and ${{G}}_2({\bf{X}}_1(t), {\bf{X}}_2(t)) =0$.
Then,  the phase model interaction function $H$ is given by \eqref{New:funH}.}

For each parameter set, the authors in  
\cite{campbell2012phase} solved \eqref{New:funH} numerically and
calculated the approximation of the phase model interaction function $H$ by the first  five terms of its Fourier series. These are given by
	\beqn
	\label{H_I_Num}
	H_I(\phi)&=2.915252-2.684797 \cos (\phi)-0.3278022 \cos (2 \phi)\\
	&~ ~~+0.05596774 \cos (3 \phi)+0.0351635 \cos (4 \phi)+	4.908449 \sin (\phi) \\
&~ ~~-0.7020183 \sin (2 \phi)-0.09934668 \sin (3 \phi)-0.01104474  \sin (4 \phi),\\	H_{II}(\phi)&=0.6271561-0.5209326 \cos (\phi)-0.08538575 \cos (2 \phi)\\
&~ ~~	-0.005648281\cos (3 \phi) -0.0002642404 \cos (4 \phi)+	1.595618\sin (\phi)\\
&~ ~~-0.04727176\sin (2 \phi)-0.00301241 \sin (3 \phi)-0.002760313 \sin (4 \phi)	\eeqn
corresponding to the  parameter sets I and II, respectively,  see \cite[Table 2]{campbell2012phase}. 
 \review{Note that the two parameter sets represent limit cycles which are created
by different bifurcations as the input current $I_{app}$ is
varied. For parameter set I the limit cycle is created in a
saddle-node on an invariant circle bifurcation, while
for parameter set II the limit cycle is created in a supercitical
Hopf bifurcation. The chosen parameter values have $I_{app}$
slightly larger than the bifurcation values.}

In \cite {campbell2012phase} the authors studied how
small epsilon needed to be for the phase model to
faithfully represented the behaviour of the full system
\eqref{MLmodelNor}, in the case of small delay. They found that for
parameter set I $\epsilon$ could be as large as $0.05$ while
for parameter set II epsilon should not exceed $0.001$.
Therefore, in the rest of this section, we take $\epsilon=0.05$ with parameter set I and $\epsilon=0.001$ when we use parameter set II.
Consequently, we choose $\tau\ge 75.988$ for  parameter set I and $\tau\ge 2197.8$ for parameter set II so that $\epsilon \Omega \tau=\mathcal{O}(1)$. Moreover, we 
compare our results with the results in \cite {campbell2012phase} when the time delay, $\tau$, in (\ref{Full_Mod}) is relatively small.

\subsection{In-phase and anti-phase solutions}
\label{Sec3_1}
To find $\omega^*$ corresponding  to the in-phase and anti-phase solutions, $\psi^*=0,\pi$, we solve (\ref{omega0}) and (\ref{omegapi}) \review{with $H$ given by either $H_{I}$ or $H_{II}$
from \eqref{H_I_Num}. Note that these equations
can only be solved numerically due to the complicated
form of $H_I$ and $H_{II}$. For particular values of $\tau$, we
represent these solutions graphically in Figure \ref{fig0} as the
intersection points of the line $y=\omega$ and the curve
$y=H(-\omega \eta -\Omega \tau)/\Omega$.}
In (\ref{omega0}), the slope of the right hand side  at any $\omega$ is $\ell_{\tau}=-\epsilon\tau H'(-\omega\eta-\Omega\tau)$.  
Then, by applying the stability condition in Remark \ref{remark1}, we see that the 
in-phase solution is stable when 
the line $y=\omega$ intersects the curve of the function $H(-\omega\eta-\Omega\tau)/\Omega$ at a point where it has negative slope, while it is unstable when the intersection is at a point with positive slope, see Figure \ref{fig0}. 
When the line $y=\omega$ alternates from intersecting the curve of $H(-\omega\eta-\Omega\tau)/\Omega$ at a point with positive slope to intersecting it at a point with negative slope, the solutions $\omega^*$ alternate between stable and unstable, see Figure \ref{fig0}.
\review{For fixed $\Omega$ and $\epsilon$, as $\tau$ increases the curve 
$H(-\eta \omega-\Omega\tau)=H(-\Omega\tau(1+\epsilon\omega))$ compresses horizontally
causing the creation and destruction of intersection points.
For specific values 
$\tau=\tau^*_1>0$, an intersection point will occur at the point where the function $H$ has 
slope one, i.e., the curve $y=H$ will be tangent to the line $y=\omega$ at these
values of $\tau$, see Figure \ref{fig000_B}.
Near such points, i.e., for $\tau$ slightly bigger or smaller,
there exist two consecutive intersection points both of which are unstable, see 
Figure \ref{fig000_C}.
Then, as $\tau$ changes further to $\tau_2^*$, one unstable point 
quickly passes through the point where $H$ has zero slope and becomes  stable, see 
Figure \ref{fig0_B}.
The values $\tau^*_1$ correspond to the saddle-node 
bifurcations of in-phase and anti-phase solutions discussed in Section \ref{bif:sec}.}
We will discuss the points $\tau_2^*$ later.  In Figure \ref{fig_SI_tau_omega_psi}, 
we plot $\omega^*$  corresponding to $\psi^*=0,\pi$ for various values of the time 
delay $\tau$, showing the many co-existing solutions which can occur and the
transitions of the solutions as $\tau$ varies.
These solutions were found by
implementing the algorithm from \cite{rahimian2011new}  in \textit{Wolfram Mathematica} to find all the solutions of \eqref{omega0} or \eqref{omegapi}.

\begin{figure}[hbt!]
	\centering
	\begin{subfigure}{0.49\textwidth}
		\centering
		\includegraphics[width=1\linewidth]{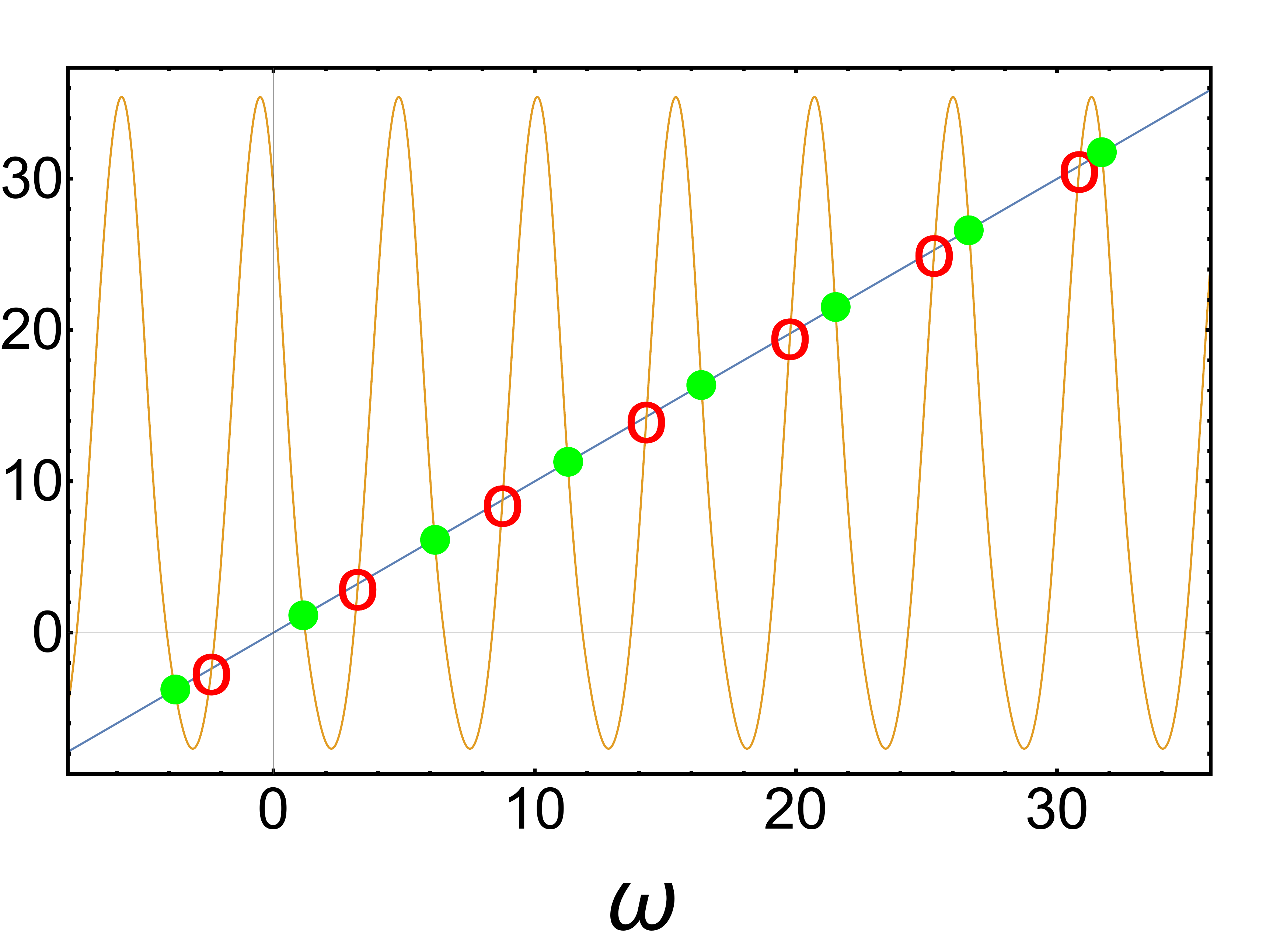}
		\caption{$ H_I(-\omega\eta-\Omega\tau)/\Omega$ and $y=\omega$. $\tau=90$.}
		\label{fig0_A}
	\end{subfigure}%
	\begin{subfigure}{.49\textwidth}
		\centering
		\includegraphics[width=1\linewidth]{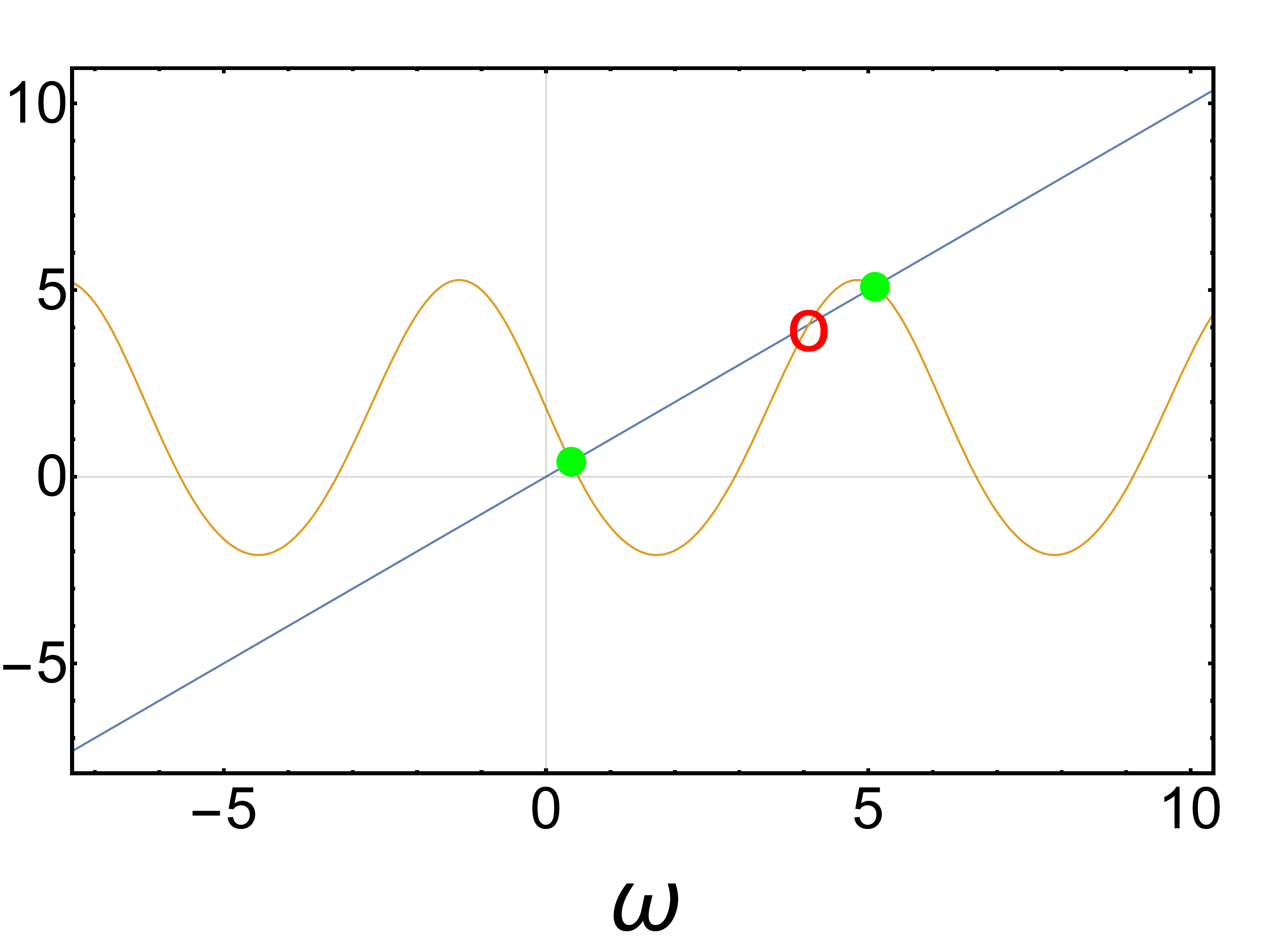}
		\caption{$ H_{II}(-\omega\eta-\Omega\tau)/\Omega$ and $y=\omega$. $\tau=2236$.}
		\label{fig0_B}
	\end{subfigure}
		\caption{Graphical representation of the solutions to (\ref{omega0}) with fixed $\tau$. The circles \textcolor{green}{
			$\CIRCLE$}/\textcolor{red}{   $\Circle$} represent stable/unstable solutions.} 
			\label{fig0}
\end{figure}

\begin{figure}[hbt!]
	\centering
	\begin{subfigure}{0.33\textwidth}
		\centering
		\includegraphics[height=3.8cm,width=5.5cm]{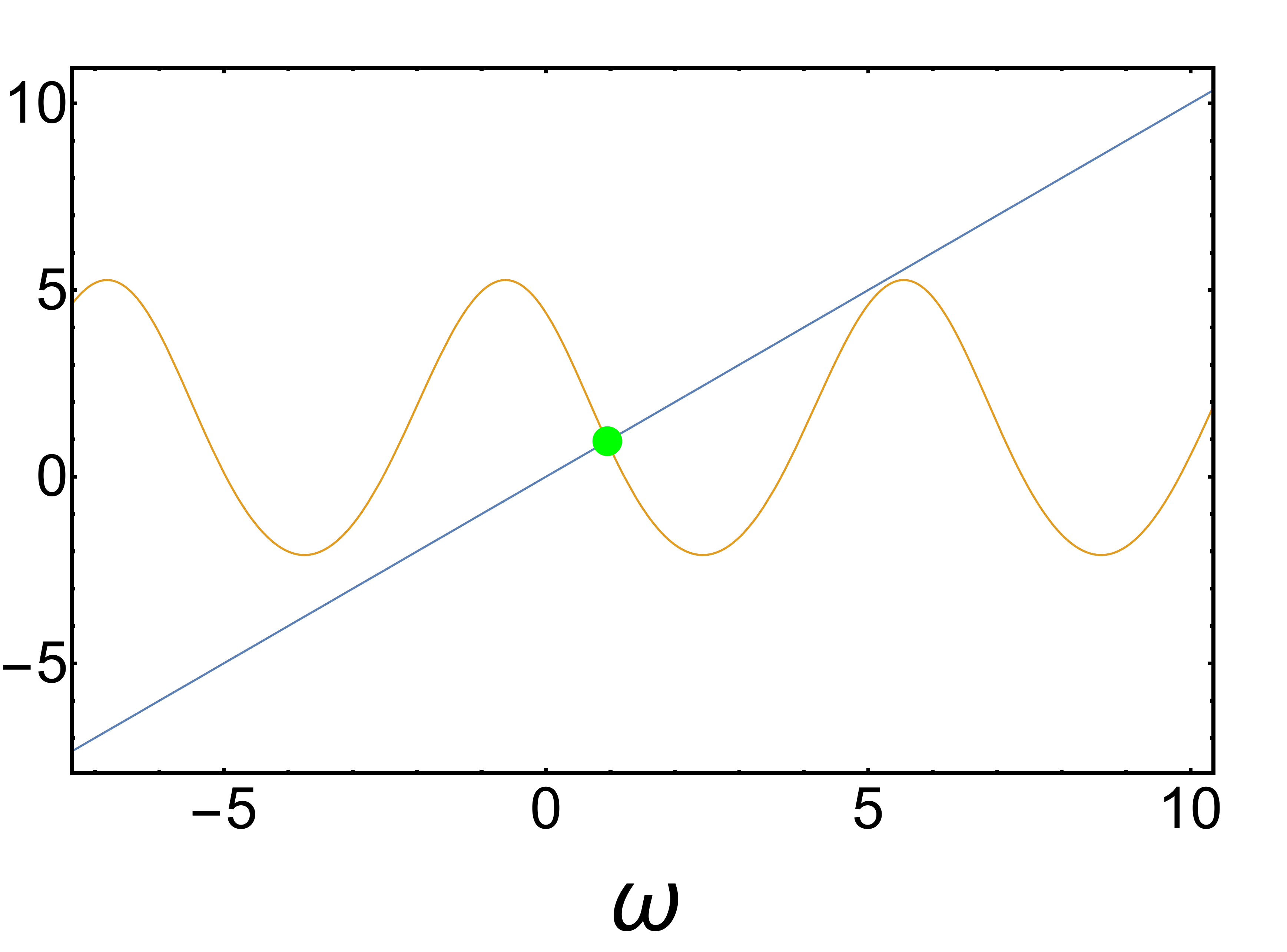}
		\caption{ $\tau=2234.4$.}
		\label{fig000_A}
	\end{subfigure}%
	\begin{subfigure}{.33\textwidth}
		\centering
		\includegraphics[height=3.8cm,width=5.5cm]{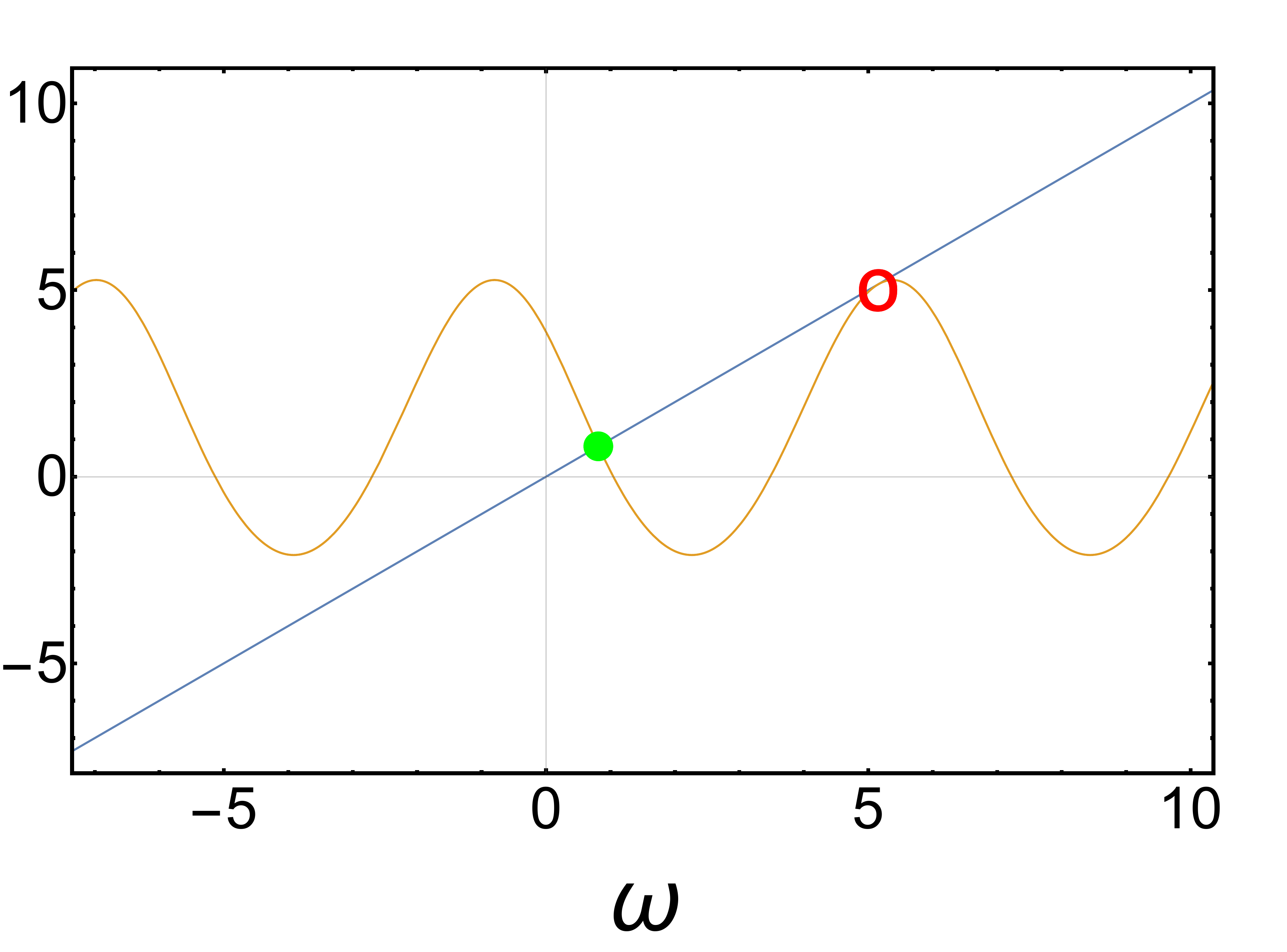}
		\caption{$\tau=2234.78$.}
		\label{fig000_B}
	\end{subfigure}
	\begin{subfigure}{.33\textwidth}
		\centering
		\includegraphics[height=3.8cm,width=5.5cm]{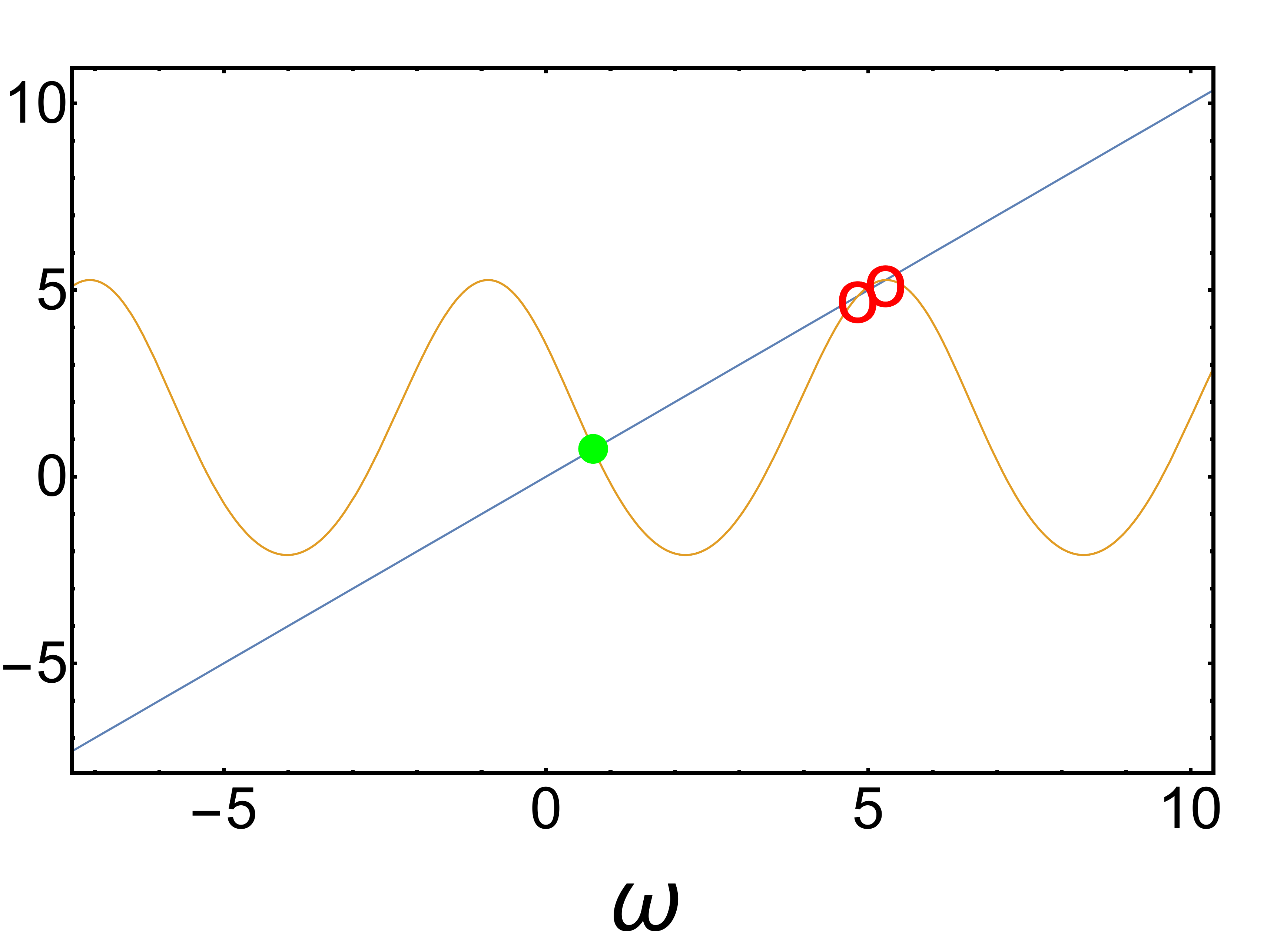}
		\caption{$\tau=2235$.}
		\label{fig000_C}
	\end{subfigure}
	\caption{Graphical 
solutions of $ H_{II}(-\omega\eta-\Omega\tau)/\Omega$ and $y=\omega$ with fixed $\tau$. The circles \textcolor{green}{
			$\CIRCLE$}/\textcolor{red}{$\Circle$} represent stable/unstable solutions.} 
			\label{fig000}
\end{figure}

To compare prediction of the phase model (\ref{phase_Mod}) and solutions of the full model (\ref{MLmodel}),  we 
solve (\ref{MLmodelNor}) 
 numerically with parameter sets I and II with various values of $\tau$ and different initial conditions. 
The
initial conditions are of the form
\beqn
\label{InCoform}
\left(v_{1}(t), w_{1}(t), v_{2}(t), w_{2}(t)\right)^T=\left(v_{10}, w_{10}, v_{20}, w_{20}\right)^T \quad t\in[-\tau\Omega,0].
\eeqn
Figure \ref{Fig_V_vs_t} shows time series of $v_i$ in (\ref{MLmodelNor}) with different initial conditions. We notice the coexistence of in-phase 
solutions with different
frequencies when $\tau=110$ with parameter set I.
The numerical solutions are obtained by using \textit{Wolfram Mathematica}. We use the command \textsf{NDSolve}  to solve the full model numerically.

 \begin{figure}[hbt!]
		\centering
			\includegraphics[width=0.7\textwidth]{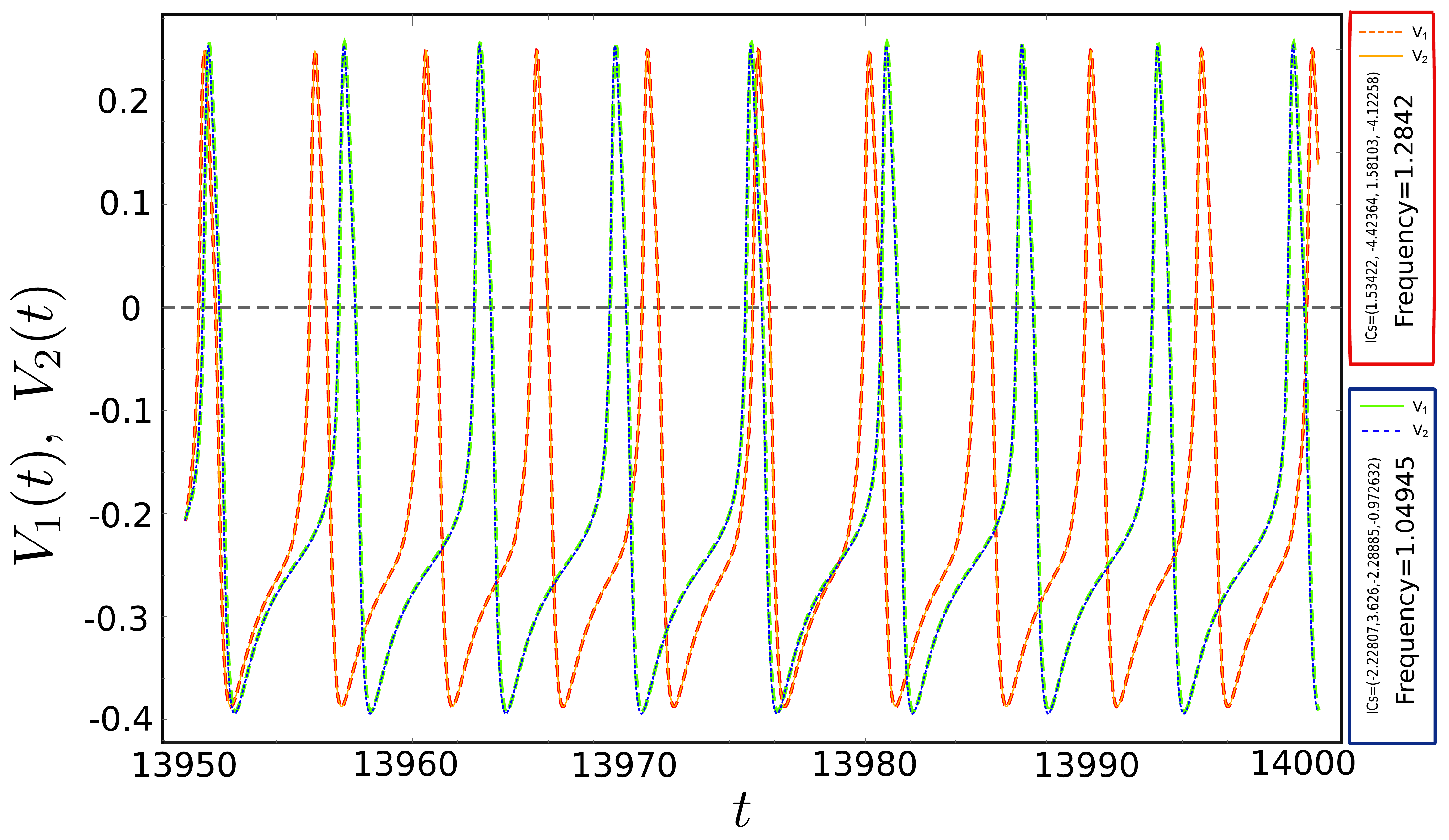}
	\caption{The coexistence of in-phase 
solutions of (\ref{MLmodelNor}) with different
frequencies when $\tau=110$ with parameter sets I. We take different initial conditions: $(1.53422,-4.42364,1.58103,-4.12258)^T$ for red/orange curves and $(-2.22807,3.626,-2.28885,-0.972632)^T$ for blue/green curves.}
	 \label{Fig_V_vs_t}
	 \end{figure}

When $\epsilon=0$,  each uncoupled equation in (\ref{MLmodelNor}) has $2\pi-$periodic solution, that is,  the frequency of each oscillator is unity. 
Consequently, when $\epsilon\ne 0$ and equation (\ref{MLmodelNor}) has a phase-locked solution,  the phase of the first oscillator  is $\theta_1(t)= t+\omega^*\epsilon t$ and that of the second oscillator is $\theta_2(t)= t+\omega^*\epsilon t+\psi^*$ where $\omega^*$ is the frequency deviation and $\psi^*$ is the phase shift. Thus, the frequency of each oscillator is $1+\omega\epsilon^*$, and the period $\mathcal{T}$ is approximately 
$$\mathcal{T}=\frac{2\pi}{1+\omega^*\epsilon}.$$
From the numerical solution of (\ref{MLmodelNor}) for a stable phase-locked solution,
we can calculate the period $\mathcal{T}$ of the oscillators and determine the 
approximate frequency deviation from 
\beqn
\omega^*\approx \frac{1}{\epsilon}\left( \frac{2\pi}{\mathcal{T}}-1 \right).
\label{appo}
\eeqn
Figure \ref{fig_SI_tau_omega_psi} shows the coexistence of stable in-phase and anti-phase periodic solutions and demonstrates that the approximation of $\omega^*$ from (\ref{appo}) is close to a stable solution of the phase model.
\review{The values of $\omega^*$ with 
the normalized error
\begin{equation}\label{Nerror}
{{\rm E_N}}=\frac{(\omega^* {\rm{~in~the~phase~model}})-(\omega^* {\rm{~in~the~full~model}})}{\omega^* {\rm{~in~the~full~model}}}
\end{equation}
are shown in Tables \ref{table_psi_0} and  \ref{table_psi_0_S2}. Note that the  quantity ${{\rm E_N}}$ is the normalized error with respect to the size of $\omega^*$ in the full model.}
Except for a few cases, the phase model gives a very accurate prediction of the values
of $\omega^*$. The phase model predicted stable phase-locked solutions that we did not find numerically, however, it is possible that further exploration with different initial conditions might find them.
 
\begin{figure}[hbt!]
	\centering
	\begin{subfigure}{.5\textwidth}
		\centering
		\includegraphics[width=1\linewidth]{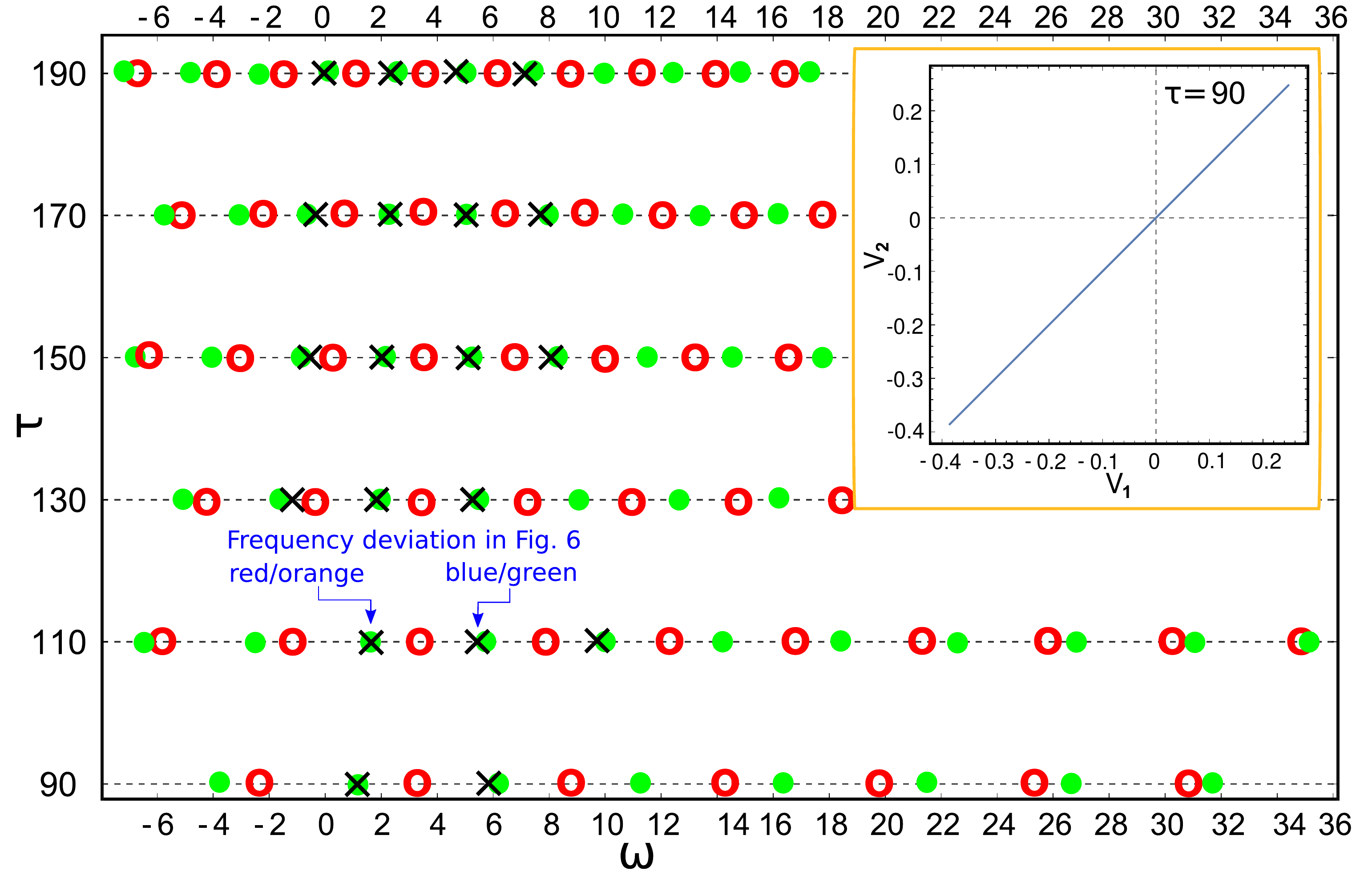}
		\caption{Parameter set I. $\psi^*=0$.}
		\label{fig_SI_tau_omega_psi_1}
	\end{subfigure}%
	\begin{subfigure}{.5\textwidth}
		\centering
		\includegraphics[width=1\linewidth]{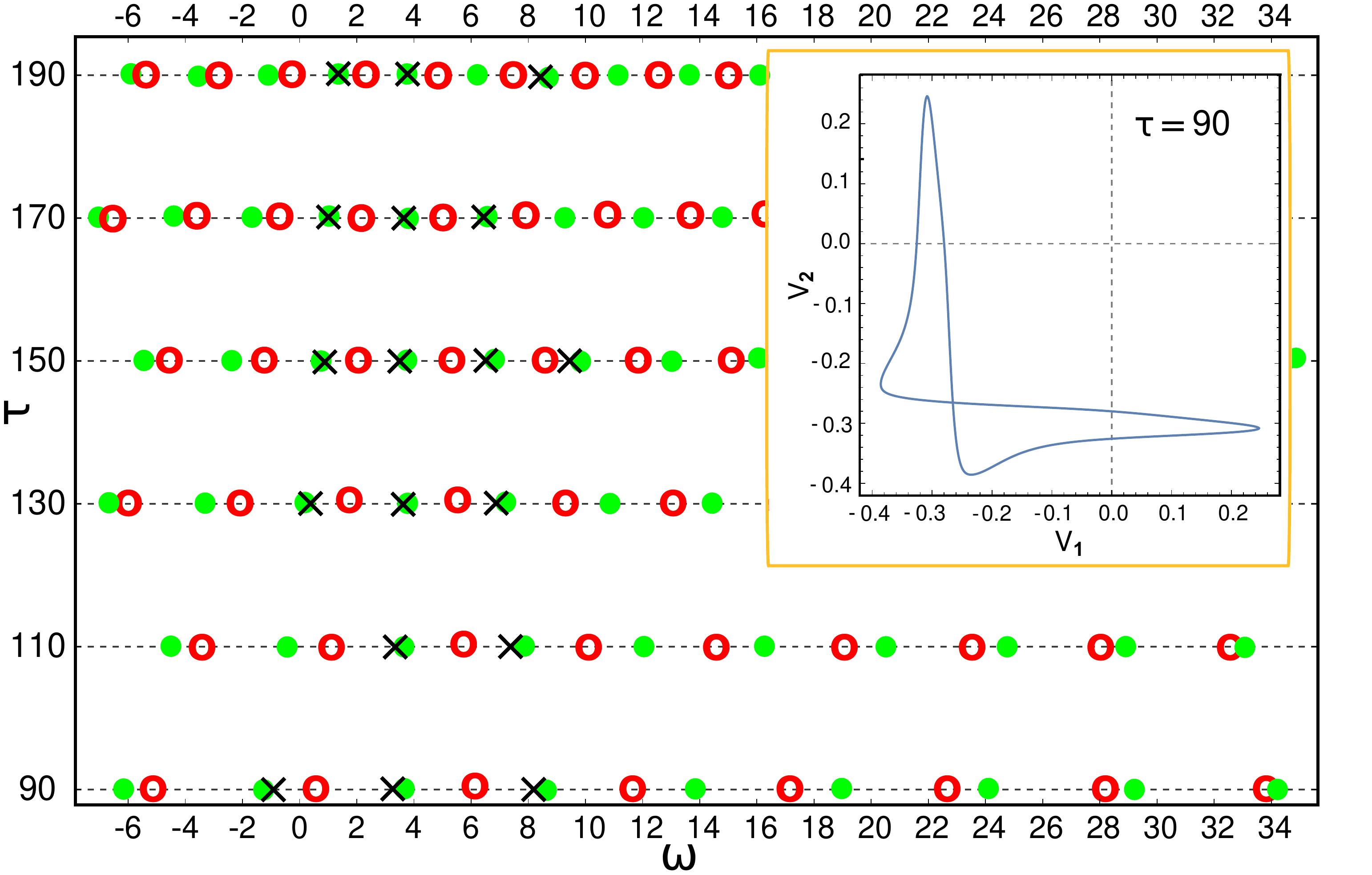}
\caption{Parameter set I. $\psi^*=\pi$}
		\label{fig_SI_tau_omega_psi_2}
	\end{subfigure}
	~\bigskip
	
	\begin{subfigure}{.5\textwidth}
		\centering
		\centering
	\includegraphics[width=1\linewidth]{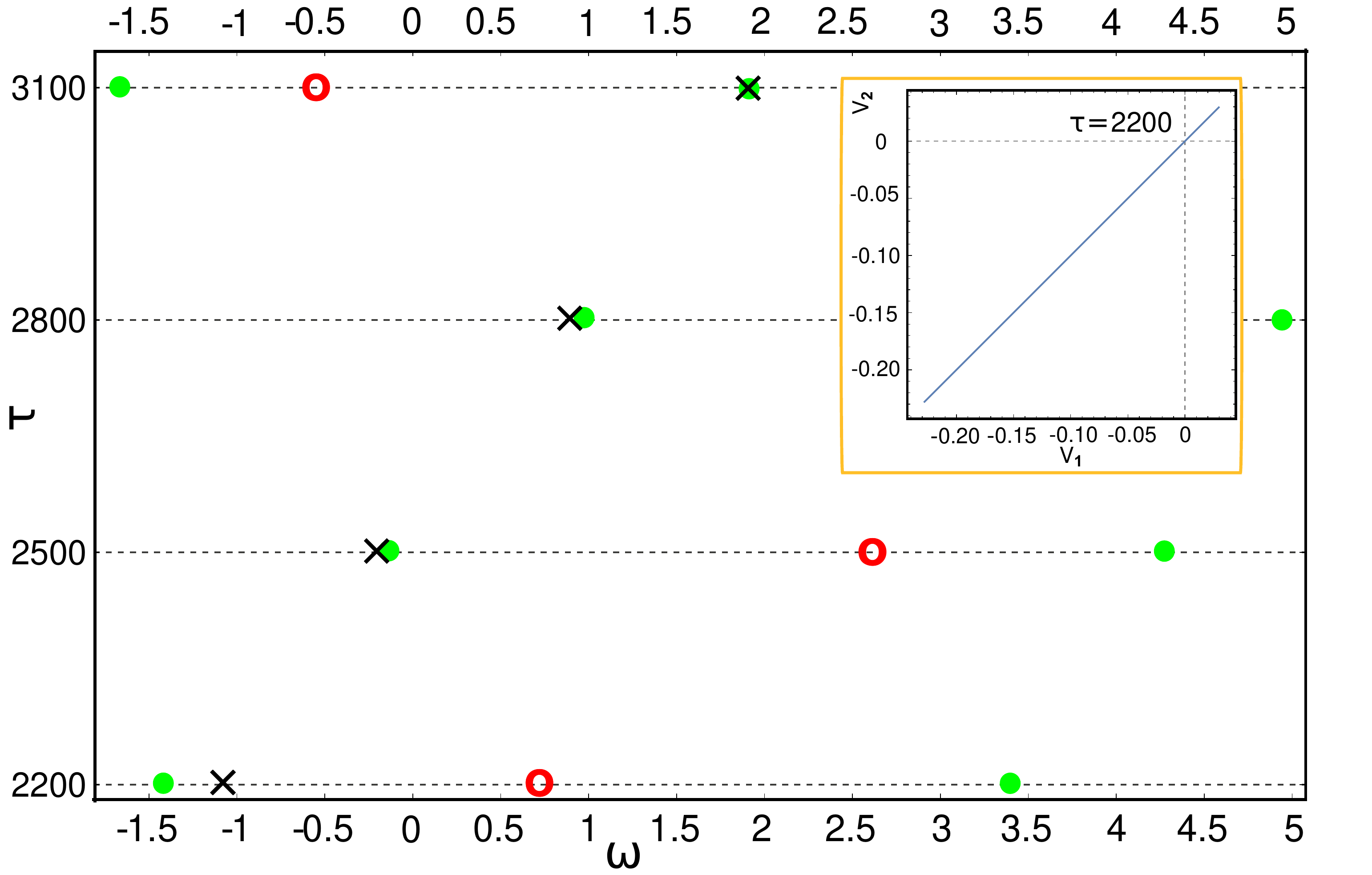}
		\caption{Parameter set II. $\psi^*=0$}
		\label{fig_SII_tau_omega_psi_1}
	\end{subfigure}%
	\begin{subfigure}{.5\textwidth}
		\centering
	\includegraphics[width=1\linewidth]{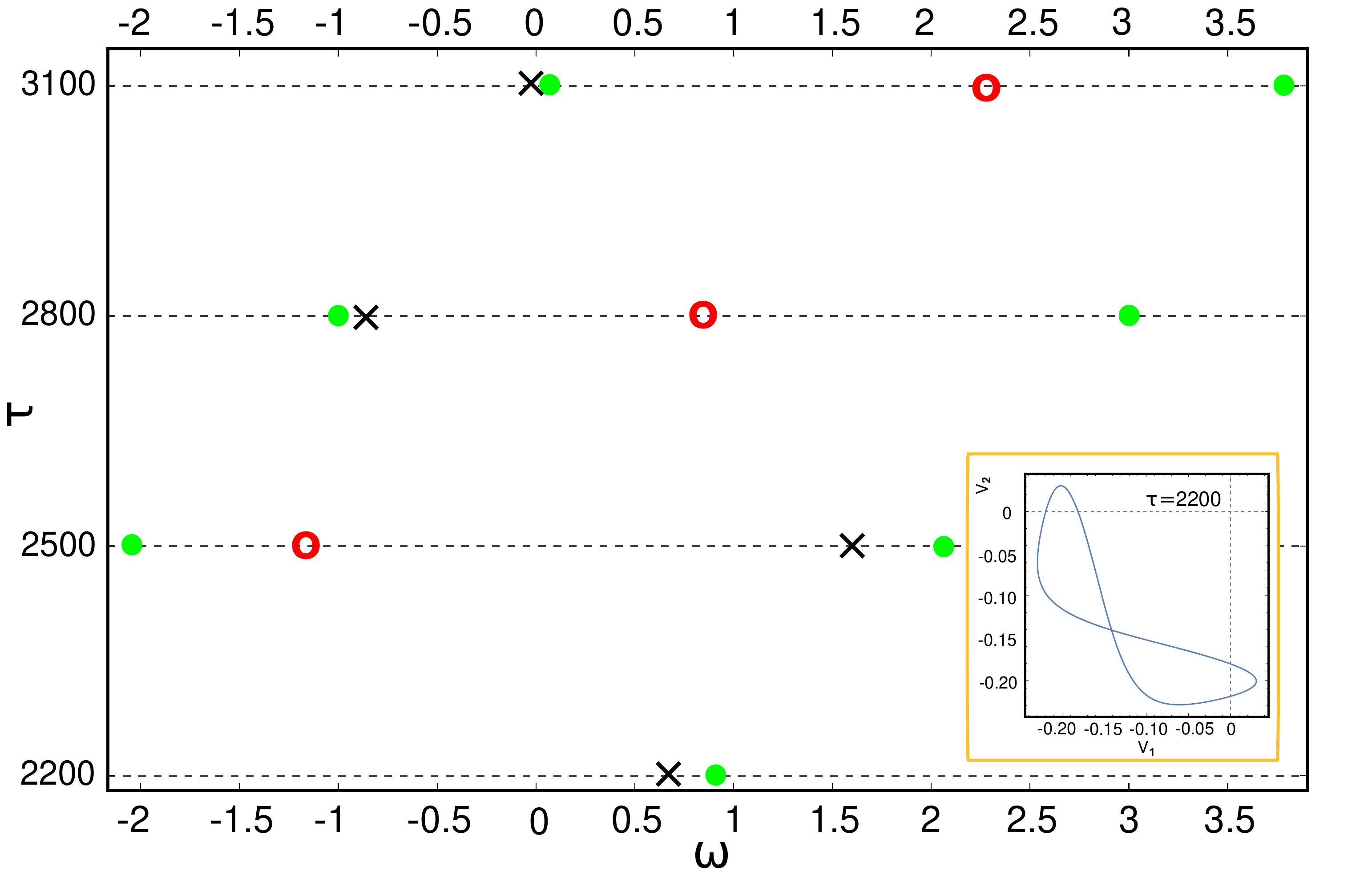}
		\caption{Parameter set II. $\psi^*=\pi$}
		\label{fig_SII_tau_omega_psi_2}
			\end{subfigure}
	\caption{ The circles   \textcolor{green}{$\CIRCLE$}/\textcolor{red}{$\Circle$} represent stable/unstable solutions to the phase model (\ref{phase_Mod}) corresponding to (\ref{MLmodelNor}) and $\boldsymbol{\times}$ represents the calculated  $\omega^*$ for each stable phase-locked periodic solution found by numerical integration of the full model (\ref{MLmodelNor}) with parameter sets I and II. The insets are in-phase and anti-phase periodic solutions of (\ref{MLmodelNor}). 
	The initial conditions for all simulations was of the form (\ref{InCoform}). For
the insets the values of $\left(v_{10}, w_{10}, v_{20}, w_{20}\right)^T$ are as follows. (a) $(1.53422,-4.42364,1.58103,-4.12258)^T$ (b) $(1.57882,4.1827,2.78262,0.358165)^T$ (c) $(1.892, -0.296437, -1.05518, 1.09985)^T$ (d) $(-1.72448,-1.46442,4.4848,1.31822)^T$} 
	\label{fig_SI_tau_omega_psi}
\end{figure}

\begin{table}[ht]
	\centering
	\scalebox{0.90}{
		\begin{tabular}{|c|c|c|c|c|c|c|c|c|c|c|c|}
			\hline
			 \multicolumn{12}{|c|}{\cellcolor[HTML]{C0C0C0}{$\psi^*=0$}}\\ \cline{1-12}
			\multirow{2}{*}{} & \multicolumn{3}{c|}{$\tau=90$} &  & \multicolumn{3}{c|}{$\tau=110$} &  \multirow{4}{*}{}  & \multicolumn{3}{c|}{$\tau=130$} \\ \cline{2-4}\cline{6-8} \cline{10-12} 
			&     {\scriptsize Phase Model}   &   {\scriptsize Full Model}     &  {\scriptsize ${\rm E_N}$}   &  &     {\scriptsize Phase Model}   &   {\scriptsize Full Model}     & {\scriptsize ${\rm E_N}$}       &  &        {\scriptsize Phase Model}   &   {\scriptsize Full Model}     &  {\scriptsize ${\rm E_N}$}     \\ \cline{1-4}\cline{6-8} \cline{10-12} 
			\multirow{2}{*}{$\omega^*$} & \small   $1.1446$   &  \small   $0.988971$   & \small $0.1574$      &  &  \small  $1.61521$   &   \small   $1.44374$   &   \small $0.1188$    &   &   \small  $-1.55286$   &  \small   $−1.31666$  &   \small  $-0.1794$    \\ \cline{2-4} \cline{6-8}\cline{10-12} 
			& \small  $6.17015$     & \small  $5.68396$    & \small   $0.0855$    &  & \small   $9.95867$   &  \small $9.67259$    & \small  $0.0296$     &  &  \small  $5.48587$    &   \small  $5.16611$   &  \small  $0.0619$    \\ \hline
			\multirow{2}{*}{} & \multicolumn{3}{c|}{$\tau=150$} &  & \multicolumn{3}{c|}{$\tau=170$} & \multirow{4}{*}{}  & \multicolumn{3}{c|}{$\tau=190$} \\ \cline{2-4}\cline{6-8} \cline{10-12} 
			&     {\scriptsize Phase Model}   &   {\scriptsize Full Model}     &  {\scriptsize ${\rm E_N}$}   &    &     {\scriptsize Phase Model}   &   {\scriptsize Full Model}     &  {\scriptsize ${\rm E_N}$}       &  &        {\scriptsize Phase Model}   &   {\scriptsize Full Model}     &  {\scriptsize ${\rm E_N}$}     \\ \cline{1-4}\cline{6-8} \cline{10-12} 
			\multirow{2}{*}{$\omega^*$} &   \small  $-0.860951$   &   \small $−0.766851$    &   \small $-0.1227$    &  &    \small  $2.38636$  &    \small$2.23525$    &   \small $0.0676$    &  &  \small $0.101849$     &   \small $0.092197$    & \small $0.1047$      \\ \cline{2-4}\cline{6-8} \cline{10-12} 
			&    \small  $2.19521$  &   \small $2.03459$    &   \small $0.0789$  & &  \small$5.11589$     &    \small$4.87837$   &   \small $0.0487$    &  &   \small$7.44601$     &   \small $7.16109$    &   \small $0.0398$    \\ \hline
						 \multicolumn{12}{|c|}{\cellcolor[HTML]{C0C0C0}{$\psi^*=\pi$}}\\ \cline{1-12}
\multirow{2}{*}{} & \multicolumn{3}{c|}{$\tau=90$} &  & \multicolumn{3}{c|}{$\tau=110$} &  \multirow{4}{*}{}  & \multicolumn{3}{c|}{$\tau=130$} \\ \cline{2-4}\cline{6-8} \cline{10-12} 
			&     {\scriptsize Phase Model}   &   {\scriptsize Full Model}     &  {\scriptsize ${\rm E_N}$}   &  &     {\scriptsize Phase Model}   &   {\scriptsize Full Model}     &  {\scriptsize ${\rm E_N}$}       &  &        {\scriptsize Phase Model}   &   {\scriptsize Full Model}     &  {\scriptsize ${\rm E_N}$}     \\ \cline{1-4}\cline{6-8} \cline{10-12} 
			\multirow{2}{*}{$\omega^*$} & \small $-1.32864$  &  \small    $-1.08976$   & \small $-0.2192$      &  &  \small  $3.68511$   &   \small   $3.39257$   &   \small $0.0862$    &   &   \small $0.193408$   &  \small   $0.169843$ &   \small  $0.1387$    \\ \cline{2-4} \cline{6-8}\cline{10-12} 
			& \small $8.70987$    & \small   $8.244$   & \small  $0.0565$    &  & \small   $7.86$  &  \small $7.36178$    & \small  $0.0677$    &  &  \small  $7.26467$    &   \small  $6.86785$   &  \small  $0.0578$   \\ \hline
			\multirow{2}{*}{} & \multicolumn{3}{c|}{$\tau=150$} &  & \multicolumn{3}{c|}{$\tau=170$} & \multirow{4}{*}{}  & \multicolumn{3}{c|}{$\tau=190$} \\ \cline{2-4}\cline{6-8} \cline{10-12} 
			&     {\scriptsize Phase Model}   &   {\scriptsize Full Model}     &  {\scriptsize ${\rm E_N}$}   &    &     {\scriptsize Phase Model}   &   {\scriptsize Full Model}     &  {\scriptsize ${\rm E_N}$}       &  &        {\scriptsize Phase Model}   &   {\scriptsize Full Model}     &  {\scriptsize ${\rm E_N}$}     \\ \cline{1-4}\cline{6-8} \cline{10-12} 
			\multirow{2}{*}{$\omega^*$} &   \small  $0.663926$   &   \small $0.600602$    &   \small $0.1054$   &  &    \small  $1.02812$ &    \small $0.946137$    &   \small $0.0867$   &  &  \small $3.76194$     &   \small $3.58466$    & \small  $0.0495$     \\ \cline{2-4}\cline{6-8} \cline{10-12} 
			&    \small  $9.92786$  &   \small $9.41424$   &   \small $0.0546$    &  &  \small $3.74932$    &    \small $3.55154$   &   \small $0.0557$   &  &   \small $8.67729$    &   \small $8.34016$    &   \small $0.0404$    \\ \hline
	\end{tabular}}
	\caption{\sloppy Comparison of $\omega^*$ between the phase  model prediction and the full model (\ref{MLmodelNor}) when $\psi^*=0,\pi$  with parameter set I. The quantity ${\rm E_N}$ is defined in \eqref{Nerror}.
  }
	\label{table_psi_0}
\end{table}

\begin{table}[ht]
	\centering
	\scalebox{1}{
		\begin{tabular}{|c|c|c|clc|c|c|}
			\hline
						 \multicolumn{8}{|c|}{\cellcolor[HTML]{C0C0C0}{$\psi^*=0$}}\\ \cline{1-8}
			\multirow{2}{*}{} & \multicolumn{3}{c|}{$\tau=2200$}                                                                                                                         & \multicolumn{1}{l|}{} & \multicolumn{3}{c|}{$\tau=2500$}                                                                                                    \\ \cline{2-4} \cline{6-8} 
			
			&{\scriptsize Phase Model} & {\scriptsize Full Model}& \multicolumn{1}{c|}{\scriptsize ${\rm E_N}$} & \multicolumn{1}{l|}{} &{\scriptsize Phase Model} & {\scriptsize Full Model}& {\scriptsize ${\rm E_N}$} \\ \cline{1-4} \cline{6-8} 
			$\omega^*$          & \small $-1.418638$           & \small $-1.10245$          & \multicolumn{1}{c|}{\small $-0.2868$}            & \multicolumn{1}{l|}{} & \small $-0.159684$          & \small $-0.223359$         & \small $0.2851$           \\ \hline
			\multirow{2}{*}{} & \multicolumn{3}{c|}{$\tau=2800$}                                                                                                                         & \multicolumn{1}{l|}{} & \multicolumn{3}{c|}{$\tau=3100$}                                                                                                    \\ \cline{2-4} \cline{6-8} 
			
			&{\scriptsize Phase Model} & {\scriptsize Full Model}& \multicolumn{1}{c|}{{\scriptsize ${\rm E_N}$}} & \multicolumn{1}{l|}{} &{\scriptsize Phase Model} & {\scriptsize Full Model}& {\scriptsize ${\rm E_N}$} \\ \cline{1-4} \cline{6-8} 
			$\omega^*$          & \small $0.9584587$           & \small $0.720141$          & \multicolumn{1}{c|}{\small $0.3309$}            & \multicolumn{1}{l|}{} & \small $1.914776$            & \small $1.880915$           & \small $0.018$            \\ \hline
									 \multicolumn{8}{|c|}{\cellcolor[HTML]{C0C0C0}{$\psi^*=\pi$}}\\ \cline{1-8}
	\multirow{2}{*}{} & \multicolumn{3}{c|}{$\tau=2200$}                                                                                                                         & \multicolumn{1}{l|}{} & \multicolumn{3}{c|}{$\tau=2500$}                                                                                                    \\ \cline{2-4} \cline{6-8} 
			&{\scriptsize Phase Model} & {\scriptsize Full Model}& \multicolumn{1}{c|}{\scriptsize ${\rm E_N}$} & \multicolumn{1}{l|}{} &{\scriptsize Phase Model} & {\scriptsize Full Model}& {\scriptsize ${\rm E_N}$} \\ \cline{1-4} \cline{6-8} 
			$\omega^*$          & \small $0.913646$           & \small $0.647528$          & \multicolumn{1}{c|}{\small $0.411$}            & \multicolumn{1}{l|}{} & \small $2.06122$          & \small $1.58187$         & \small $0.303$           \\ \hline
			\multirow{2}{*}{} & \multicolumn{3}{c|}{$\tau=2800$}                                                                                                                         & \multicolumn{1}{l|}{} & \multicolumn{3}{c|}{$\tau=3100$}                                                                                                    \\ \cline{2-4} \cline{6-8} 
			
			&{\scriptsize Phase Model} & {\scriptsize Full Model}& \multicolumn{1}{c|}{{\scriptsize ${\rm E_N}$}} & \multicolumn{1}{l|}{} &{\scriptsize Phase Model} & {\scriptsize Full Model}& {\scriptsize ${\rm E_N}$} \\ \cline{1-4} \cline{6-8} 
			$\omega^*$          & \small $-1.011297$           & \small $-0.880915$          & \multicolumn{1}{c|}{\small $-0.148$}            & \multicolumn{1}{l|}{} & \small $0.0496789$            & \small $-0.03801$           & \small $-2.307$            \\ \hline
	\end{tabular}}
	\caption{\sloppy Comparison of $\omega^*$ between the phase  model prediction and the full model (\ref{MLmodelNor}) when $\psi^*=0,\pi$  with parameter set II. The quantity The quantity ${\rm E_N}$ is defined in \eqref{Nerror}.}
	\label{table_psi_0_S2}
\end{table}

\subsection{Out-of-phase solutions}
\label{Sec3_2}

\review{To find  phase-locked solutions other than the in-phase and anti-phase solutions, we fix $\tau$ and solve 
\beqn
\omega^*&=\frac{1}{\Omega} H_{II}(\psi^*-\omega^*\eta-\Omega\tau),\\
\omega^*&=\frac{1}{\Omega} H_{II}(-\psi^*-\omega^*\eta-\Omega\tau)
\label{sysA1}
\eeqn
for $\omega^*$ and $\psi^*$. 
Figure \ref{Contour} shows all solutions to (\ref{sysA1}) when $\tau=100\backslash2205$ with the parameter set I$\backslash$II. 
As seen for the existence  of in-phase and anti-phase solutions in Section \ref{Sec3_1},   the number of phase-locked solutions with the parameter set I is bigger that II. For the purpose of clarity in the  bifurcation figures,   we consider the parameter set II in this section.}

In Figure \ref{Fig_new}, we observe that there are four non-trivial phase-locked solutions:  
$\psi^*_1=1.85996$ 
and $\psi^*_2=2.13981$ in $(0,\pi)$; and $\psi^*_3=2\pi-\psi^*_1=4.42323$ and $\psi^*_4=2\pi-\psi^*_2=4.14338$ in $(\pi,2\pi)$. Moreover, we have $\omega^*_1=\omega^*_3=0.14125$ and $\omega^*_2=\omega^*_4=0.14125$ where $\omega^*_i$ is the corresponding  frequency deviation  to $\psi^*_i$, $i=1,2,3,4$.
 This agrees with Proposition \ref{prop_1}.

\begin{figure}[hbt!]
	\centering
	\begin{subfigure}{0.5\textwidth}
		\centering
		\includegraphics[width=0.67\textwidth]{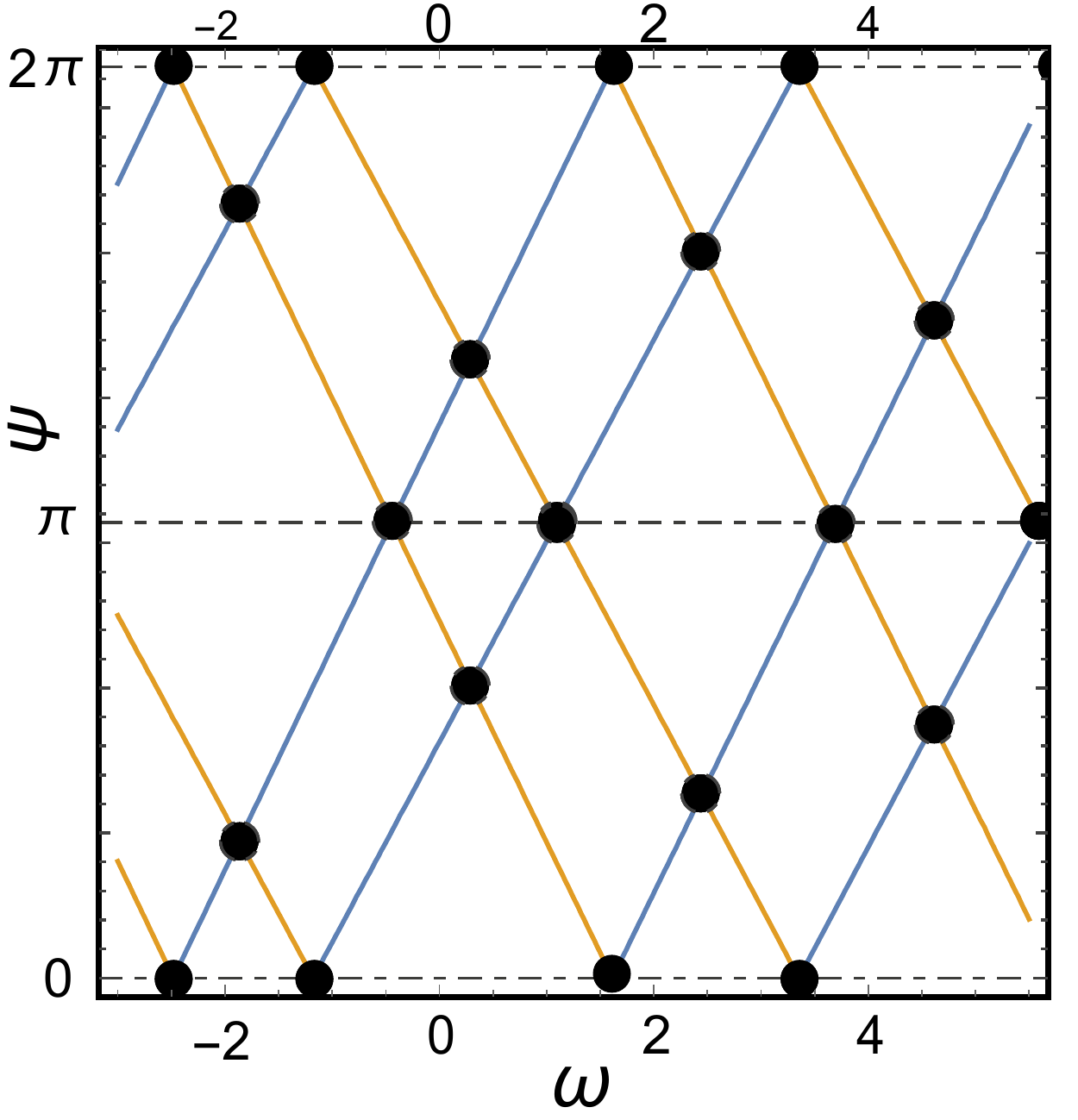}
		\caption{Parameter set I with $\tau=100$.}
		 \label{Fig_new_1}
	\end{subfigure}%
	\begin{subfigure}{.5\textwidth}
		\centering
			\includegraphics[width=1\textwidth]{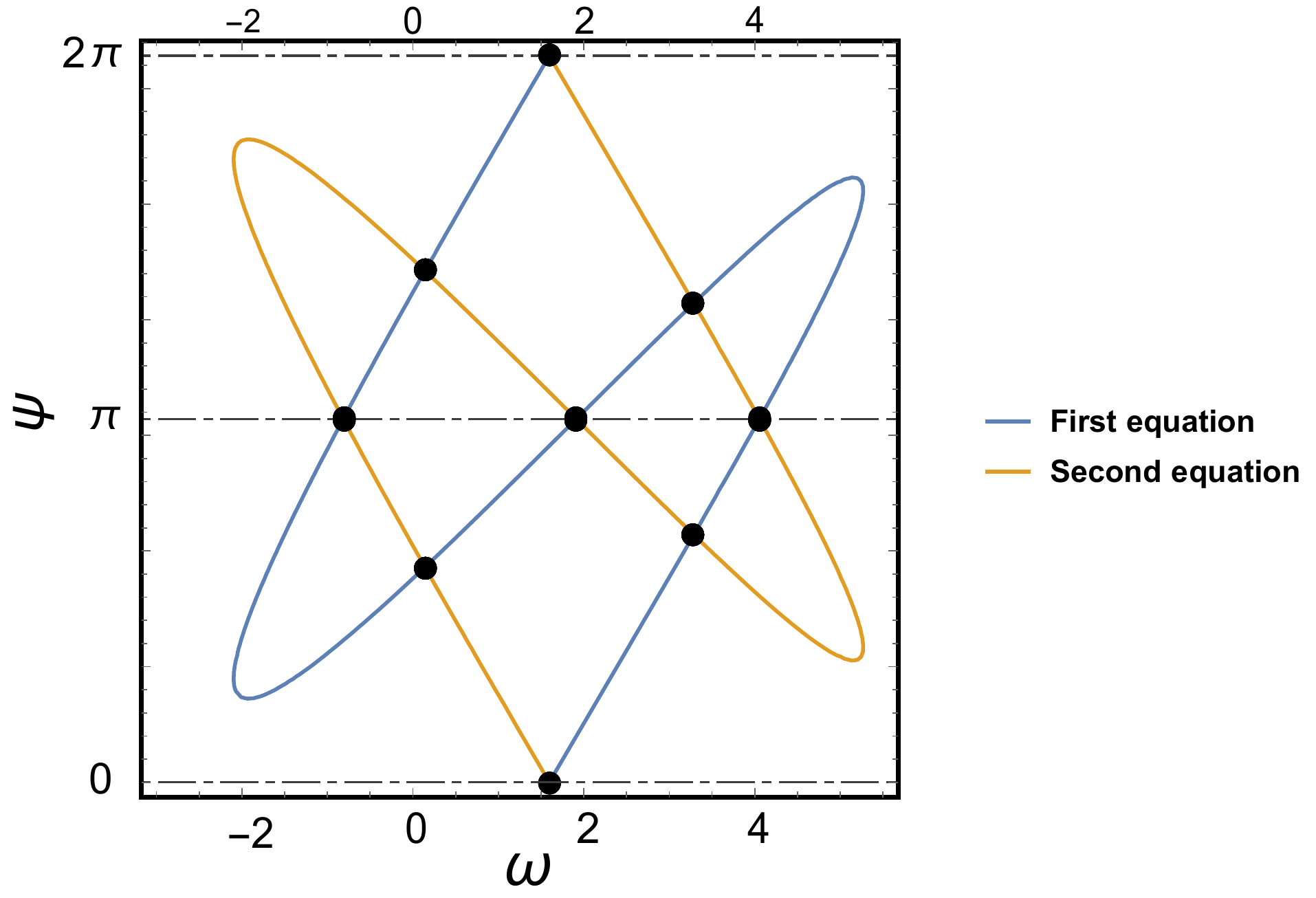}
		\caption{Parameter set II with  $\tau=2205$.}
		 \label{Fig_new}
	\end{subfigure}
		\caption{Contour plot of the equations in (\ref{sysA1}) to show the graphical solutions of (\ref{sysA1}) with fixed $\tau$.}
			\label{Contour}
\end{figure}

\review{In Figure \ref{New_Fig9}a, we plot all solutions of system (\ref{sysA1}) in $\tau\psi-$plane and mark the stability using the criteria in Section \ref{sec_stability}. Note that since this representation suppresses $\omega^*$, the multiple in-phase or anti-phase solutions which occur for
particular values of $\tau$ in Figure \ref{fig_SI_tau_omega_psi} are superimposed.
As $\tau$ varies, we observe that a stable solution corresponding to $\psi^*=0,\pi$ always exists with the appearance of an unstable solution in disjoint intervals of $\tau$,  while all the out-of-phase solutions are unstable. More precisely, for the in-phase solution, as $\tau$ increases, we notice that an unstable solution
disappears at  $\tau\approx 2203$, exists between $\tau\approx2207.5$ and 
$\tau\approx2217$, and  reappears at $\tau\approx 2221$.  The same behaviour occurs 
for the anti-phase solution at different values of $\tau$. Near the appearance and
disappearance of these unstable solutions the unstable out-of-phase solutions appear
and disappear.
As we observe in Figure \ref{Contour}, there are multiple solutions $(\omega^*,\psi^*)$ of (\ref{sysA1}) when $\tau$ is fixed. 
To study the creation and destruction of solutions further,
we take particular values for $\tau$ and show  all solutions in the blue rectangles from Figure \ref{New_Fig9}a in the $\omega\psi-$plane, see Figures \ref{New_Fig9}b$-$\ref{New_Fig9}i.
We now see that there are pitchfork bifurcations where 
a stable in-phase or 
anti-phase solution becomes unstable as two unstable out-of-phase solutions merge 
together, 
see Figures \ref{New_Fig9}b$-$\ref{New_Fig9}c  and \ref{New_Fig9}f$-$\ref{New_Fig9}g.
This correspond
to the values $\tau_2^*$ discussed above.
Moreover, there are saddle-node bifurcations  where two unstable in-phase or 
anti-phase solutions collide then vanish, see Figures \ref{New_Fig9}d$-$\ref{New_Fig9}e and \ref{New_Fig9}h$-$\ref{New_Fig9}i.
This corresponds to the value $\tau_1^*$ discussed above. 
For other parameter values, we observe the opposite sequence of bifurcations: two unstable in-phase or anti-phase solutions are created by a saddle-node bifurcation after which one gets stabilized by a pitchfork bifurcation involving two unstable out-of-phase solutions.
All the bifurcations are as predicted for the general model in Section \ref{bif:sec}.
We did not observe any saddle-node bifurcations of out-of-phase solutions for this
parameter set.}

\review{To help understand these bifurcations, we plot solutions in the $\tau\omega-$plane and the solutions near $\psi=\pi$ in the $\tau\omega\psi-$space in Figures \ref{fig1_B}$-$\ref{fig1_C}, respectively.
Considering the case $\psi^*=\pi$, we observe that:}
\begin{itemize}
    \item the pitchfork bifurcation occurs when two unstable out-of-phase solutions  merge together with one stable anti-phase solution \textcolor{green}{$\CIRCLE$} to produce one unstable anti-phase solution \textcolor{red}{$\blacksquare$},
    
    \item the saddle-node bifurcation occurs when the created unstable anti-phase solution \textcolor{red}{$\blacksquare$} in the above collides  with another unstable anti-phase \textcolor{red}{$\blacksquare$} and both vanish.
\end{itemize}

\begin{figure}[hbt!]
		\centering
			\includegraphics[width=1\textwidth]{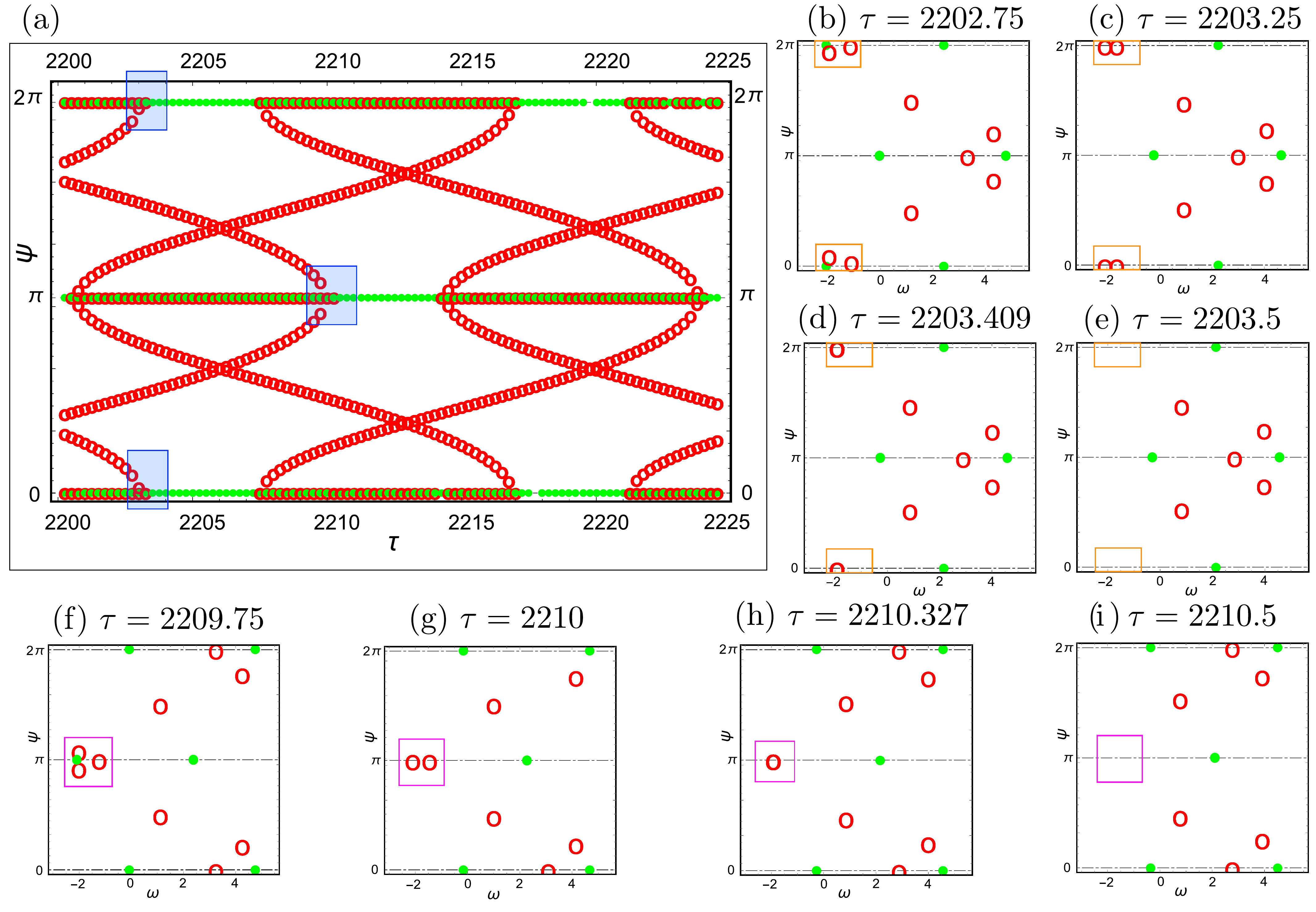}
\caption{The solutions of phase model (\ref{phase_Mod}) corresponding to (\ref{MLmodelNor}) 
	in the blue rectangles in Figure \ref{fig1} in $\omega\psi-$plane.
	The circles   \textcolor{green}{$\CIRCLE$}/\textcolor{red}{$\Circle$} represent stable/unstable solutions  of (\ref{sysA1}).} 	
	\label{New_Fig9}
	 \end{figure}

\begin{figure}[hbt!]
	\centering
	\begin{subfigure}{.54\textwidth}
		\centering
		\includegraphics[width=1\textwidth]{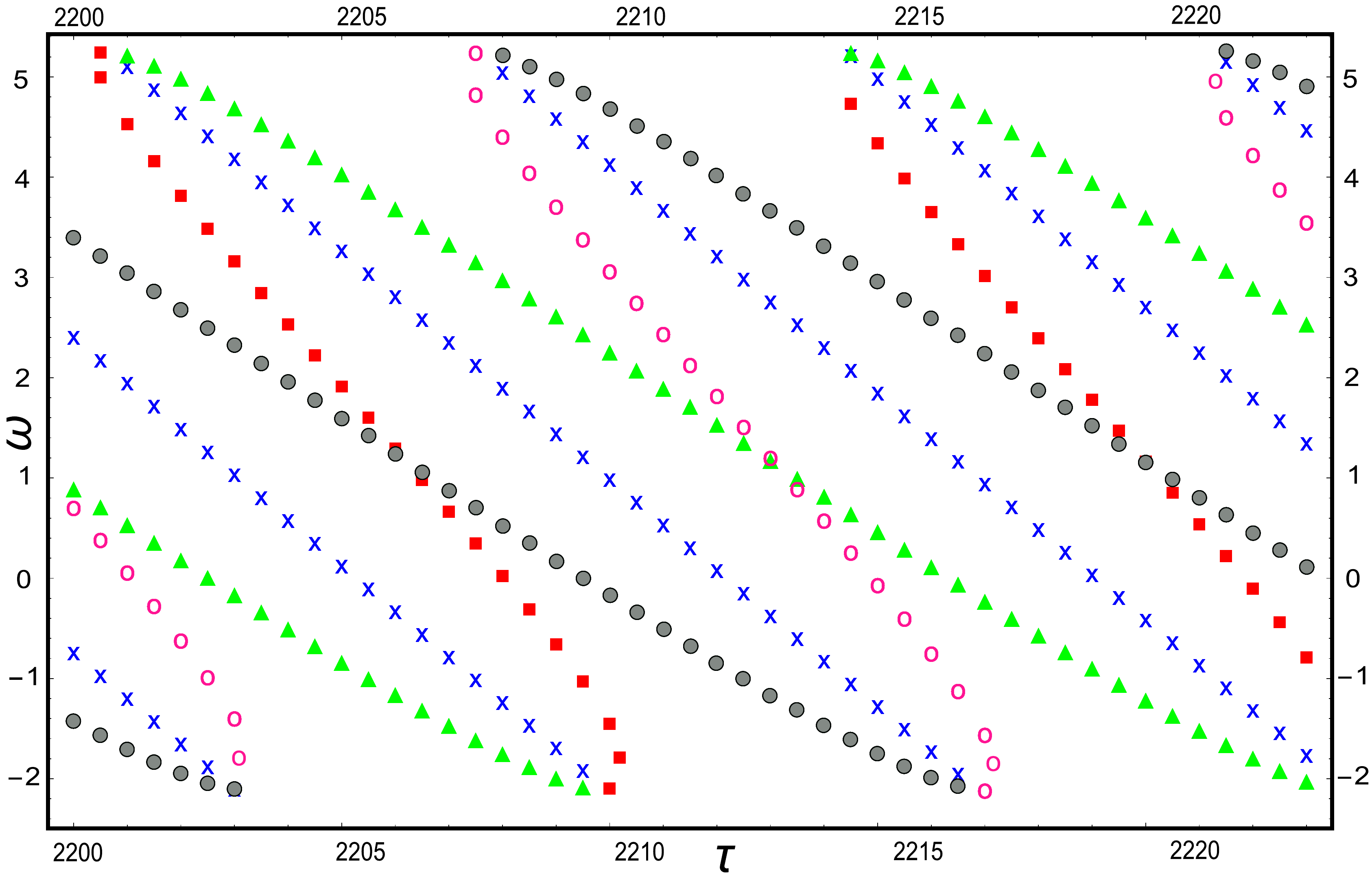}
		\caption{$\omega$ vs $\tau$}
		\label{fig1_B}
	\end{subfigure}
	\begin{subfigure}{0.37\textwidth}
		\centering
		\includegraphics[width=0.9\textwidth]{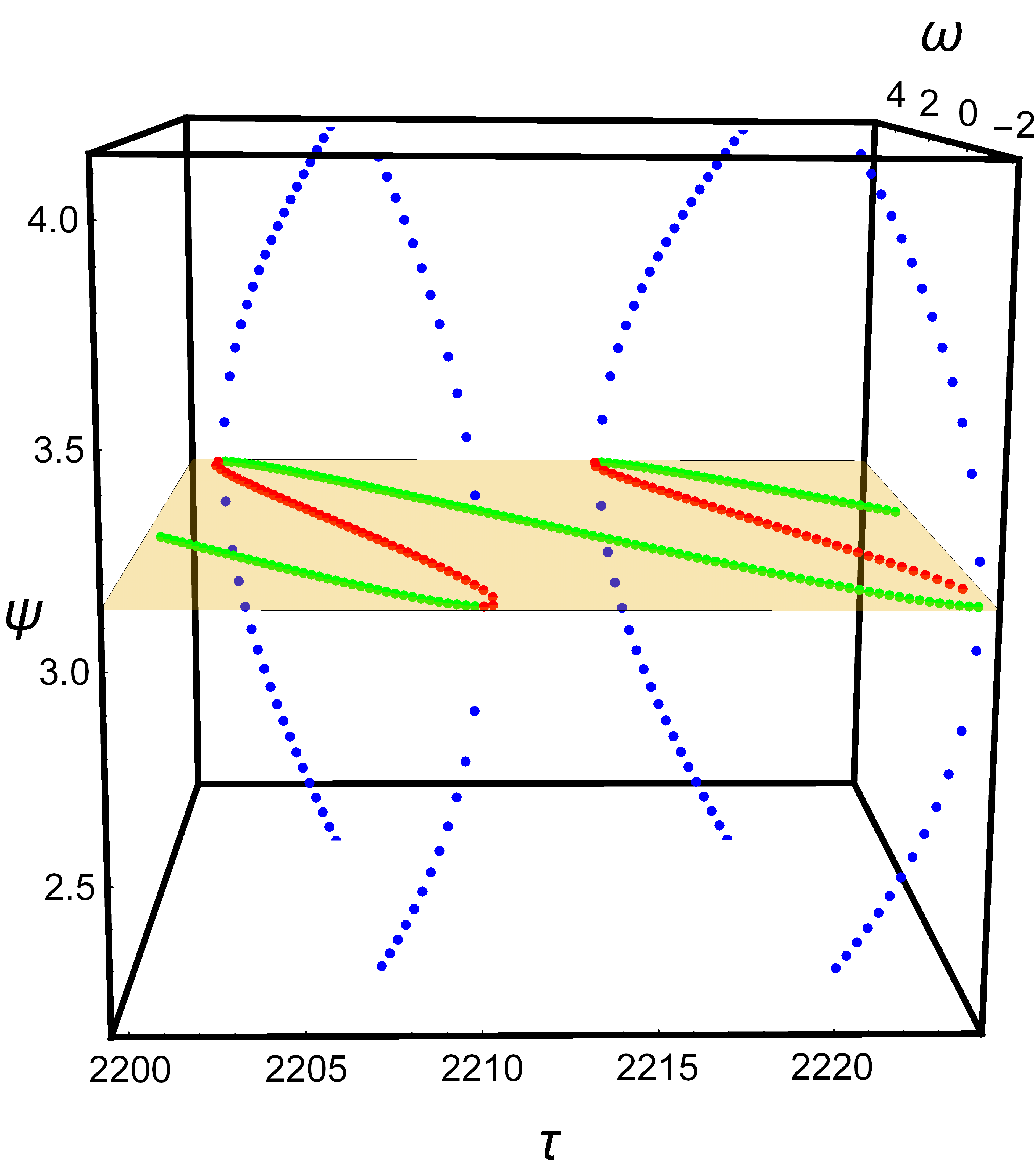}
		\caption{$\tau$, $\omega$ and $\psi$}
		\label{fig1_C}
	\end{subfigure}
	\caption{Numerical bifurcation diagram with respect to $\tau\in(2200,2225)$ for
the solutions of the phase model (\ref{phase_Mod}) corresponding to the Morris-Lecar model (\ref{MLmodelNor}) with parameter set II. The circles   \textcolor{gray}{$\CIRCLE$}/\textcolor{pink}{$\textrm{\ding{98}}$} represent stable/unstable in-phase solutions,  \textcolor{green}{$\blacktriangle$}/\textcolor{red}{$\blacksquare$} represents stable/unstable anti-phase solutions, and  \textcolor{blue}{$\times$} represents unstable out-of-phase solutions of (\ref{sysA1}). } 
	\label{fig1}
\end{figure}

\subsection{Small delay}

\label{Sec3_3}

In this subsection, we consider small time delay, in the sense that, $\Omega \tau=\mathcal{O}(1)$ with respect to the small parameter $\epsilon$, and compare the results with \cite{campbell2012phase}  where the authors studied this case using the parameter set II. 
In \cite{campbell2012phase}, the authors studied the dynamics of  the phase model corresponding to the full model (\ref{Full_Mod}) without introducing the frequency deviation in their analysis because 
the time delay $\eta$  was neglected in the phase model when $\Omega \tau=\mathcal{O}(1)$. We have stated some results from \cite{campbell2012phase} in Section \ref{Sec2_small_Delay}.

As in the previous section we solve (\ref{omega0}) and (\ref{omegapi}) to find
$\omega^*$ for the in-phase and anti-phase solutions and \eqref{sysA1} to find
$(\psi^*,\omega^*)$  for the out-of-phase solutions. We choose $\tau\in (0,15)$,
which is similar to the range chosen by \cite{campbell2012phase}.
In contrast with the results of the last section, here we observe that for 
$\psi^*=0,\pi$ there is a {\em unique} solution $\omega^*$ 
for each $\tau$ in the range we considered. 
This agrees with the prediction of the phase model in Section \ref{Sec2_small_Delay}.
We describe our results in more detail below.
 
\begin{figure}[hbt!]
	\centering
		\begin{subfigure}{0.33\textwidth}
		\centering
		\includegraphics[width=0.93\textwidth]{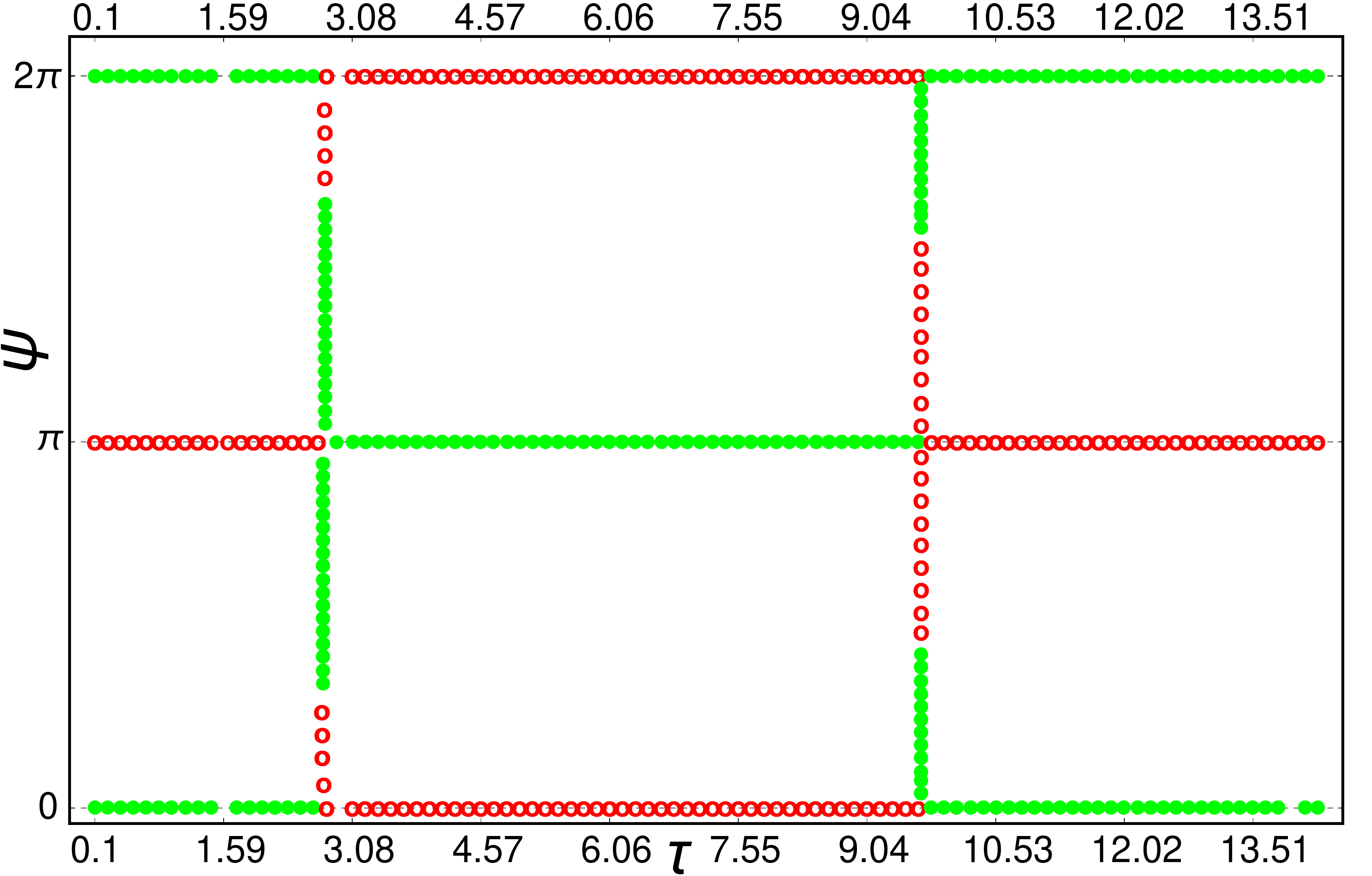}
		\caption{$\tau\in(0,15)$}
		\label{fig2_A}
	\end{subfigure}\hspace{-0.5cm}
\begin{subfigure}{.33\textwidth}
		\centering
		\includegraphics[width=0.95\textwidth]{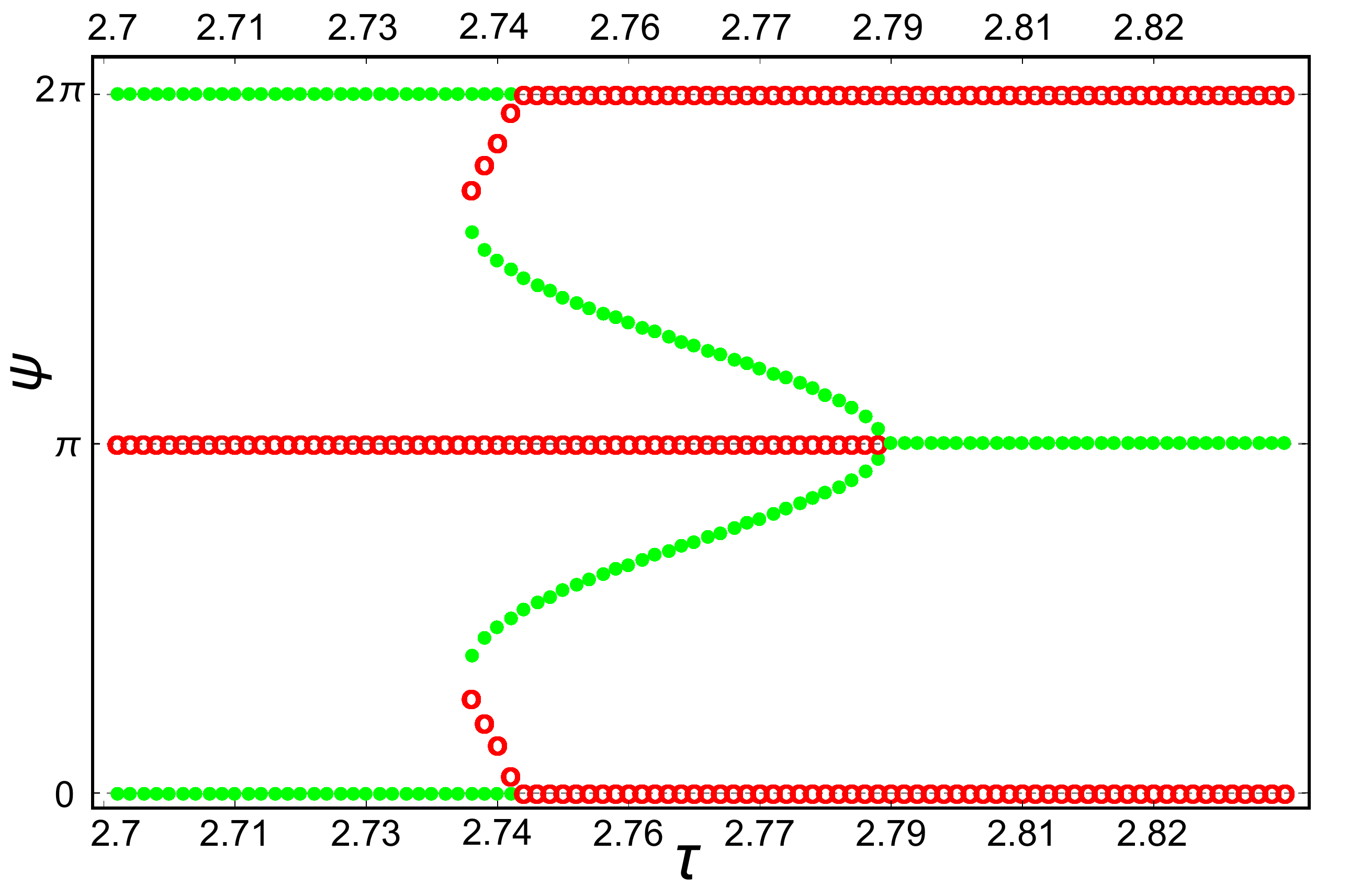}
		\caption{$\tau\in(2.7,2.84)$}
		\label{fig2_B}
	\end{subfigure}\hspace{-0.5cm}
	\begin{subfigure}{.33\textwidth}
		\centering
		\includegraphics[width=0.9\textwidth]{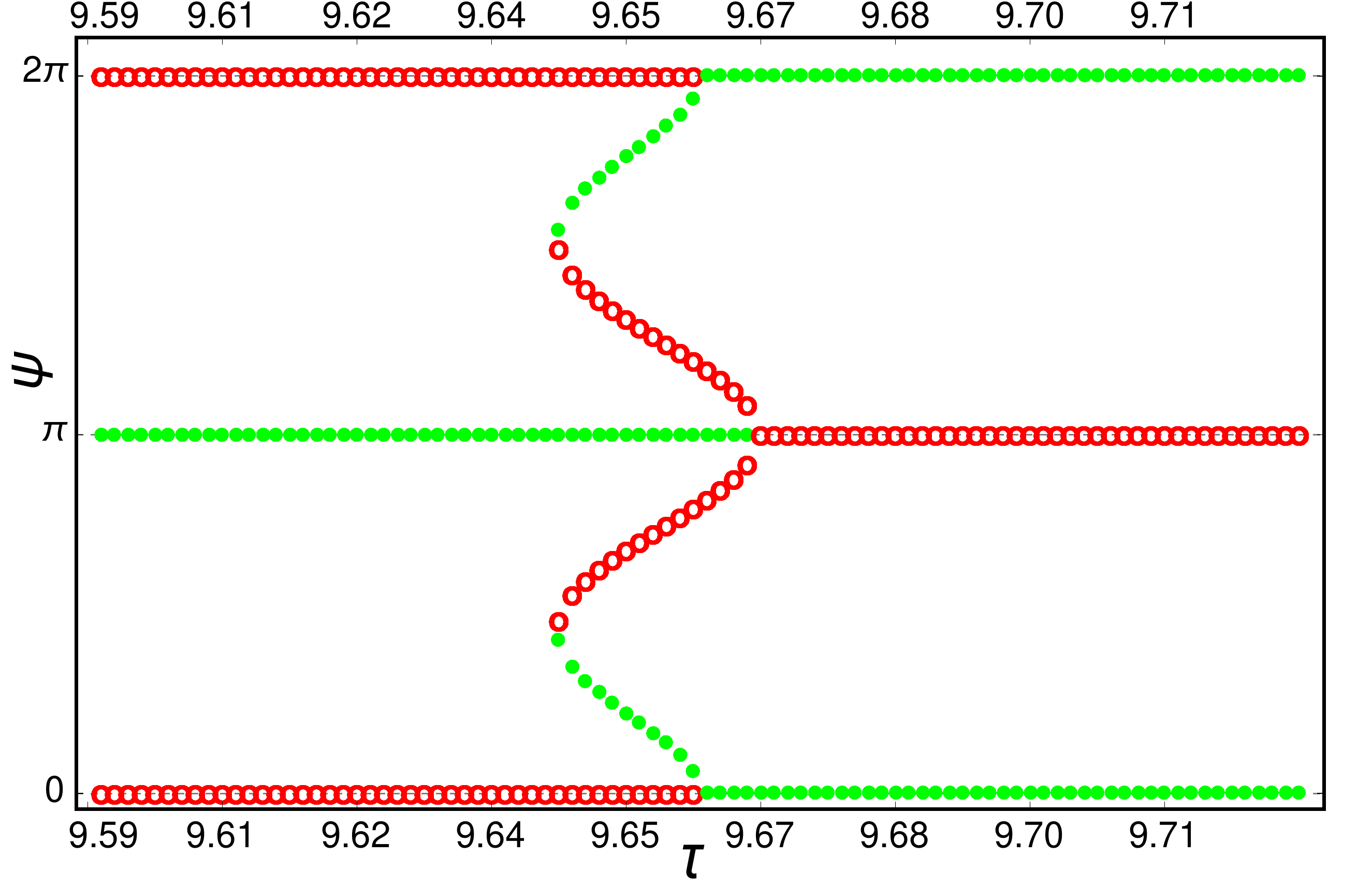}
		\caption{$\tau\in(9.59,9.73)$}
		\label{fig2_C}
	\end{subfigure}	
		\caption{ Numerical bifurcation diagram with respect to $\tau\in(0,15)$ for
the phase model (\ref{phase_Mod}) corresponding to the Morris-Lecar model (\ref{MLmodelNor}) with parameter set II. The circles   \textcolor{green}{$\CIRCLE$}/\textcolor{red}{$\Circle$} represent stable/unstable solutions in the phase model (\ref{sysA1}).} 	\label{fig2}
\end{figure}

In Figure \ref{fig2}, we plot the in-phase and anti-phase solutions as $\tau$ varies in $(0,15)$ in the $\tau\psi-$plane. We note that there is similar behaviour in Figure \ref{fig2_A} and  \cite[Figure 4b]{campbell2012phase}. The in-phase and anti-phase solutions change stability as $\tau$ increases and their stabilities appear to
be the opposite of each other.  To examine the behaviour near changes of stability, 
in Figures \ref{fig2_B}$-$\ref{fig2_C}  we show the bifurcation diagrams zoomed close
to the two switching points.  We see that the transition from
stable in-phase solution to stable anti-phase solution involves
two pitchfork bifurcations and one saddle-node bifurcation of out-of-phase
solutions, which agrees with  \cite{campbell2012phase}.
Figure \ref{fig21} shows this behaviour when the solutions are plotted in the 
$\tau\omega-$plane.
\review{Furthermore, we observe  in Figures \ref{fig2_B}$-$\ref{fig2_C} that there are small intervals of $\tau$ where bistability occurs.
Figure \ref{Fig_v1v2_small_delay} shows the coexistence of stable anti-phase and out-of-phase solutions.}

 \begin{figure}[hbt!]
		\centering
			\includegraphics[width=0.7\textwidth]{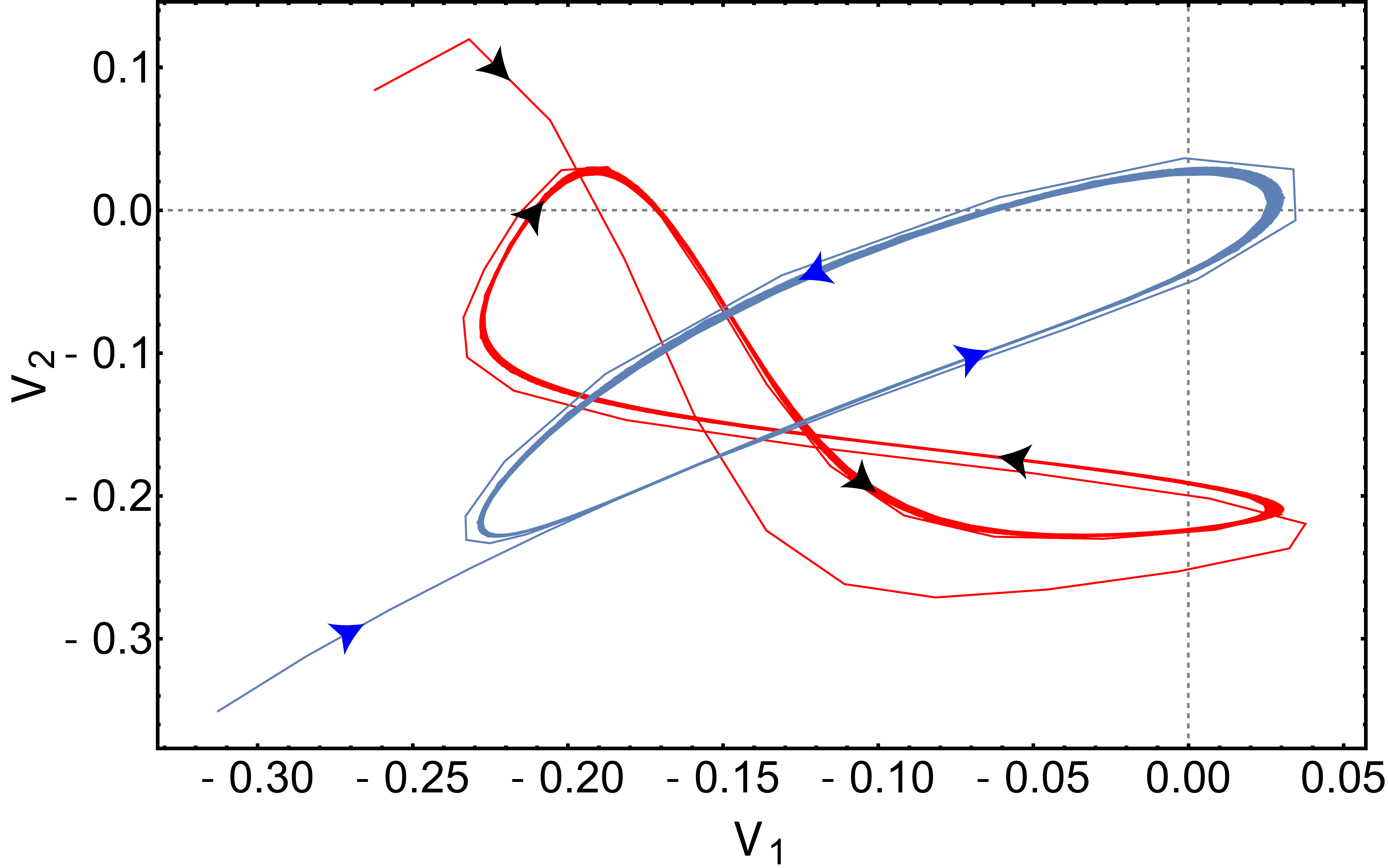}
				\caption{The coexistence of stable anti-phase (red) and out-of-phase  (blue) solutions of (\ref{MLmodelNor})  when $\tau=9.661$ with parameter sets II. We take different initial conditions: $(0.664192, 0.204054, 5.58914, 0.762568)^T$ for red curve and $(-0.883364, -0.200879, -0.686477, -0.989329)^T$ for blue curve.}
	 \label{Fig_v1v2_small_delay}
	 \end{figure}

\begin{remark}
The results in this section are consistent with the results in  
\cite{izhikevich1998phase}, which indicate that a phase model 
where the time delay enters as a phase shift is accurate
when $\tau$ is small in the full model (\ref{Full_Mod}) in 
the sense that \review{$\Omega\tau=\mathcal{O}(1)$   with respect to $\epsilon$ for $0<\epsilon\ll 1$.}
\end{remark}

\begin{figure}[hbt!]
		\centering
\begin{subfigure}{0.33\textwidth}
		\centering
		\includegraphics[width=.93\textwidth]{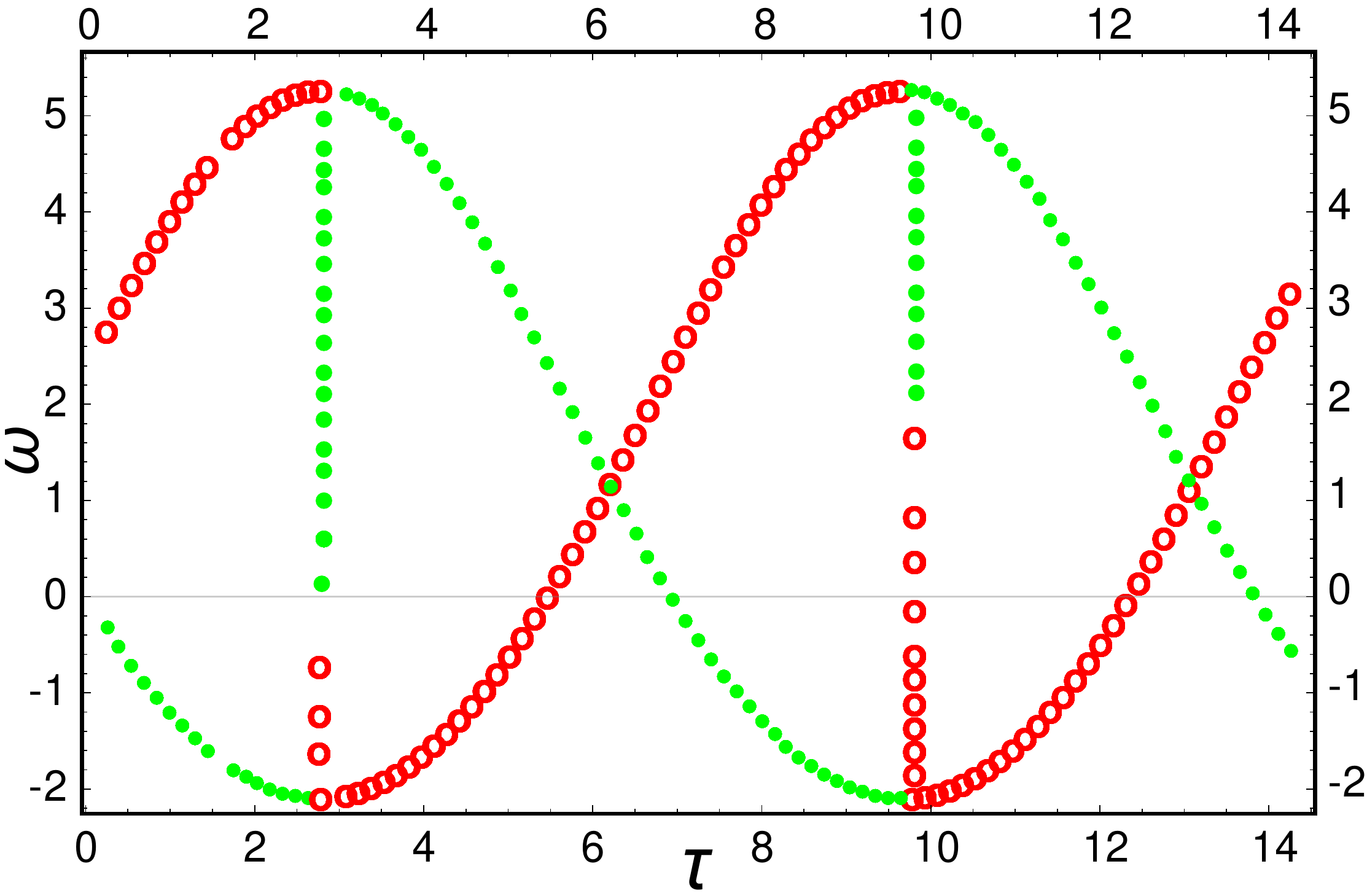}
		\caption{$\tau\in(0,15)$}
		\label{fig2_A0}
	\end{subfigure}%
	\begin{subfigure}{.33\textwidth}
		\centering
		\includegraphics[width=0.9\textwidth]{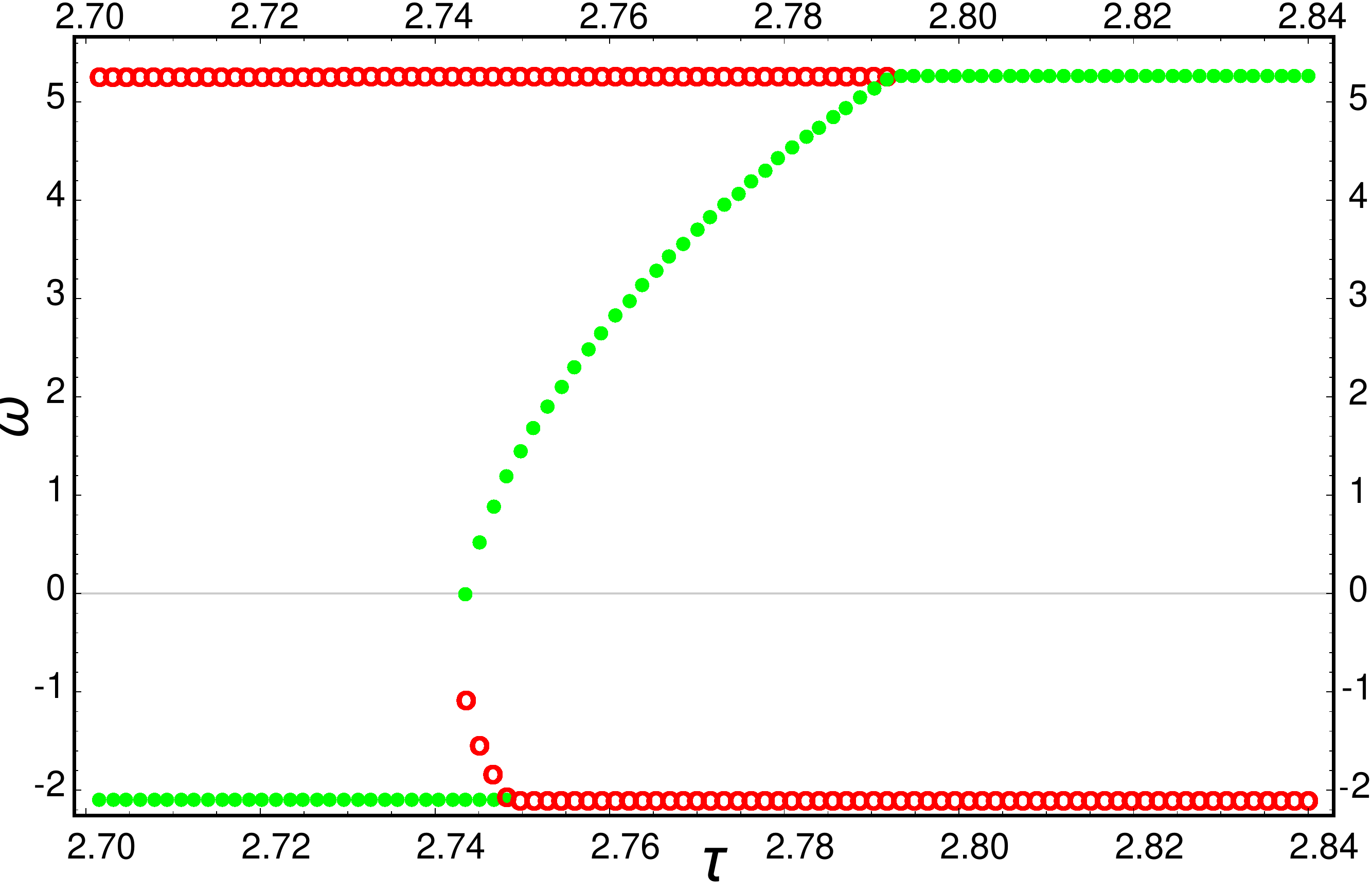}
		\caption{$\tau\in(2.7,2.84)$}
		\label{fig2_B0}
	\end{subfigure}
	\begin{subfigure}{.33\textwidth}
		\centering
		\includegraphics[width=0.9\textwidth]{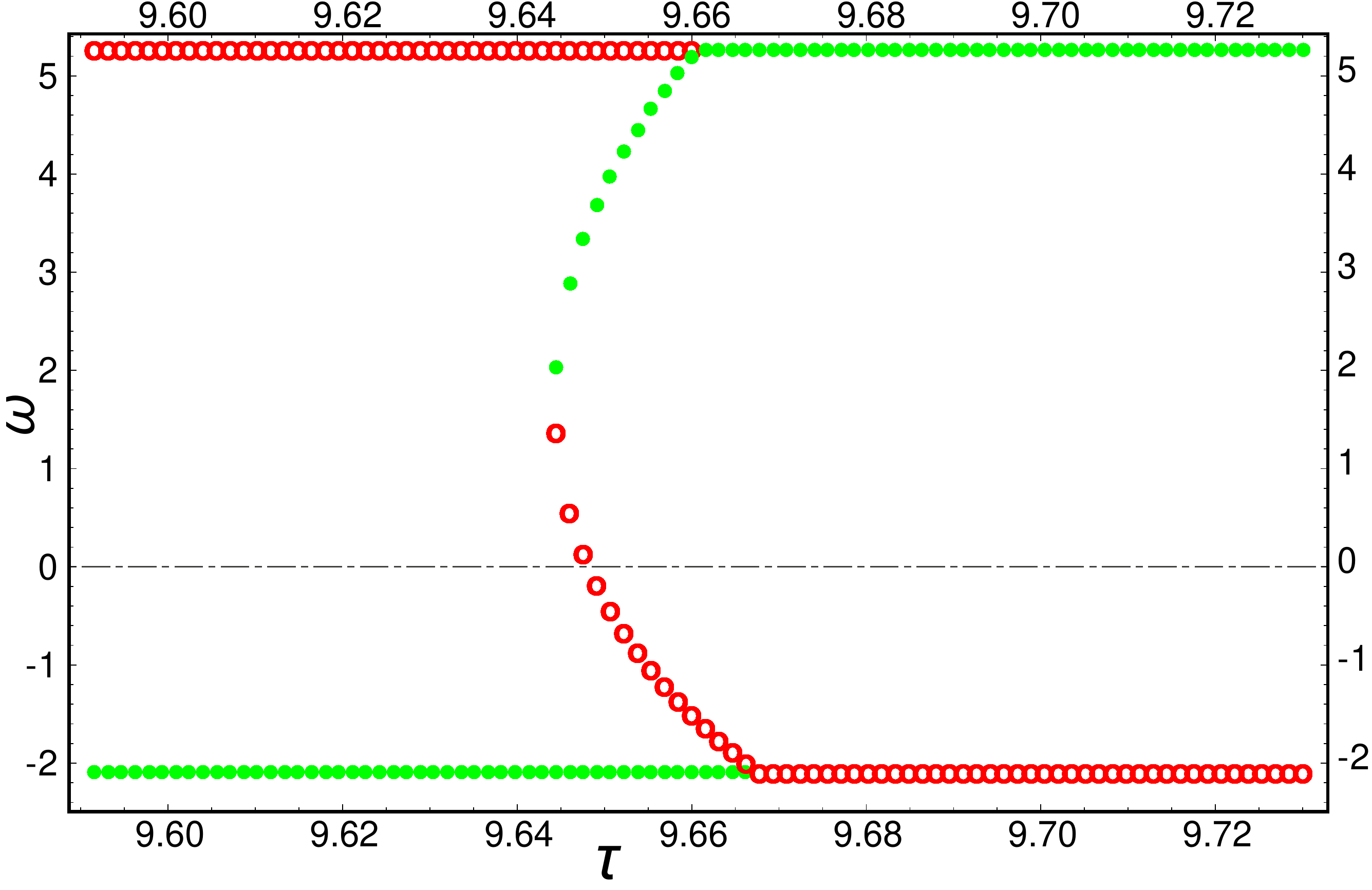}
		\caption{$\tau\in(9.59,9.73)$}
		\label{fig2_C0}
	\end{subfigure}
	\caption{Numerical bifurcation diagram with respect to $\tau$ for the 
phase model  \eqref{phase_Mod} corresponding to the Morris-Lecar model (\ref{MLmodelNor}) with parameter set II. The circles   \textcolor{green}{$\CIRCLE$}/\textcolor{red}{$\Circle$} represent stable/unstable solutions in the phase model (\ref{sysA1}).} 	\label{fig21}
\end{figure}

\section{Conclusions}
\label{sec_conc}

In this paper, we studied the phase-locking dynamics of a system of two weakly connected oscillators with time-delayed interaction. 
By applying the theory of weakly coupled oscillators, we transformed the system into a phase model with an explicit delay in the argument of the phases.  We showed that the system always
has phase-locked solutions corresponding to in-phase (synchronous, $0$ phase difference) 
and anti-phase (phase difference of half the period) solutions. Further, we showed for
small delay (\review{$\Omega\tau=\mathcal{O}(1)$}) the in-phase and anti-phase solutions
are unique, but for large delay multiple solutions of each type may exist, corresponding to different frequencies. Finally, we showed that phase-locked solutions with any
other phase differences (out-of-phase solutions) are also possible. 
Since the phase model is an infinite-dimensional system of delay differential equations, the linearized system about the phase-locked solutions has a countable infinity of eigenvalues. 
Through the stability analysis for our model, we discussed the distribution of the 
eigenvalues on the complex plane to 
provide stability conditions for the in-phase, anti-phase and out-of-phase solutions. 
We found that the zero eigenvalue always exists for any choice of parameters and functions which corresponds to the motion along the phase-locked solutions. We showed that the only
way in which bifurcations can occur is through the existence of (additional) zero
eigenvalues and argued
that the following bifurcations may occur: saddle-node bifurcations of two in-phase solutions
with different frequencies, saddle-node bifurcations of two anti-phase solutions 
with different frequencies, saddle-node bifurcations of two different out-of-phase solutions,
pitchfork bifurcations where two out-of-phase solutions arise from an in-phase
or anti-phase solution. We showed that the saddle-node bifurcations of in-phase and 
anti-phase solutions only involve unstable solutions. 

Our results on in-phase and anti-phase solutions agree with those in
\cite{Schuster1989Mutual,ermentrout2009delays}, which study the phase model 
(\ref{Kuramoto_model_Exp_Delay}), with $n=2$  and $H(\cdot)=\sin(\cdot)$. 
We note that they emphasized the need for large coupling-strength
for multiple in-phase/anti-phase solutions to exist, however, we show that it is possible
with weak coupling and sufficiently large delays. They do not study out-of-phase
solutions as these are not possible in their model due to the restriction on $H$.
As can be seen in the literature 
\cite{crook1997role,park2016weakly,wall2013synchronization,zhang2015robust}, in order for 
phase models derived from biophysical oscillator models to adequately capture the 
dynamics of the full model,  the function $H$ generally must include multiple Fourier modes.
\review{In \cite{campbell2012phase} it was shown that out-of-phase solutions and pitchfork
bifurcations cannot occur in a phase model with
small delay if only the first Fourier modes are included
in $H$. However, when the time delay is large, we showed that both out-of-phase solutions and pitchfork
bifurcations can occur in the phase model with   only the first Fourier modes of $H$.
In general, in the case of large time delay, the bifurcation structure may change if some modes are dropped.
If the coefficients of the modes that are dropped are small, then the bifurcation structure wouldn't change much. The bifurcation points may just move around. If the coefficients of the modes dropped are big enough then there could be large changes in the bifurcation structure.}

When the delay is small (\review{$\Omega\tau=\mathcal{O}(1)$}), Campbell and Kobelevskiy studied the system 
\beqn
\label{Campbell&Kobelevskiy}
\frac{d\theta_1}{d{t}}&=\Omega+\epsilon H(\theta_2(t)-\theta_1(t)-\Omega\tau),\\
\frac{d\theta_2}{d{t}}&=\Omega+\epsilon H(\theta_1(t)-\theta_2(t)-\Omega\tau),
\eeqn
and proved that in-phase and anti-phase solutions are stable  when $H'(\phi^*-\Omega \tau)>0$, $\phi^*\in\{0,\pi\}$ in \cite{campbell2012phase}. 
On the other hand, when the time delay is large \review{$\epsilon\Omega\tau=\mathcal{O}(1)$}, we proved that these solutions are stable whenever $H'(\phi^*-\omega^*\epsilon\Omega\tau-\Omega \tau)>0$ where $\omega^*$ is the corresponding frequency deviation. 
It is clear that the stability condition in the first case is independent of the coupling strength parameter and the frequency deviation. Indeed, under the assumption $\theta_1(t)=\Omega+\omega t$ and $\theta_2(t)=\Omega+\omega t+\phi^*$ (see (\ref{phases})), the terms of the frequency deviation $\omega$ will cancel out inside the function $H$ in (\ref{Campbell&Kobelevskiy}).
In fact, in \cite{campbell2012phase}, the authors reduce (\ref{Campbell&Kobelevskiy}) into a single ordinary differential equation and study the dynamics of the model without introducing the frequency deviation. 
Due to the explicit delay in the phase model, we couldn't reduce the model into a single equation.
For the out-of-phase solutions $\phi^*\notin \{0,\pi\}$, the stability condition $H'(\phi^*-\Omega \tau)>0$  is still valid when   the delay is small. 
While for the large delay the stability becomes more complicated since the explicit delay
is an additional parameter that needs to be considered in the phase model.

As an example we considered two Morris-Lecar oscillators with delayed, diffusive coupling. 
We adopted the parameter values from \cite{campbell2012phase} to compare the results when 
the time delay is small.  We studied the existence and stability of the phase-locked 
solutions, and explored the bifurcations in the phase model by using a four mode trunction 
of the Fourier series for the interaction function and compared these results with
numerical simulations of the full model.
When the time delay $\tau$ is large,  we found: 
\begin{itemize}
	\item There exist more than one frequency deviation $\omega$ corresponding to the in-phase and anti-phase solutions, i.e., co-existence of multiple stable and unstable solutions;
\item All out-of-phase solutions are unstable;
\item Both the pitchfork and saddle-node bifurcations of in-phase and anti-phase solutions
occur. 
\end{itemize}
When the time delay is small, we observed:
\begin{itemize}
	\item \review{Unique solution in each phase-locked solution category (in-phase, anti-phase and out-of-phase).}
	
		\item The occurrence of saddle-node bifurcations of out-of-phase solutions and
pitchfork bifurcations of in-phase and anti-phase solutions.
\end{itemize}
Our results agree  with \cite{campbell2012phase}  when the time delay is small  and  are consistent with the results in  \cite{izhikevich1998phase}, that the explicit time delay can be neglected in the phase model when $\tau$ is small.

A special type of phase-locked solutions, so-called \textit{symmetric cluster solutions}, can appear in a network of $n$ identical oscillators, see e.g., \cite{campbell2018phase,okuda1993variety}, 
\begin{equation}
\label{Campbell&Wang1}
\frac{d \mathbf{X}_{i}}{d t}={\mathbf{F}}\left(\mathbf{X}_{i}(t)\right)+\epsilon \sum_{j=1}^{n} a_{i j} {\mathbf{G}}\left(\mathbf{X}_{i}(t), \mathbf{X}_{j}\left(t-\tau\right)\right), \quad i=1, \ldots, n, \quad \mathbf{X}_i\in\mathbb{R}^m. 
\end{equation}
In these solutions, also called travelling wave solutions, oscillators in 
the same cluster are synchronized while
those in different clusters have non-zero phase-difference.
In \cite{campbell2018phase}, Campbell and Wang determined conditions for existence and stability of symmetric cluster solutions in  (\ref{Campbell&Wang1}) when $\tau$ is small and 
the coupling matrix is circulant. Stability conditions for cluster solutions in networks 
with small distance dependent delays and random, nearest neighbour coupling have
been formulated by several authors (see \cite{ermentrout2009delays,ko2004wave} and references therein).
When the time delay is large, Earl and Strogatz provided the stability condition 
for the in-phase solution  ($\theta_i(t)=\Omega t$, i.e., one cluster solution), see  \cite{earl2003synchronization}.
For future research,  it would be interesting to study the existence and stability of 
symmetric cluster solutions in (\ref{Campbell&Wang1}) with large time delay.

\section*{Acknowledgments}
The authors would like to thank the anonymous referees for their careful
reading and helpful suggestions.

\appendix
\numberwithin{equation}{section}
\makeatletter
\def\@seccntformat#1{\@ifundefined{#1@cntformat}%
   {\csname the#1\endcsname\quad}
   {\csname #1@cntformat\endcsname}
}
\newcommand{\section@cntformat}{Appendix:\ }
\makeatother
\section{\review{Phase reduction}}
Assume that the system  (\ref{newmodel}) admits an exponentially asymptotically stable periodic orbit with natural frequency $\Omega$ when  $\epsilon=0$. 
It follows from the time rescaling $\rho\to\Omega\rho$ that the natural frequency of the periodic orbit becomes $1$   and (\ref{newmodel}) can be written as 
\begin{equation}
\frac{{d{{\mathbf{X}}_i}}}{{d\rho}} = \frac{1}{\Omega} {\mathbf{F}}({{\mathbf{X}}_i}(\rho)) + \frac{\epsilon}{\Omega}\sum\limits_{j = 1}^n {{K_{ij}}} {\mathbf{G}}({{\mathbf{X}}_i}(\rho),{{\mathbf{X}}_j}(\rho - \Omega\tau )), \quad i = 1, \ldots n,\ {{\mathbf{X}}_i} \in {\mathbb{R}^m} 
\label{newmodel_A}
\end{equation}
Consequently,  there exists  a normally hyperbolic
invariant manifold $M=\gamma\times\cdots\times\gamma$ of system (\ref{newmodel_A}) when  $\epsilon=0$, where $\gamma$ is an exponentially orbitally stable $2\pi-$periodic solution of 
	\begin{equation}
	\label{App_0}
		\frac{d\mathbf{X}_i}{d\rho}=\frac{1}{\Omega}{\mathbf{F}}({\mathbf{X}_i}(\rho))\quad  i=1,\ldots,n.
	\end{equation}
Hence,  the solution of the $i^{th}$ equation of  (\ref{newmodel_A}) in an $\epsilon$ neighborhood of $M$ can be written as
\begin{equation}\label{App_1}
    {{\mathbf X}_i}(\rho) = \gamma \left( {\rho  + {\varphi _i}(t)} \right) + \epsilon {P_i}\left( {\rho ,\varphi_1(t),\ldots, \varphi_n(t),\epsilon} \right)
\end{equation}
where the term $\epsilon{P_i}$ is a smooth vector function which denotes the deviation from the manifold $M$ in the normal plane. 

Recall that $t=\epsilon\rho$ and let $\eta:=\epsilon\Omega\tau$, then the substitution of \eqref{App_1} in \eqref{newmodel_A} gives
 \beqn
\frac{{d{{\mathbf{X}}_i}}}{{d\rho }} &= \frac{1}{\Omega }{\mathbf{F}}\left[ {\gamma \left( {\rho  + {\varphi _i}(t)} \right) + \epsilon P_i\left( {\rho ,\varphi_1(t),\ldots, \varphi_n(t),\epsilon } \right)} \right]\\
&~ 
 + \frac{\epsilon }{\Omega }\sum_{j=1}^{n} K_{ij}{\mathbf{G}}\left[ {\gamma \left( {\rho  + {\varphi _i}(t)} \right) + \epsilon P_i\left( {\rho ,\varphi_1(t),\ldots, \varphi_n(t),\epsilon } \right),} \right.\\
&~\qquad \qquad \left. {\gamma \left( {\rho  - \Omega \tau  + {\varphi _j}(t - \eta )} \right) + \epsilon P_j\left( {\rho ,{\varphi _1}(t - \eta ),\ldots,{\varphi _n}(t - \eta ),\epsilon } \right)} \right]. 
\label{App_3}
\eeqn
Due to the infinite differentiability of ${\mathbf{F}}$ and ${\mathbf{G}}$, it follows from \eqref{App_3} that 
\beqn
\frac{{d{{\mathbf{X}}_i}}}{{d\rho }} = \frac{1}{\Omega }{\mathbf{F}}\left[ {\gamma \left( {\rho  + {\varphi _i}(t)} \right)} \right] + \frac{\epsilon }{\Omega }D{\mathbf{F}}\left[ {\gamma \left( {\rho  + {\varphi _i}(t)} \right)} \right]P_i\left( {\rho ,\varphi_1(t),\ldots, \varphi_n(t),\epsilon } \right)\\
 + \frac{\epsilon }{\Omega }\sum_{j=1}^{n} K_{ij}{\mathbf{G}}\left[ {\gamma \left( {\rho  + {\varphi _i}(t)} \right),\gamma \left( {\rho  - \Omega \tau  + {\varphi _j}(t - \eta )} \right)} \right] + \mathcal{O}\left( {{\epsilon ^2}} \right)
 \label{App_4}
\eeqn
where  $D{\mathbf{F}}$ is the Jacobian matrix of  ${\mathbf{F}}$.

Now, we differentiate  ${\mathbf X}_i$ in \eqref{App_1} with respect to $\rho$ to have 
\begin{equation}\label{App_2}
   \frac{{d{{\mathbf X}_i}}}{{d\rho }} = \gamma '\left( {\rho  + {\varphi _i}(t)} \right)\left( {1 + \epsilon \frac{{d{\varphi _i}}}{{dt }}} \right) + \epsilon \frac{{\partial {P_i}\left( {\rho ,\varphi_1(t),\ldots, \varphi_n(t),\epsilon } \right)}}{{\partial \rho }} + \mathcal{O}\left( {{\epsilon ^2}} \right).
\end{equation}
Note that
\begin{equation}\label{App_5}
 {\gamma ^\prime }\left( {\rho  + {\varphi _i}(t)} \right) = \frac{1}{\Omega }{\mathbf{F}}\left[ {\gamma \left( {\rho  + {\varphi _i}(t)} \right)} \right].
\end{equation}
Thus, from  \eqref{App_4} and \eqref{App_2}, we obtain
\begin{align}\label{App_6}
   {\mathbf{F}}\left[ {\gamma \left( {\rho  + {\varphi _i}(t)} \right)} \right]\frac{{d{\varphi _i}(t)}}{{dt}}  &+ \Omega \frac{{\partial y_i\left( {\rho ,\varphi_1(t),\ldots, \varphi_n(t)} \right)}}{{\partial \rho }}
= D{\mathbf{F}}\left[ {\gamma \left( {\rho  + {\varphi _i}(t)} \right)} \right]y\left( {\rho ,\varphi_1(t),\ldots, \varphi_n(t)} \right)\nonumber \\
&{ +\sum_{j=1}^{n} K_{ij} {\mathbf{G}}\left[ {\gamma \left( {\rho  + {\varphi _i}(t)} \right),\gamma \left( {\rho  - \Omega \tau  + {\varphi _j}(t - \eta )} \right)} \right] + \mathcal{O}(\epsilon )} 
\end{align}
 where $y_i\left( {\rho ,\varphi_1(t),\ldots, \varphi_n(t)} \right): = P_i\left( {\rho ,\varphi_1(t),\ldots, \varphi_n(t),0} \right) +  \mathcal{O}(\epsilon )$ because $P_i$ is smooth function of $\epsilon$. 
Consequently, since $t=\epsilon\rho$, we replace $\frac{\partial y_i}{\partial \rho}$ by $\frac{d y_i}{d \rho}$  in \eqref{App_6}. Hence,  we can write \eqref{App_6} as:
\begin{equation}\label{App_7}
 \frac{{dy_i}}{{d\rho }} = A\left( {\rho ,{\varphi _i}} \right)y + {b_i}\left( {\rho ,{\varphi _1},\ldots,{\varphi _n}} \right) + \mathcal{O}(\epsilon )
\end{equation}
where $\varphi_i$ is $\varphi_i(t)$,  
\[A\left( {\rho ,{\varphi _i}} \right) = \frac{1}{\Omega }{\mathbf{F}}\left[ {\gamma \left( {\rho  + {\varphi _i}(t)} \right)} \right]\]
and  
\[{b_i}\left( {\rho ,{\varphi _i},\ldots,{\varphi _n}} \right) = \frac{1}{\Omega }\left[ {\left( {\sum\limits_{j = 1}^n {{K_{ij}}} {\bf{G}}\left[ {\gamma \left( {\rho  + {\varphi _i}(t)} \right),\gamma \left( {\rho  - \Omega \tau  + {\varphi _j}(t - \eta )} \right)} \right]} \right) - {\bf{F}}\left[ {\gamma \left( {\rho  + {\varphi _i}(t)} \right)} \right]\frac{{d{\varphi _i}(t)}}{{dt}}} \right].\]
Since ${{\varphi _i}(t)}$ and ${{\varphi _i}(t - \eta)}$ in  $b$  do not depend directly on $\rho$, we have a linear non-homogeneous
system for $y_i$, where both the matrix $A$ and the vector $b_i$ are $2\pi$ periodic in $\rho$.
  
To study existence and uniqueness of solutions to \eqref{App_7},  
we consider the adjoint linear homogeneous system
\begin{equation}\label{App_8}
    \frac{{d{Q_i{\left( {\rho ,{\varphi _i}} \right)}}}}{{d\rho }} =  - A{\left( {\rho ,{\varphi _i}} \right)^T}{Q_i}{\left( {\rho ,{\varphi _i}} \right)}
\end{equation}
 with the normalization condition:
 \begin{equation}\label{App_9}
   \frac{1}{{2\pi }}\int\limits_0^{2\pi } {Q_i^T\left( {\rho ,{\varphi _i}} \right)} {\mathbf{F}}\left[ {\gamma \left( {\rho  + {\varphi _i}} \right)} \right]d\rho  = 1.
\end{equation}
Since the limit cycle $\gamma$ is exponentially orbitally stable, the 
homogeneous ($b_i\equiv 0$) linear system of the form \eqref{App_7}
the adjoint system \eqref{App_8} both have $1$ as
a simple Floquet multiplier, and all the other multipliers lie inside the unit circle.
Thus, system  \eqref{App_8}-\eqref{App_9} has a unique nontrivial periodic solution  $ {q_i\left( {\rho ,{\varphi _i}} \right)}$.

Now, by the Fredholm alternative, the linear non-homogeneous system \eqref{App_7}
has a unique periodic solution $Y_i$ if and only if the following orthogonality condition
holds:
\begin{equation}\label{App_10}
\left\langle {{q_i},{b_i}} \right\rangle  + \mathcal{O}(\epsilon ) = \frac{1}{{2\pi }}\int\limits_0^{2\pi } {q_i^T\left( {\rho ,{\varphi _i}} \right)} {b_i}\left( {\rho ,{\varphi _i},\ldots,{\varphi _n}} \right)d\rho  + \mathcal{O}(\epsilon ) = 0.
\end{equation}
Assume that $q_i(\rho,0)$ is found. Hence, $q_i(\rho,\varphi_i)=q_i(\rho+\varphi_i,0)$ because  $A\left( {\rho ,{\varphi _i}} \right) = \frac{1}{\Omega }D{\mathbf{F}}\left[ {\gamma \left( {\rho  + {\varphi _i}(t)} \right)} \right] = A\left( {\rho  + {\varphi _i}(t)} \right)$.
Thus, when we substitute $b_i$  in \eqref{App_10}, we obtain the following:
\beqnn
&\frac{1}{{2\pi \Omega }}\int\limits_0^{2\pi } {q_i^T\left( {\rho  + {\varphi _i},0} \right)} \left( {\sum\limits_{j = 1}^n {{K_{ij}}{\bf{G}}\left[ {\gamma \left( {\rho  + {\varphi _i}(t)} \right),\gamma \left( {\rho  - \Omega \tau  + {\varphi _j}(t - \eta )} \right)} \right]} } \right)d\rho  + {\cal O}(\epsilon )\\
&\qquad\hspace{6cm}= \frac{1}{{2\pi}}\int\limits_0^{2\pi} {q_i^T\left( {\rho  + {\varphi _i},0} \right)} D{\mathbf{F}}\left[ {\gamma \left( {\rho  + {\varphi _i}(t)} \right)} \right]\frac{{d{\varphi _i}(t)}}{{dt}}d\rho.
\eeqnn
Since $\frac{{d{\varphi _i}(t)}}{{dt}}$ is treated as a parameter and is independent of $\rho$, it follows from the normalization condition \eqref{App_9} that 
\[\frac{{d{\varphi _i}(t)}}{{dt}} = \frac{1}{{2\pi \Omega }}\int\limits_0^{2\pi } {q_i^T\left( {\rho  + {\varphi _i},0} \right)} \left( {\sum\limits_{j = 1}^n {{K_{ij}}{\bf{G}}\left[ {\gamma \left( {\rho  + {\varphi _i}(t)} \right),\gamma \left( {\rho  - \Omega \tau  + {\varphi _j}(t - \eta )} \right)} \right]} } \right)d\rho  + {\cal O}(\epsilon )\]
Letting $s=\rho+\varphi_i$ leads to
\[\frac{{d{\varphi _i}(t)}}{{dt}} = \frac{1}{{2\pi \Omega }}\sum\limits_{j = 1}^n {{K_{ij}}} \left( {\int\limits_0^{2\pi } {q_i^T\left( {s,0} \right)} {\bf{G}}\left[ {\gamma \left( s \right),\gamma \left( {s - \Omega \tau  + {\varphi _j}(t - \eta ) - {\varphi _i}(t)} \right)} \right]ds } \right) + {\cal O}(\epsilon )\]
Define
\[H\left( {{\varphi _j}(t - \eta ) - {\varphi _i}(t) - \Omega \tau } \right) = \frac{1}{{2\pi }}\int\limits_0^{2\pi } {q_i^T\left( {s,0} \right)} {\mathbf{G}}\left[ {\gamma \left( s \right),\gamma \left( {s - \Omega \tau  + {\varphi _j}(t - \eta ) - {\varphi _i}(t)} \right)} \right]ds.\]
Thus, we have system \eqref{pairwise} with 
\begin{equation}\label{pairwise_AA} 
    \frac{{d{\varphi _i}(t)}}{{dt}} = \frac{1}{{\Omega}}\sum\limits_{j = 1}^n {{K_{ij}}}H\left( {{\varphi _j}(t - \eta ) - {\varphi _i}(t) - \Omega \tau } \right)+ \mathcal{O}(\epsilon ).
\end{equation}

Recall that $\eta=\epsilon\Omega\tau$. Hence, when $\Omega \tau =\mathcal{O}(1)$ with respect to $\epsilon$, we have
\[\varphi_i(t-\eta)=\varphi_i(t-\epsilon\Omega\tau)=\varphi_i(t)+\mathcal{O}(\epsilon),\qquad i=1,\ldots,n.\]
Consequently, the Taylor series expansion for $h$ with respect to $\epsilon$ gives 
\beqnn
H\left( {{\varphi _j}(t - \epsilon \Omega \tau ) - {\varphi _i}(t) - \Omega \tau } \right) &=H\left( {{\varphi _j}(t) - {\varphi _i}(t) - \Omega \tau  + \mathcal{O}(\epsilon )} \right)\\
&=H\left( {{\varphi _j}(t) - {\varphi _i}(t) - \Omega \tau} \right)  + \mathcal{O}(\epsilon ),
\eeqnn
that is, no delay appears in the argument of the phases. Hence, (\ref{pairwise_AA}) 
becomes: 
\begin{equation}\label{smaleDealy_AA} 
    \frac{{d{\varphi _i}(t)}}{{dt}} = \frac{1}{{\Omega}}\sum\limits_{j = 1}^n {{K_{ij}}}H\left( {{\varphi _j}(t) - {\varphi _i}(t) - \Omega \tau } \right)+ \mathcal{O}(\epsilon ).
\end{equation}

\bibliographystyle{ieeetr}
\bibliography{references}

\end{document}